\documentclass[12pt]{amsart} 
\usepackage{amsmath, amsthm, amscd, amsfonts,amssymb,graphicx,enumerate,tikz-cd}
\numberwithin{equation}{section}
\makeatletter
\@namedef{subjclassname@2010}{%
  \textup{2010} Mathematics Subject Classification}
\makeatother

\newcommand{\teq}{\arabic{section}.\arabic{equation}}

\newcommand{\teql}{\Alph{section}.\arabic{equation}}

\newcommand{\sqr}[2]{{\vcenter{\vbox{\hrule height.#2pt\hbox{\vrule width.#2pt
height#1pt \kern#1pt\vrule width.#2pt}\hrule height.#2pt}}}}

\newcommand{\ssquare}{{\qquad\hfill$\square$}}

\newcounter{eqcount}
\newcounter{ttopic}

\newenvironment{edesc}{\refstepcounter{equation}\begin{enumerate}}%
{\end{enumerate}}

\newcommand{\sdisplay}[1]{{{$\scriptstyle\bullet$}\! #1 \!{$\scriptstyle\bullet$}}} 
\newcommand{\Sdisplay}[1]{{\vskip.2in\noindent $\star$ {\eu #1} $\star$\hskip.1in:}} 

\newcommand{\ring}[1]{{\mathbb #1}}
\newcommand\bZ{{\ring{Z}}}

\newcommand\bC{{\ring{C}}} \newcommand\bR{{\ring{R}}}
\newcommand\bF{{\ring{F}}} \newcommand\bQ{{\ring{Q}}}
\newcommand\bH{{\ring{H}}}
\newcommand{\csp}[1]{{\mathbb #1}}
\newcommand{\tsp}[1]{{\mathcal #1}}

\newcommand{\prP}{\csp{P}}
\newcommand{\afA}{\csp{A}}

\newcommand{\sO}{{\tsp{O}}} 
\newcommand{\sQ}{\tsp{Q}}
\newcommand{\sP}{{\tsp {P}}} 
 \newcommand{\sS}{{\tsp {S}}}
\newcommand{\sT}{{\tsp {T}}} \newcommand{\sH}{{\tsp {H}}}
\newcommand{\sX}{{\tsp {X}}} 
\newcommand{\sM}{{\tsp {M}}} 
\newcommand{\sG}{{\tsp {G}}}

\newcommand{\eql}[2]{{\rm (\ref{#1}\ref{#2})}} 

\newcommand{\vect}[1]{{\pmb #1}} 
 \newcommand{\bg}{\vect{g}}
 
\newcommand{\bp}{{\vect{p}}} \newcommand{\bx}{{\vect{x}}}
 \newcommand{\bw}{{\vect{w}}}
 \newcommand{\bz}{{\vect{z}}}
  
\newcommand{\row}[2]{{#1_1,\ldots,#1_{#2}}}

\newcommand{\smatrix}[4]{{\big(\begin{array}{cc}
\!\lower2pt\hbox{$\scriptstyle#1$} &\lower2pt\hbox{$\scriptstyle#2$}\!
\\\! \raise2pt\hbox{$\scriptstyle#3$} &\raise2pt\hbox{$\scriptstyle#4$}
\!\end{array}\big)}}
\newcommand{\col}[2]{{\big(\begin{array}{c}
\!\lower2pt\hbox{$\scriptstyle#1$}  \!
\\\! \raise2pt\hbox{$\scriptstyle#2$}
\!\end{array}\big)}}

\newcommand{\texto}[1]{{\textr{#1}}}
\newcommand{\GL}{\texto{GL}} \newcommand{\SL}{\texto{SL}}
 \newcommand{\ind}{\texto{ind}}
\newcommand{\PSL}{\texto{PSL}} 
 
\newcommand{\Hom}{\texto{Hom}} \renewcommand{\ni}{\texto{Ni}}
 \newcommand{\Pic}{\texto{Pic}}
\newcommand{\textr}[1]{{\text{\rm #1}}}
 \newcommand{\ord}{\textr{ord}}
\newcommand{\abs}{\textr{abs}}  
 
 \newcommand{\inn}{\textr{in}}
 
\newcommand{\pr}{\textr{pr}}

\newcommand{\RET}{{\text{\rm RET}}}
\newcommand{\BCL}{{\text{\rm BCL}}}
\newcommand{\IGP}{{\text{\rm IGP}}}

\newcommand{\rd}{\texto{rd}}

\newcommand{\tG}[1]{{}_{#1}\tilde G}

\newcommand{\GAP}{{\bf GAP}}

\newcommand{\textb}[1]{{\text{\bf #1}}}
\newcommand{\bfC}{{\textb{C}}}
\newcommand{\longmapright}[2]{\smash{\mathop{\hbox to
#2pt{\rightarrowfill}}\limits^{#1}}}
\newcommand{\Longmapright}[2]{\smash{\mathop{\hbox to
#2pt{\Rightarrowfill}}\limits^{#1}}}
\newcommand{\longmapleft}[2]{\smash{\mathop{\hbox to
#2pt{\leftarrowfill}}\limits^{#1}}}

\newcommand{\mapright}[1]{\smash{\mathop{\longrightarrow}\limits^{#1}}}

\newcommand{\np}{{+}}   \newcommand{\nm}{{-}}

\newcommand{\lrang}[1]{{\langle #1\rangle}}

\newcommand{\eqdef}{\stackrel{\text{\rm def}}{=}}

\newcommand{\pa}[2]{{\frac{\partial #1} {\partial #2}}}

\newfont{\sevenrm}{cmr7}
\newfont{\bsevenrm}{cmbx7}
\newfont{\mathseven}{cmsy7}
\newfont{\bigmath}{cmsy10 scaled 1200}
\newfont{\fiverm}{cmr5}
\newfont{\bfiverm}{cmbx5}
\newfont{\hel}{cmbx10 scaled 1400}
\newfont{\eu}{eufb10}
\newfont{\sseu}{eufm5}
\newfont{\seu}{eufm7}
\newfont{\Cal}{cmmib10}
\newfont{\sCal}{cmmib7}
\newfont{\zch}{eusb10}

\theoremstyle{plain}
\newtheorem{thm}{Theorem}[section] 
\newtheorem{lem}[thm]{Lemma}
\newtheorem{princ}[thm]{Principle}

\newtheorem{prop}[thm]{Proposition}

\theoremstyle{definition}
\newtheorem{defn}[thm]{Definition}
\newtheorem{exmp}[thm]{Example}
\newtheorem{guess}[thm]{Conjecture}
\newtheorem{quest}[thm]{Question}
\newtheorem{prob}[thm]{Problem}

\theoremstyle{remark}
\newtheorem{rem}[thm]{Remark}
\newtheorem*{sol}{Solution}

\newcommand{\xs}{\times^s\!}

\def\pic #1 by #2 (#3){\vbox to #2{\hrule width 
#1 height 0pt depth 0pt\vfill\special{picture #3}}}
\def\scaledpicture#1
by #2 (#3 scaled #4){{\dimen0=#1 \dimen1=#2\divide\dimen0 by
 1000 \multiply\dimen0 by #4\divide\dimen1 by 1000 \multiply\dimen1 
by #4\pic \dimen0 by \dimen1 (#3 scaled #4)}}

\newcommand{\comm}[1]{{}}
\renewcommand{\RET}{{\rm RET}}
\newcommand{\bG}{{\tsp {G}}}

\newcommand{\bP}{{\tsp {P}}}

\newcommand{\Spin}{{\text{\rm Spin}}}
\newcommand{\lm}{{\text{\rm n-lm}}}
\newcommand{\Ext}{{\text{\rm Ext}}}
\renewcommand{\phi}{\varphi}
\newcommand{\Fr}{\text{\rm Fr}}
\newcommand{\MT}{\text{\bf MT}}

\newcommand{\HIT}{\text{\bf HIT}}

\newcommand{\TL}{\text{TimeLine}}

\newcommand{\HM}{\text{\bf HM}}

\newcommand{\mx}{\text{\rm mx}}
\newcommand{\OIT}{\text{\bf OIT}}

\renewcommand{\BCL}{\text{\bf BCL}}

\newcommand{\Ad}{\text{\rm Ad}}
\newcommand{\bS}{\bar{\vect S}}

\newcommand{\RIGP}{\text{\bf RIGP}}

\newcommand{\rk}{\text{\rm rk}}

\newcommand{\lcm}{\text{lcm}}

\newcommand{\C}{{\text{\rm C}}}
\newcommand{\CM}{\text{CM}}

\newcommand{\one}{{\pmb 1}}

\newcommand{\ab}{{{}_{\text{\rm ab}}}}
\newcommand{\geng}{{{\text{\bf g}}}}
\newcommand{\sh}{{{\text{\bf sh}}}}

\newcommand{\Cu}{{{\text{\rm Cu}}}}

\newcommand{\sF}{{\tsp F}}
\newcommand{\sK}{{\tsp K}}

\newcommand{\Cyc}{{\text{\rm Cyc}}}

\newcommand{\fG}[1]{{\,{}_{#1}\tilde G}} 
\newcommand{\tfG}[2]{{\,{}_{#1}^{#2} G}}

\newenvironment{exmpl}{\begin{exmp}}{\hfill $\triangle$ \end{exmp}}

\begin{document}
\baselineskip=17pt
\hoffset.75in

\title[Open Image Theorem]{Moduli relations between \\  $\ell$-adic Representations and the \\ Regular Inverse Galois Problem}

\author[M.~D.~Fried]{Michael
D.~Fried}
\address{Emeritus, UC Irvine \\ 1106 W 171st Ave, Broomfield CO 80023}
\email{mfried@math.uci.edu}

\date{} 

\begin{abstract} There are two famous Abel Theorems. Most well-known, is his description of \lq\lq abelian (analytic) functions\rq\rq\ on a one dimensional compact complex torus. The other collects together those complex tori, with their prime degree isogenies, into one space. Riemann's generalization of the first features his famous $\Theta$ functions. His deepest work aimed at extending Abel's second theorem; he died before he fulfilled this.

That extension is often pictured on complex higher dimension torii. For Riemann, though, it was to spaces of Jacobians of compact Riemann surfaces, $W$, of genus $\geng$, toward studying the functions $\phi: W \to \prP^1_z$   on them.  Data for such pairs $(W,\phi)$ starts with a monodromy group $G$ and conjugacy classes $\bfC$ in $G$. Many applications come from putting all such covers attached to $(G,\bfC)$  in natural  -- {\sl Hurwitz\/} -- families. 

We  {\sl connect\/} two such applications: The {\sl Regular Inverse Galois Problem\/} (\RIGP) and {\sl Serre's Open Image Theorem\/} (\OIT).  We call the connecting device {\sl Modular Towers\/} (\MT s).  Backdrop for the \OIT\ and \RIGP\ uses Serre's books \cite{Se68} and \cite{Se92}. Serre's \OIT\ example is the case where \MT\ levels identify as modular curves. 

With an example that isn't modular curves, we explain conjectured \MT\ properties --  generalizing a Theorem of Hilbert's -- that would conclude an $\OIT$ for all \MT s.  Solutions of pieces on both ends of these connections are known in significant cases.   \end{abstract}

\maketitle 

\setcounter{tocdepth}{3} 
\tableofcontents

\vspace*{2mm}

\vspace*{2mm}

\section{Introduction} Denote the Riemann sphere uniformized by an inhomogeneous variable $z$, as in a first course in complex variables, by $\prP^1_z=\bC\dot{\cup}\{\infty\}$. Indeed, \S\ref{whyrationalfuncts} actually starts there, and explains geometric monodromy groups from their contrasting algebraic and analytic approaches; ours is analytic. 

\S \ref{explainingFrattini}  -- The \TL\ of the \MT\ program -- is the most unusual in the paper. It is a list of extended abstracts of papers, showing how the material of the rest of the sections can fit together into a coherent program. 
 Here is  the format for historical references, as applied to \sdisplay{\cite{Se68}},  indicating Serre's book. The reference year telegraphs that it is in \S\ref{pre95},  pre-1995 material. There a reader will find a $\star$-display, here $\star$ {\cite{Se68}} $\star$, elaborating  on how Serre's book relates to our topics.  
 
\subsection{Preliminaries on the \RIGP\ and the \OIT}  \label{preimRIGPOIT} 
\S\ref{RIGPOITMT}  ties together three acronyms, the subjects of this paper: \RIGP\ (Regular Inverse Galois Problem), \OIT\ (Open Image Theorem) and \MT\  (Modular Towers). The inverse Galois problem asks
whether any finite group $G$ is the (Galois) group of some Galois field extension $F/\bQ$. That is, $G$ is a quotient of the absolute Galois group, $G_\bQ$ of the rational numbers.  

The {\sl regular\/} (and much stronger) version replaces $F/\bQ$ with $F^*/\bQ(z)$ where $F^*\cap \bar \bQ=\bQ$, with $\bar \bQ$ the algebraic closure of $\bQ$.   Despite being stronger, the regular version points to explicit spaces for a given $G$ on which to look for $\bQ$ points giving \RIGP\ solutions. Prop.~\ref{RIGPback}, \cite[Main Thm]{FrV91}, uses  key definitions for how the {\sl monodromy version\/} of the \RIGP\ works: 
$$\begin{array}{c} \text{Inner and absolute Nielsen classes  \eqref{bcycs}; Hurwitz spaces } \\ \text{Princ.~\ref{princHS};   fine moduli \sdisplay{\cite{Sh64}}. An exposition expanding} \\ \text{ Serre's book on the \RIGP\ already appeared in \cite{Fr94}.}\end{array}$$ 

\begin{prop} \label{RIGPback} For each Nielsen class $\ni(G,\bfC)$    with inner classes having fine moduli, there is a variety $\sH(G,\bfC)^\inn$ whose $K$ points -- $\sH(G,\bfC)^\inn(K)$, $K$ a number field -- correspond to \RIGP\ realizations over $K$ in the Nielsen class \sdisplay{\cite{Se92} and \cite{Fr94}}.  \end{prop}

It has often succeeded -- say, by using Thm.~\ref{genuscomp} as in the \S\ref{mainex} example -- in giving solutions. The connection is by applying Hilbert's Irreducibility Theorem. Indeed,  one view of our \OIT\ generalization is that it aims at a very strong version of Hilbert's Theorem (see \S\ref{outlineMTs}). 

Serre's version of the \OIT\ started in \cite{Se68}. Each formulation -- including Thm.~\ref{OITsf} which considers a collection of modular curve towers, each corresponding to a prime $\ell$ -- compares the arithmetic monodromy group of a $j$-line point in a tower with the geometric monodromy group of the tower. Serre's case is the model.  \sdisplay{\cite{Fr78}} uses it in an exposition guiding what we can expect at this time from our conjectures.  

This paper is a guide to an in-progress book \cite{Fr20}.\footnote{It is also an expansion of a previous conference proceedings article.} \S\ref{abstmts} gives an abstract of \cite{Fr20}, and then how subsections introduce  \MT s as formed from  Hurwitz spaces  and Frattini covers. \S\ref{l-adicreps} has a dictionary on {\sl Nielsen classes\/} -- descriptions of sphere covers -- with their braid action, for quickly defining \MT s and where the $\ell$-adic representations come into this. Our running example, starting in \S\ref{mainex} is explicit on these.  By this paper's end we see how the principle conjectures apply to it. 

The several conjectures about \MT s have a division between them: Some go with the \RIGP, others with the \OIT. 
$$\begin{array}{c} \text{We label the main conjectures accordingly:} \\ \text{Main \RIGP\ Conj.~\ref{mainconj} and Main \OIT\ Conj.~\ref{OITgen}.}\end{array}$$ 

\subsection{Key words and modular curve comparison} \label{RIGPOITMT} 
Among finite groups $G$, excluding $G$ nilpotent, there are a finite number of  natural projective sequences of groups covering  $G$ for which a positive answer to a simple \RIGP\ question produces \MT s with particular diophantine properties.  They carry $\ell$-adic representations allowing a conjectural version of a diophantine statement generalizing Serre's original \OIT.  

\subsubsection{Hurwitz spaces}  \label{introhursp} 
\S\ref{WhatGauss}  starts with 1-variable rational functions (as in junior high). With these we introduce {\sl Nielsen classes\/} attached to $(G,\bfC)$: with $\bfC=\{\row \C r\}$ a collection of $r=r_\bfC$ conjugacy classes in $G$. Denote the lcm of the orders of elements in $\bfC$ by $N_{\bfC}$. 

For nontriviality, we assume the collection of all elements in $\bfC$ generate $G$. A natural {\sl braid group action\/} on a Nielsen class gives us  {\sl Hurwitz\/} spaces:  moduli spaces of covers of $\prP^1_z$.

Consider $U_r$, subsets of $r$ distinct unordered elements $\{\bz\}\subset \prP^1_z$. It is also  
\begin{equation}\label{discriminant}  \begin{array}{c} \text{projective $r$-space $\prP^r$ minus its {\sl discriminant locus\/} ($D_r$).}\\ \text{If the Nielsen class is non-empty, then each Hurwitz  } \\ \text{ space component is an unramified cover of $U_r$.}\end{array}\footnote{$\prP^r$ is the symmetric product on $\prP^1_z$. The discriminant locus is all of the $\row z r$ with $ z_i = z_j$ for some $i \not = j$.} \end{equation}  
 
Significantly, as in our examples, Hurwitz spaces can have more than one component. That distinguishes the tower of Hurwitz spaces from a \MT\ lying on such a tower. Usually, however, components are either conjugate (under $G_\bQ$ action)   or they have moduli properties that distinguish them. Separating these two situations is a major issue with the \OIT. 

\subsubsection{Towers of reduced Hurwitz spaces} 
In our examples we easily see when a collection of conjugacy classes is generating. It can be harder to decide for given $r=r_\bfC$ whether a Nielsen class is nonempty.  In particular, for any prime $\ell$ dividing $|G|$, where $G$ has no $\bZ/\ell$ quotient, forming \MT s starts with considering choices of $r$ {\sl generating\/}  conjugacy classes, whose elements have $\ell'$ (prime to $\ell$)  order.  

That gives a {\sl canonical\/} tower of Hurwitz spaces over a base Hurwitz space (\S\ref{modtowdef}). Points on each  level represent equivalence classes of covers of $\prP^1_z$, with monodromy group a cover of $G$   and data given by $\bfC$. 

Powers, $\ell^{k\np1}$, of $\ell$ reference general tower levels, $\sH_k$, $k\ge 0$, all sharing similar {\sl moduli\/} properties. 
Our emphasis is on relating the \OIT\ and the \RIGP. So,  we usually -- unless otherwise said -- concentrate on these $\prP^1_z$ cover equivalences  \S\ref{equivalences}: 
\begin{equation} \label{basicequiv} \begin{array}{c} \text{$\sH(G,\bfC)$ denotes {\sl inner\/} classes; for comparing with modular}\\ 
\text{curves $\sH(G,\bfC)^{\rd}$ denotes inner {\sl reduced\/}  classes.} \end{array} \end{equation} 
We often speak of both types without extra decoration, except when it is crucial to include reduced equivalence. That comes   from extending the $\SL_2(\bC)$ action in \eqref{SL2C} to include the Hurwitz space itself. 

\subsubsection{Computations and the case $r=4$}  \label{compsr=4} 
Reduced equivalence cuts down the (complex) dimension of the Hurwitz spaces by 3. For {\sl all\/} 1-dimensional, reduced examples, the {\sl normalization\/} of $\prP^1_j$ in the function field of a reduced Hurwitz space component $\sH^\rd$, gives a unique projective (nonsingular) curve cover $\bar \sH^\rd\to \prP^1_j$.\footnote{The same normalization process works for $r>4$, but the reduced Hurwitz spaces may be singular, even though normal, and the target space is $J_r$ (\S\ref{SL2C}).} Cusps on this space are, on one hand, the points over $\infty$.  Significantly for computations:
\begin{equation} \label{cuspsr=4} \begin{array}{c} \text{we associate to a cusp an orbit in a {\sl reduced Nielsen class\/}  of a} \\ \text{ cusp group, $\Cu_4$,  \eqref{cusps-comps}. Cusp widths are  ramification indices}\\ \text{ over $j=\infty$. They are also lengths of the cusp orbits.}\end{array} \end{equation}  

Each component of $\sH_k$ is an upper half-plane quotient  ramified over $\prP^1_j\setminus \{\infty\}$ at $0,1$.
Thm.~\ref{genuscomp} efficiently  computes genuses, cusps  and their widths, of each compactified component.\footnote{In our examples beyond dihedral groups, congruence subgroups don't give these moduli space upper half-plane quotients.} 

\MT s generalize modular curve towers. As is well-known, modular curves are moduli spaces for elliptic curves and their torsion, defined by {\sl congruence subgroups\/} of $\SL_2(\bZ)$.  Less well known, they are the image -- under a map that is an isomorphism on the underlying spaces, losing none of the abelian variety data -- of moduli for {\sl certain\/} covers of $\prP^1_z$.

Our illustrating example is Prop.~\ref{A43-2}. Here $G$ is the alternating group, $A_4$, of degree 4. We use this for comparing with modular curves and illustrating the formulation of \MT s, and braid and $\ell$-adic actions.

For a given union of Hurwitz space components, $\sH'\le \sH(G,\bfC)$, we can often compute the following. 
\begin{defn} \label{moddeffield} {\sl Moduli definition field\/}: a minimal field extension $\bQ_{\sH'}$, of $\bQ$ contained in the field of definition of any object   representing $\bp\in \sH'$.\footnote{This is a field of definition of the whole structure space representing the moduli problem, including the collection of families of representing objects.}\end{defn} 

$$\begin{array}{c}\text{When, $\sH'=\sH(G,\bfC)$, $\bQ_{\sH'}=\bQ_{G,\bfC}$ is an explicit cyclotomic field}\\\text{ by the {\sl branch cycle lemma\/} (\BCL\ Thm.~\ref{bcl}).}\end{array}$$ 

When {\sl fine moduli\/} holds, the residue class field $\bQ_{\sH'}(\bp)$ of $\bp$ is a field of definition of a representing object. Yes, unlike modular curves there can be -- significantly -- more than one component as in \S \ref{mainex}.  
There are cases with fine moduli for one component, but not the other; fine moduli for an inner space may be lost by using  reduced equivalence.  

\subsection{Beyond modular curves}  \label{abstmts} \S\ref{Fr20abstract} uses Serre's two books as a backdrop for the basic goal, relating the \RIGP\  and the \OIT\ by outlining \cite{Fr20}. \S\ref{outlineMTs} shows how we introduce properties of \MT s. 

Before that,  \S\ref{dihedtohyper} shows how finding {\sl involution realizations of dihedral groups\/} maps to finding cyclotomic points on hyperelliptic jacobians. The moduli problems slightly differ. Yet, this gives an isomorphism of the underlying spaces for significant \MT s to hyperelliptic jacobians. 

\subsubsection{Involution realizations of dihedral groups}  \label{dihedtohyper} 
Many people study modular curves, and their higher dimensional variant, spaces of hyperelliptic jacobians. Hurwitz spaces are less known and more general. 

As a prelude to Frattini ideas  generalizing these classical problems, we connect  this section's  title to cyclotomic points on hyperelliptic jacobians. 

Ex.~\ref{Dlk+1}, overlapping with Ex.~\ref{ncmodcurves}, even before the formal definition of Hurwitz spaces,  \S\ref{equivalences} introduces the (connected) levels of \MT s that will help understand them in general. 

\begin{exmpl} \label{Dlk+1} Consider $G_{k\np1}=D_{\ell^{k\np1}}=\bZ/\ell^{k\np1}\xs \bZ/2$ the order $2\ell^{k\np1}$ dihedral group, with $\ell$ an odd prime.\footnote{All primes eventually count; but here  we stay with $\ell$ odd.}; $\C$ is the involution conjugacy class of $G_{k\np1}$.  Consider the following two collections (formal definitions of the words appear in \S\ref{analgeom}).  

\begin{edesc} \label{Dl} \item \label{Dla}  Connected covers $\phi: W\to \prP^1_z$ of degree $\ell^{k\np1}$ ramified (only) at $r\ge 4$ (branch) points of $\prP^1_z$, with {\sl branch cycles\/}  elements of $\C$. 
\item  \label{Dlb}  Replace  $\phi$ in \eql{Dl}{Dla} by $\hat \phi: \hat W\to \prP^1_z$, Galois with group $D_{\ell^{k\np1}}$ with its induced cover $\hat W/\lrang{h}\to \prP^1_z$, $h\in \C$. \end{edesc}  With $\bfC=\bfC_{2^r}$ ($r$ repetitions of $\C$; $r$ must be even), the label for these collections as complex analytic spaces is $\sH(D_{\ell^{k\np1}},\bfC_{2^r})^\abs$ (resp.~$\sH(D_{\ell^{k\np1}},\bfC_{2^r})^\inn$) using the respective \eqref{eqname} equivalences  (and Ex.~\ref{ncmodcurves}).  \end{exmpl}  

Both spaces in \eqref{Dl} have fine moduli, and therefore, if $K$ is a number field, a $K$ point on either space produces a representing cover. 
\begin{defn} Call $\hat \phi$ from \eql{Dl}{Dlb} an {\sl involution realization\/} ($\C$-realization) over a number field $K$, if $\hat \phi$ corresponds to $\hat \bp\in \sH(D_{\ell^{k\np1}},\bfC_{2^r})^\inn(K)$. \end{defn}

For inner Hurwitz spaces, where the representing object includes an isomorphism up to conjugacy with the Galois group, the following makes sense. Choose $\hat \phi$ from \eql{Dl}{Dlb}. Mod out by the normal subgroup $U$  generated by an element of order $\ell^{k\np1}$ in $G_{k\np1}$. 

That gives an unramified cover $\hat W \to \hat W/U\eqdef X$; $\phi_X: X\to \prP^1_z$ is a hyperelliptic cover -- of genus $\geng_r=\frac{r\nm 2}2$ -- of $\prP^1_z$, ramified at $r$ points.  Denote the divisor classes of degree $t$ on $X$ by $\Pic^t(X)$, and the $\ell^{k\np1}$ torsion points on $\Pic^0(X)$ by $T_{\ell^{k\np1}}(X)$.  

Denote a copy of $\bZ/\ell^{k\np1}$ on which $G_\bQ$ acts as if it was the multiplicative group generated by $e^{{2\pi i}/\ell^{k\np1}}$ by $\mu_{\ell^{k\np1}}$. By replacing $G_\bQ$ by $G_K$, $K$ a number field, we can use the same notation. Prop.~\ref{dihcase} is \cite[Lem.~5.3]{DFr94}. 

\begin{prop} \label{dihcase} The set of involution realizations of $D_{\ell^{k\np1}}$ over $K$ associated to a fixed $X$ injects into the $G_K$ equivariant injections from $\mu_{\ell^{k\np1}}$ into $T_{\ell^{k\np1}}(X)$. It is onto, if in addition $\Pic^1(X)$ has a $K$ point.\footnote{The curve $X$ naturally embeds in  $\Pic^1(X)$. To get a $K$ embedding in $\Pic^0(X)$ requires translation by a $K$ divisor class of degree 1.}  \end{prop} 

Relating all algebraic points on both spaces requires considering all number fields $K$. Then, to get an isomorphism of the underlying spaces, 

\begin{equation} \label{hypequiv} \begin{array}{c} \text{equivalence hyperelliptic covers $\phi_X: X\to \prP^1_z$ and $\phi_{X'}: X'\to \prP^1_z$} \\ \text{ if they differ by composing $\phi_X$ and $\phi_{X'}$ by an $\alpha\in \PSL_2(\bC)$.}\end{array}\end{equation} This is reduced equivalence. Also,   consider the crucial \RIGP\ case,  $K=\bQ$. This says involution realizations of dihedral groups over $\bQ$ give {\sl cyclotomic\/} points on hyperelliptic jacobians. Then,  for a given $X$, if $\Pic^1(X)$ is isomorphic to $\Pic^0(X)$ over $\bQ$,  the converse is true. 

\begin{quest} \label{C-realvshyp} Fix $\ell$ and $r$. We ask these two questions for each $k\ge 0$.  
\begin{edesc} \label{mcdih} \item \label{mcdiha} Is there a $\C$-realization $\hat \phi_k: \hat W_k\to \prP^1_z$ over $\bQ$ in $\ni(D_{\ell^{k\np1}},\bfC_{2^r})^\inn$? 
\item \label{mcdihb}  Is there a $\mu_{\ell^{k\np1}}$ point on a hyperelliptic jacobian of dimension $\geng_r$? 
\end{edesc}  Prop.~\ref{dihcase} says \lq\lq Yes\rq\rq\  to \eql{mcdih}{mcdiha} implies \lq\lq Yes\rq\rq\    to \eql{mcdih}{mcdihb}.
\end{quest} 

\begin{equation} \label{impossdih} \begin{array}{c} \text{Impossibility of \eql{mcdih}{mcdiha} is the dihedral version of the Main} \\ \text{ \RIGP\ Conj.~\ref{mainconj}. Impossibility of \eql{mcdih}{mcdihb} is the hyperelliptic} \\ \text{version of the {\sl Torsion Conjecture\/} \ref{Tormain} on abelian varieties.} \end{array}\end{equation} 

\sdisplay{\cite{FrK97}} and \sdisplay{\cite{CaD08}} have a version that applies to any $(G,\ell)$ with $\ell$-perfect (Def.~\ref{lperfectdef}) $G$. Further, the latter gives good reasons to explicitly relate Hurwitz spaces and classical spaces as does Quest.~\ref{C-realvshyp}. That includes our \OIT\ case in \sdisplay{\cite{FrH20}}. 
 
\subsubsection{An abstract  for \cite{Fr20}} \label{Fr20abstract}

\begin{center} {\large Monodromy, $\ell$-adic representations \\ and the  Inverse Galois Problem}   \end{center} 

This book connects to, and extends, key themes in two of Serre's books:

\begin{edesc} \label{serreBooks} \item  \label{serreBooksa} \emph{Topics in Galois theory} (the original and enhanced reviews \cite{Fr94}, and in French, translated by Pierre D\`ebes \cite{D95}); and
\item \label{serreBooksb} \emph{Abelian $\ell$-adic representations and Elliptic Curves\/} (see the 1990 review by Ken Ribet \cite{Ri90}). \end{edesc} 
Its theme is to relate $\ell$-adic representations, as in generalizing Serre's \OIT, and the \RIGP\ using \MT s. Indeed,  Galois theory/cohomology  interprets many problems, not just the \RIGP, along the way.  

\cite[\S 4]{Fr20} explains \MT s (started in 1995) as a program motivated by such a relationship. \cite[ \S 1 and \S 2]{Fr20} includes exposition tying research threads prior to \MT s. Especially, it recasts \cite{FrV91} and \cite{FrV92} to  modernize investigating moduli definition fields of components of Hurwitz spaces vis-a-vis lift invariants with examples.

Then, we interpret the \OIT\ in a generality not indicated by Serre's approach.   \cite[\S 5]{Fr20} uses one example -- in  that all modular curves are one example -- clearly not of modular curves. This explains why our approach to (families of) covers of the projective line can handle a barrier noted by Grothendieck to generalizing the \OIT. \cite[\S 3]{Fr20} joins the 1st and last 3rd of the book, in an approach to the lift invariant  and Hurwitz space components based on the {\sl Universal Frattini cover\/} of a finite group.

The {\sl lift invariant\/} (Def.~\ref{Anliftdef}) is a tool that applies to many aspects of Hurwitz spaces. That includes giving distinguishing characteristics to components and cusps. For example:
\begin{edesc} \label{cusplift} \item \label{cusplifta}  The criterion for nonempty \MT s  \S\ref{nonemptyMT}. 
\item \label{cuspliftb}  Classifying cusps and components on which they lie;  geometrically separating them for $G_\bQ$ action \eqref{cusptype} in \sdisplay{\cite{Fr06} and \cite{CaTa09}}. 
\item \label{cuspliftc}  Using \eql{cusplift}{cuspliftb} and \S\ref{prodladic} $\ell$-adic representations, for labeling abelian variety collections a' la the Torsion Conjecture \sdisplay{\cite{CaD08}}.
\end{edesc} 
There are still unsolved problems. Many, though, as listed in Rem.~\ref{braidorbits2} on Thm.~\ref{level0MT}, have precise formulations and corroborating evidence. 

\subsubsection{Defining \MT  s}  \label{outlineMTs} Ex.~\ref{Dlk+1}, even before  \S\ref{braids} constructs Hurwitz spaces, gave an  example on how the \RIGP\ fits with properties of some classical spaces. In the same style we give  ingredients from which we form  \MT s. 

For a given finite group $G$, as with the RIGP, start with a set of generating conjugacy classes $\bfC$ (or a related collection of such).

\begin{defn} \label{frattdef} Call a cover of groups, $\psi: H\to G$,   {\sl Frattini\/} if, given a subgroup $H^*\le H$ with $\psi(H^*)=G$, then $H^*=H$. It is $\ell$-Frattini if in addition the kernel is an $\ell$-group.\footnote{Called a Frattini $p$-cover in \cite[Rem.~22.11.9]{FrJ86}$_2$.}
\end{defn} 
\begin{defn} \label{lperfectdef} If $\ell| |G|$, then $G$ is $\ell$-perfect if it has no $\bZ/\ell$ quotient. \end{defn}  

For the related \OIT,  we build on a (assumed nonempty) Hurwitz space $\sH(G,\bfC)^\inn$ formed in \S\ref{braids}. Initial ingredients for forming \MT s: 
\begin{equation} \label{GlCing} \begin{array}{c} \text{a prime $\ell| |G|$, as in \S\ref{univfratpre}, for which} \\ \text{elements of $\bfC$ are $\ell'$, and $G$ is $\ell$-perfect.}\end{array}\footnote{Dropping this $\ell'$ assumption is sometimes important, but we leave that out here.} \end{equation}  

Then, as in \S\ref{lfratquot-ab}, there are a {\sl finite number\/} (with indexing set $I$) of projective sequences $\{{}_iG_k\}_{k=0}^\infty$, $i\in I$, ${}_iG_0=G$, with limit ${}_i\tilde G=\lim_{\infty \leftarrow k}\ {}_iG_k$ having  the following properties. 
\begin{edesc}  \label{abquot-def} \item ${}_i\tilde \psi: {}_i\tilde G\to G$ is an $\ell$-Frattini cover;  and 
\item  $\ker({}_i\tilde \psi)$ is a $\bZ_\ell[G]$ module, free  of finite $\bZ_\ell$ rank ${}_i m$; with 
\item  ${}_i m$ the same as the vector space dimension of the common $\bZ/\ell[G]$ module $\ker({}_iG_{k\np 1} \to G)/ \ker({}_iG_k \to G)$, $k\ge 0$.\end{edesc}  

There is an $i_{\mx}\in I$ for which 
$$ \text{each ${}_i\tilde \psi: {}_i\tilde G\to G$ is a quotient of ${}_{i_{\mx}} \tilde \psi: {}_{i_{\mx}}\tilde G\to G$.}$$ \S\ref{lfratquot-ab} denotes  ${}_{i_{\mx}}\tilde G$ by $\tG \ell {}_\ab$; ${}_{i_{\mx}}m$ is a characteristic number attached to $(G,\ell)$. For $i\in I$, form a tower of Hurwitz spaces ${}_i\bH= \{\sH({}_iG_k, \bfC)^\inn\}_{k=0}^\infty$. 

For a given $i$, regard a component $\sH'$ associated to the Nielsen class braid orbit $\sO'$ on  $\sH({}_iG_k, \bfC)^\inn$ as a level $k$ vertex of a graph. Attach an  arrow from $\sH'$ to $\sH''\leftrightarrow \sO''$ on $\sH({}_iG_{k\np1}, \bfC)^\inn$ if the homomorphism ${}_i G_{k\np1}\to {}_i G_{k}$ induces a map (surjection) $\sO''\to \sO'$. 
$$\text{Components form a {\sl component tree\/} on ${}_i\bH$.}$$  

\begin{defn} \label{MTdef} A \MT\ for $i\in I$ is a  (nonempty) projective system of absolutely irreducible components on  ${}_i\bH$ with a unique vertex on each level of the ${}_i\bH$ component tree. 
 \end{defn} 

\S\ref{nonemptyMT} gives an if and only if criterion for ${}_i\bH$ to be nonempty (at all levels). 
\begin{lem} From the Tychonoff Theorem,\footnote{This is as applied to using Falting's Theorem as in \S\ref{toughestpt}.} conclude that if ${}_i\bH$ passes this criterion, there is at least one \MT\ on ${}_i\bH$. \end{lem} 

Ex.~\ref{obstcomps} lists cases explaining such phenomena  as in \cite[Chap.~9]{Se92}. 

For $i\in I$, the notation shows $\bfC$ interprets meaningfully in ${}_iG_k$, $k\ge 0$.  Given $(G,\ell,\bfC,i)$, there is a related Nielsen limit class: 
\begin{equation} \label{RIGPOITNCs} \begin{array}{c} \text{$\ni(G^\lm,\bfC)$ referring to the {\sl Nielsen limit\/} group, } \\ \text{$G^\lm$, giving a final step in the \OIT. }\end{array}\end{equation}
Again: In the direction $(G,\bfC)\Rightarrow(G^\lm,\bfC)$, $\bfC$ didn't change. We  explain how $G^\lm$  works in \S\ref{limgptie}, comparing with its arising in Serre's \OIT\ as in Ex.~\ref{ncmodcurves} and in  Rem.~\ref{nieslimrem} for our example. \footnote{Often both the \RIGP\ and the \OIT\   consider reduced Nielsen classes. For one, that suits the moduli problems as in  Ex.\ref{Dlk+1}.} 

\begin{exmpl}[Modular curves] \label{ncmodcurves} The \MT\ hyperelliptic jacobian case in Ex.\ref{Dlk+1} for $r=4$ is the isomorphism of the levels in the respective two sequences $\{\sH(D_{\ell^{k\np1}},\bfC_{2^4})^{\inn,\rd}\}_{k=0}^\infty$ and  $\{X_1^0(\ell^{k\np1})\}_{k=0}^\infty$.  The superscript ${}^0$ denotes restricting the classical modular curve $X_1$ to $\prP^1_j\setminus{ \{\infty\}}$. 

The corresponding $G^\lm$ in \eqref{RIGPOITNCs} is $(\bZ/\ell)^2\xs \bZ/2$ \cite[Chap.~6 \S 3]{Fr20}. The modular curves $X_0(\ell^{k\np1})$ appear in a similar sequence, with those spaces identified with $\{\sH(D_{\ell^{k\np1}},\bfC_{2^4})^{\abs,\rd}\}_{k=0}^\infty$,\footnote{Here $N_{S_{\ell^{k\np1}},\bfC}(D_{\ell^{k\np1}})=\bZ/\ell^{k\np1}\xs (\bZ/\ell^{k\np1})^*$.} including  the variant for the general case of even $r$ in Ex. \ref{Dlk+1}. 
\end{exmpl} 

\subsubsection{\MT s; the remaining sections}  \label{remainsects} \MT\ levels ($k\ge 0$) are moduli of $\prP^1_z$ covers, up to precise equivalences \S\ref{equivalences}. Computations use an {\sl automatic braid action\/} on Nielsen classes \S\ref{braids}, starting with  {\sl Orbit principle\/} \eqref{braid-comps}.  
\begin{equation} \label{braid-comps} \begin{array}{c}  \text{Absolutely irreducible components of a Hurwitz space $\leftrightarrow$} \\ \text{ braid (\S\ref{braids}) orbits on the corresponding Nielsen class $\ni(G,\bfC)$.}\end{array} \end{equation} 

The \S\ref{mainex}  example shows precisely how computations work, to interpret  the  major problems.  \S\ref{modtowdef} defines the moduli spaces of the levels of a \MT\ for a given $(G,\bfC)$,  starting with defining the  {\sl Universal Frattini cover\/}, $\tilde G$, of $G$ and  $\tilde G_\ab$, its {\sl abelianized\/} version. This produces a tower of Hurwitz spaces  (from Princ.~\ref{princHS}) in the notation of \eqref{projhurworbits}. 

Def.~\ref{evenfratt} gives the key  -- {\sl eventually $\ell$-Frattini} -- for formulating the \OIT\ conjecture, Conj.~\ref{OITgen}.  \S\ref{uselfratt} explains why that generalizes Serre's \OIT\ for modular curves. Also, its resemblance to a grand version of {\sl Hilbert's Irreducibility Theorem}. 

Basic assumptions.  The Hurwitz spaces, $\sH(G,\bfC)$,  will be reduced versions of inner equivalence classes of covers in the Nielsen class.  Denote projective $r$-space by  $\prP^r$.  For a given \MT\ its  whole system of tower levels  covers a natural configuration space. 

\begin{equation} \label{SL2C} \begin{array}{c}  \text{Inner space target: $U_r\eqdef \prP^r$  minus its {\sl descriminant locus}.} \\ \text{Add reduced equivalence:  the target space is $U_r/\SL_2(\bC)\eqdef J_r$.}\end{array} \end{equation}  

Then, $J_4$ is the classical $j$-line, minus $\{\infty\}$. We put $\{\infty\}$ back whenever we need it.  \S\ref{mainex}  has one of the collection of examples with results toward the main conjecture of this paper. All examples start with assumptions about $(G,\bfC)$ -- using the \BCL\ -- to assure any $G_\bQ$ orbit of points over $\bar \bQ$ on a Hurwitz space remains on the Hurwitz space. 

In addition to Orbit Principle \eqref{braid-comps}, we can {\sl usually\/}  give geometric names to $G_\bQ$ orbits of absolutely irreducible components on the levels of a \MT\ (starting with level 0) based on Cusp Principle \eqref{cusps-comps}.  This uses a subgroup -- {the cusp group\/} Def.~\ref{cusp-group} -- of the braid group.   

\begin{equation} \label{cusps-comps} \begin{array}{c} \text{Cusp group orbits on a braid orbit $\leftrightarrow$ cusps on the component;} \\ \text{the form of cusp orbit elements distinguish braid orbits \eqref{cusptype}.}  
\end{array} \end{equation} Typically we can geometrically distinguish different components in a Nielsen class, corresponding to additional moduli properties recognized by $G_\bQ$. 

\S\ref{l'RIGP} is on a theme -- $\ell'$ \RIGP\ applied to a finite group $G$. This generalizes for each $G$ how Ex.~\ref{Dlk+1} connects dihedral groups and cyclotomic points on hyperelliptic Jacobians. 

 \S\ref{toughestpt}  is on the difficulties the \RIGP\ causes for the \OIT. Especially around the hardest case in Serre's \OIT\ requiring Falting's Theorem for its completion. That was many years after Serre's program started. 

The modular curve $X_1(\ell)\to \prP^1_j$, attached to $(D_\ell,\bfC_{2^4},\ell)$, with one component, includes cusps we call {\sl Harbater-Mumford} (\HM; Def.~\ref{HMrep}). So, we label the unique component as \HM. Other types of distinguishing cusps appear as $(G,\bfC)$ changes, as in our examples. 

\S\ref{overview} gives background, motivation and the conjectures driving \cite{Fr20}.  \S\ref{guides} shows how the modular curve case expanded beyond dihedral groups. 

\S\ref{S1title} lists a series of problems with a large literature outside the considerations of modular curves that suddenly had a connection to them using parts of Serre's version of the \OIT.  

\S\ref{explainingFrattini} goes through three phases of the \MT\ timeline. This led to the main conjectures by which the \RIGP\ formulates entirely in terms of Hurwitz spaces, finally tying to a general version of the \OIT.

The history: \S\ref{pre95}  is prior to 1995, \S\ref{95-04} goes to 2004 with the main \RIGP\ conjectures on \MT s.  Then the progress on cusps, and Frattini ideas led to connecting the \RIGP\ to the \OIT.  

\S\ref{05-to-now} discusses papers in the period of formulating the generalization of the \OIT.  This continues Ex.~\ref{mainex}, as Thm.~\ref{level0MT},  describing level 0  of  a \MT\  beyond Serre's \OIT. This example is of  \MT\ levels for all primes $\ell$ where the number of components increases both with $\ell$ and, for each $\ell$, with the level $k$. Yet the {\sl tree structure\/} (Def.~\ref{MTdef}) of those components remains coherent from level to level based on the lift invariant. 

{\sl $\ell$-Frattini covers\/} arises in two profinite ways.  Consider a \MT,  $\bH=\{\sH_k\to J_r\}_{k=0}^\infty$: a projective sequence of absolutely irreducible Hurwitz space components covering $J_r$, attached to $(G,\bfC, \ell)$, $r=r_{\bfC}$ (as in  \S\ref{outlineMTs}) satisfying c\eqref{GlCing}.
\!\!\!\begin{edesc} \label{sepfrat} \item \label{sepfrata}  The construction of $\bH$ comes from the universal abelianized $\ell$-Frattini cover, $\tfG \ell {}_\ab$, of $G$. 
\item \label{sepfratb}  The guiding \MT\ conjecture is that the sequence of geometric monodromy groups of  $\bH$ is eventually $\ell$-Frattini. 
\end{edesc} 

\section{Changing from isogonies to sphere covers} \label{WhatGauss}  \S\ref{whyrationalfuncts}  goes from rational function covers of $\prP^1_z$ to  any covers of the sphere to produce all we need. The starting objects are {\sl Nielsen classes\/} attached to $(G,\bfC)$ with  $G$ a finite group, and $r=r_\bfC$ conjugacy classes, $\bfC$, in $G$.  Thereby we introduce the basic moduli -- reduced {\sl Hurwitz\/} -- spaces, $\sH(G,\bfC)^\rd$. 

\MT s  generalizes modular curves as in Ex.~\ref{Dlk+1}, starting from essentially any finite group G and prime $\ell$ for which $G$ is $\ell$-perfect (Def.~\ref{lperfectdef}).  This dihedral case is discussed in  \sdisplay{\cite{DFr94}} and surveyed in \cite{Fr94} (reprinted in several languages by Serre). 

That generalization starts with an \RIGP\ statement with no reference to \MT s. Yet, it  forces the existence of \MT  s using a diophantine conjecture (no rational points at high levels), that generalizes to \MT s (Main \RIGP\ Conj.~\ref{mainconj}) that  there are no points on high modular curve levels. 
 
We can compute properties of the  tower levels from assiduous use of an Artin braid group quotient: the {\sl Hurwitz  monodromy group\/} $H_r$ (and its significant subgroups). \S\ref{braids} defines \MT s. 

\S\ref{mainex} is an explicit example on  using Nielsen classes, labeling cusps, and computing genuses using Thm.~\ref{genuscomp}, as alluded to in \S\ref{compsr=4} (when $r=4$) for computing properties of reduced Hurwitz spaces as covers of $\prP^1_j$. 

\subsection{Part I: Data for sphere covers} \label{whyrationalfuncts} 

Here is notation for a rational function $f\in \bC(w)$: $f: \prP^1_w \to \prP^1_z: w \mapsto f(w)=z$:  $$f = f_1(w)/f_2(w),\text{ with $(f_1,f_2)=1$, $n=\deg(f)=\max(\deg(f_1),\deg(f_2))$. }$$
$$\text{For example: $\frac{w^3\np1}{w^5\nm3}$ has {\sl degree} $\deg(f)=5$.}$$  

\subsubsection{Branch points and local monodromy} {\sl Branch points} are places $z'$ where the distinct $w' \mapsto z'$, $\row {w'} t_{z'}$, have cardinality $t_{z'}<n$. Take $z'=0$. For simplicity assume $$\text{no $w'= \infty$ ($\deg(f_1)\ge \deg(f_2)$), and  the $f_i\,$s have leading coefficient 1.}$$ We write solutions, $w$, of $f(w)=z$, as analytic functions in a recognizable variable.  For $1\le k\le t$, write $f(w) = (w - w_k')^{e_k} m_k(w)$ with $m_k(w_k') \not = 0$. Form an expression in the variable $z^{1/e_k} = u_k(z)$: 
$$w_k(z^{1/e_k})\eqdef w_k' + u_k(z) + a_2 u_k(z)^2  + a_3 u_k(z)^3  \dots = w_k(u_k(z)), k = 1,\dots, t_{z'}. $$ 

Substitute $w_k(z^{1/e_k})\mapsto w$ in $f(w)=z$. Look at leading powers of $u_k(z)$ on the left and on the right. They are equal. Now solve inductively for $a_2, a_3, \dots$, so  the left side is identically equal to $z$. Easily verify these.  
\begin{edesc} 
\item The result for $w_k(z^{1/e_k})$ is analytic in a neighborhood of $u_k(z) = 0$. 
\item With $\zeta_{m}=e^{2\pi \sqrt{-1}/m}$, substitution(s) $u_k(z) \mapsto \zeta_{e_k}^j u_k(z)$, \\ $j= 1,\dots, e_k$,  give $e_k$ distinct solutions $w\in \bC((u_k(z)))/\bC((z))$.  \end{edesc} 
Take $\bar e\eqdef \bar e_{z'} = \lcm(\row e t_{z'})$. Now, put those solutions together. 

Write all $w_k(\zeta_{e_k}^jz^{1/e_k})\,$s, $k=1,\dots,t_{z'}$, as power series in $z^{1/\bar e}$:  
$$\text{ Substitute the obvious power of $z^{1/\bar e}$ for each  $u_k(z)$.}$$ 
This gives $n$ distinct solutions, $L_{z'}$, of $f(w) = z$ in the field $\bC((z^{1/\bar e}))$ and a natural permutation on $L_{z'}$ from the substitution  
\begin{equation} \label{powseries} \hat g_{z'}: z^{1/\bar e} \mapsto e^{2\pi i/\bar e} z^{1/\bar e}. \end{equation} 
This gives an element in  $S_n$ (the identify, if all $e_k\,$s are 1), the symmetric group, on the letters $L_{z'}$.\footnote{It  abuses notation by still having $z'=0$, but we are about to drop that by using $\row z r$ for a labeling of the branch points.}

Do this for each branch point, $\row {z'} r$, to get $\bg\eqdef (\row g r)$. 
\begin{edesc} \label{branchcyc} 
\item  \label{branchcyca}  How can we compare entries of $\bg$,  by having them all act on one set of symbols\footnote{When it is convenient we will take these symbols to be  $1, \dots, n$.}   rather than on $r$ different sets, $L_{z_1},\dots, L_{z_r}$? 

\item  \label{branchcycb} With success on \eql{branchcyc}{branchcyca}, denote the group $\lrang{\bg}$ the $\bg$ generate   by $G_f$. Of what use is it and $\bg$?
\item \label{branchcycc} Was there anything significant about using rational functions?  
\end{edesc} 

Below I answer questions \eqref{branchcyc}. In \S\ref{deform}, with those answers,  as a preliminary, we extend to do the following. 
\begin{equation} \label{deform} \text{Invert this, to get $\bg \mapsto f=f_\bg$, regardless of the branch points.}\footnote{Even if we restricted to rational functions, this inversion is nontrivial.}\end{equation}  

I produce $G_f$, up to isomorphism as a subgroup of $S_n$, answering \eql{branchcyc}{branchcyca}, first using {\sl Algebraic Geometry\/}, then using {\sl Analytic Geometry}. 

\subsubsection{Algebraic Geometry} \label{alggeom} Take any cover, $\phi:W\to \prP^1_z$, not necessarily rational given by $f$ above. We introduce a compact Riemann surface cover, $\hat \phi:  \hat W \to \prP^1_z$, the {\sl Galois closure\/} of $\phi$, minimal with these properties:  
\begin{edesc} \label{galoisClos} \item $\hat \phi$ factors through $\phi$; and   
\item \label{galoisClosb} it is a {\sl Galois cover\/} of $\prP^1_z$. 
\end{edesc} The group of those automorphisms is $G_\phi$.  
The phrase {\sl Galois cover\/} in \eql{galoisClos}{galoisClosb} means having $\deg(\hat \phi)$ automorphisms commutating with $\hat \phi$. 

Form the fiber product of $\phi$, $\deg(\phi)=n$ times (assume $n>1$): $$(\prP^1_w)_\phi^{(n)} \eqdef \{(w_1,\dots, w_n) \in (\prP^1_w)^n \mid \phi(w_1)=\cdots = \phi(w_n) \}\text{ over $\prP^1_z$.}$$ The resulting object is {\sl singular\/} if two coordinates, $w_k',w_l'$, lying over the same branch point $z_i,$ have both $e_k>1$ and $e_l > 1$. To see this, project onto the $(k,l)$ coordinates through the point $(w_k',w_l')\in \prP^1_w\times_{\prP^1_z}\prP^1_w$. Around $(w_k',w_l')$, this space (with its map $\phi$) is locally analytically isomorphic to $$\{ (w_k,w_l) \in D_{w'_k=0} \times_{D_{z'=0}} D_{w'_l=0} | w_k^{e_k} - w_l^{e_l} = 0 \}\to D_{z'=0}.$$  

The {\sl Jacobian criterion\/} reveals the singularity at $(0,0)$:  both partials $\pa{\ } {w_k},\pa{\ } {w_l}$ of $w_k^{e_k} - w_l^{e_l}$ are 0 at (0,0). Indeed,  normalizing this as a cover of $D_{z=0}$ results in $\gcd(e_k,e_l)$ copies of the cover $w \mapsto w^{\lcm(e_k,e_l)}$.\footnote{A disk, $D_{z'=0}$, around $z'=0$, is a convenient open set, as we see in \S\ref{braids}. Technically, that means someone has selected a metric on, say, $\prP^1_z$.} 

Denote the normalization by ${}^*W^{\{n\}}_\phi$. The normalization may have several components. One for certain is the {\sl fat diagonal}, $$\Delta_\phi^{\{n\}} =  \text{closure of the locus where 2 or more of those $w_i \,$s are equal.}$$ Remove the components of $\Delta_\phi^{\{n\}}$. On the result, there is a natural action of $S_n$, by permuting those distinct $w_i\,$ s, that extends to the whole normalized (since 1-dimensional, nonsingular) ${}^*W^{\{n\}}_\phi$. 

If ${}^*W^{\{n\}}_\phi$ is irreducible, then it is a Galois over $\prP^1_z$ with group $S_n$. If it is {\sl not\/} irreducible, consider a component, $\tilde \phi: \tilde W_\phi\to \prP^1_z$. Then, $$G_\phi= \{ g\in S_n | g \text{ preserves }  \tilde W_\phi\}, \!\!\text{ {\sl geometric\/} monodromy of $\phi$: $|G_\phi |= \deg(\tilde  \phi)$.}$$  

Denote the conjugacy class of $g\in G_\phi$ by $\C_g$. Though $\bg=(\row g r)$ still depends on how we labeled points over branch points, this approach does define the conjugacy classes $\C_{g_i}$, $i= 1,\dots, r$, in $G_\phi$ as follows. Consider maps of the function field of $W_\phi$ into $\bC(((z-z_i)^{1/{\bar e_i}}))$ fixed on $\bC((z-z_i))$ and restrict the automorphism $\hat g_{z_i}$ of \eqref{powseries} to the image of that embedding. 

For many purposes, this construction is inadequate to that of \S\ref{analgeom}. Since, however, it is algebraic, it gives another group we need.  

\begin{defn} \label{arithmon}  If  the cover $\phi$ is defined over a field $K$, then $\hat G_{\phi,K}\eqdef \hat G_\phi$,   the {\sl arithmetic monodromy\/} of $\phi$ (over $K$), is defined exactly as above, except take $\hat \phi: \hat W \to \prP^1_z$ to be a component defined over $K$. \end{defn} That is, if we started with $\tilde W$, a geometric component, then we would take for $\hat \phi: \hat W \to \prP^1_z$, the union of the conjugates $\tilde \phi^\sigma: \tilde W^\sigma\to \prP^1_z$, $\sigma \in G_K$. This would be defined and irreducible over $K$. 

\subsubsection{Analytic Geometry}  \label{analgeom} For $g\in S_n$ with $t$ disjoint cycles, $\ind(g)=n \nm t$ is its {\sl index\/}. 
For $\bfC$, $r$ conjugacy classes -- some may be repeated, count them with multiplicity -- in a group $G$, use $\bg\in G^r\cap \bfC$ to mean  an $r$-tuple $\bg$ has entries in {\sl some\/} order (with correct multiplicity) in $\bfC$. We denote the group the entries generate by $\lrang{\bg}$. 

\begin{exmp} If $G=S_4$, and $\bfC=\C_{2^2}\C_{3^2}$ consists of two repetitions each of the class, $\C_2$, of 2-cycles, and the class, $\C_3$, of 3-cycles, then  both $${}_1\bg=((1\,2), (2\,3\,4),(3\,4), (1\,3\,4))\text{ and } {}_2\bg=((2\,3\,4),(1\,3),(1\,3),(3\,2\,4)) $$ are in  $(S_4)^4\cap \bfC$. For both \eql{bcycs}{bcycsa} holds; only for ${}_2\bg$ does  \eql{bcycs}{bcycsb} hold.  \hfill $\triangle$
\end{exmp} 

Our next approach to the Galois closure, based on Thm.~\ref{BCYCs}, gives us a better chance to answer questions \eql{branchcyc}{branchcyca}-\eql{branchcyc}{branchcycc}.  To simplify notation, unless otherwise said, always make these two assumptions. 
\begin{edesc} \label{assumptions} \item  \label{assumptionsa} Conjugacy classes, $\bfC=\{\row \C r\}$, in $G$ are {\sl generating}. 
\item \label{assumptionsb} $G$ is given as a transitive subgroup of $S_n$.  \end{edesc} 
 Meanings: \eql{assumptions}{assumptionsa} $\implies$ the full collection of elements in $\bfC$ generates $G$; and \eql{assumptions}{assumptionsb} $\implies$ the cover generated by $\bg$ is connected. Even with \eql{assumptions}{assumptionsa}, it may be nontrivial to decide if there is $\bg \in G^r\cap\bfC$ that generates.
 
 Thm.\ref{BCYCs} is a version of {\sl Riemann's Existence Theorem\/} (\RET), done in detail in \cite{Fr80}. A less complete version is in \cite{Vo96}.   

\begin{thm} \label{BCYCs} Assume  $\bz\eqdef \row z r\in \prP^1_z$ distinct. Then, some degree $n$ cover $\phi: W \to \prP^1_z$ with branch points $\bz$, and $G=G_\phi\le S_n$ produces  classes  $\bfC$  in $G$, if and only if there is $\bg\in G^r\cap \bfC$ with these properties:  
\begin{edesc} \label{bcycs}   \item  \label{bcycsa}   $\lrang{\bg} = G$ ({\sl generation}); and 
\item  \label{bcycsb}   $\prod_{i=1}^r g_i = 1$ ({\sl product-one}). 
\end{edesc} 
Indeed, $r$-tuples satisfying \eqref{bcycs} give all possible Riemann surface covers -- both up to equivalence  (see \S\ref{equivalences})-- with these properties. 

Refer to one of those covers attached to $\bg$ as $\phi_\bg: W_\bg \to \prP^1_z$. 
\begin{equation} \label{RH} \text{The genus $\geng_{\bg}$ of $W_\bg$ appears in } 2(\deg(\phi) \np \geng_\bg -1) = \sum_{i=1}^r \ind(g_i). \end{equation}  \end{thm} 

The set of $\bg$ satisfying \eqref{bcycs} are the Nielsen classes associated to $(G,\bfC)$. A Riemann surface $W_\bg$  is isomorphic to $\prP^1_w$ over $\bC$ if and only if formula \eqref{RH} -- {\sl Riemann-Hurwitz\/} --  gives $\geng_\bg=0$. 

Denote $\prP^1_z\setminus \{\bz\}$ as $U_{\bz}$ and choose $z_0\in U_{\bz}$. Thm.~\ref{BCYCs} follows from existence of {\sl classical generators\/} of $\pi_1(U_{\bz},z_0)$. These are paths $\sP=\{\row P r\}$ on $U_\bz$ based at $z_0$, of form $\lambda_i\circ\rho_i\circ \lambda^{-1}_i$ with these properties. 
\begin{edesc} \label{generatorspi} \item  $\rho_i\,$s are non-intersecting clockwise loops around the respective $z_i\,$s. 
\item The $\lambda_i\,$s go from $z_0$ to a point on $\rho_i$. 
\item  Otherwise there are no other intersections.  
\item The $\row \lambda r$ emanate clockwise from $z_0$. 
\end{edesc}

Suppose a cover, $\phi:W\to\prP^1_z$, has a labeling of the fiber $w_{1}^\star,\dots w_{n}^\star$ over $z_0$. Then, analytic continuation of a lift, $P^\star_{k,i}$, of  $P_i$, starting at $w_{k}^\star$ will end at a point, say $w_{(k)g_i}^\star$, on $k\in \{1,\dots,n\}$. 

This produces the permutations $\row g r$ satisfying \eqref{bcycs}.  This results from knowing \eqref{generatorspi} implies $\row P r$  are generators of $\pi(U_\bz,z_0)$, and  
\begin{equation} \label{fundgp} \text{they have product 1 and {\sl no other relations\/}.} \end{equation} 
From \eqref{fundgp}, mapping $P_i\mapsto g_i$, $i=1,\dots,r$, produces a permutation representation $\pi(U_\bz,z_0)\to G\le S_n$. 

From the theory of the fundamental group, this gives a degree $n$ cover $\phi^0: W^0\to U_\bz$. 
Completing the converse to Thm.~\ref{BCYCs} is not immediate. You must fill in the holes in $\phi_0$ to get the desired $\phi: W\to \prP^1_z$. A full proof, starting from \cite{Ahl79}, is documented in \cite[Chap.~4]{Fr80}. 

\begin{defn} \label{NielsenClass} Given $(G,\bfC)$, the set of $\bg$ satisfying \eqref{bcycs} is the {\sl Nielsen class\/}  $\ni(G,\bfC)^\dagger$, with $\dagger$ indicating an equivalence relation referencing the permutation representation $T: G \to S_n$. \end{defn}
A cover doesn't include an ordering of its branch points. Adding such would destroy most applications number theorists care about. This makes sense of saying {\sl a cover is in the {\sl Nielsen class\/} $\ni(G,\bfC)^\dagger$}. 

\begin{rem}[Permutation notation] \label{permnot} Our usual assumptions start with a faithful transitive permutation representation $T: G\to S_n$ with generating conjugacy classes $\bfC$, from which we may define a Nielsen class $\ni(G,\bfC)^\dagger$.  Except in special cases, the notation for $T$ indicates the symbols on which elements $T(g)$, $g\in G$, acts are $\{1,\dots,n\}$. Example: Denote the stabilizer of the symbol 1 by $G(T,1)$ or just $G(1)$ if $T$ is understood. 

Or, if $T$ comes from a cover $\phi:W\to \prP^1_z$, denote the permutation representation by $T_\phi$ even when applied to $\hat G_\phi$. Then, denote the group of $\hat W/W$ by $\hat G_\phi(1)$. That indicates it is the subgroup stabilizing the integer 1 in the representation, and $G_\phi(1)=\hat G_\phi(1)\cap G_\phi$. \end{rem}  

\subsection{Part II: Braids and deforming covers} \label{braids}  

Take a basepoint ${}_0\bz$  of $U_r$, and denote $\pi_1(U_r,{}_0\bz)$, the {\sl Hurwitz monodromy group},\footnote{Other call this the sphere braid group. I first used it in 1968 while at the Institute for Advanced Study, when I discovered Hurwitz used it for a special case of what I generalized. It was a monodromy group in my work, and no topologist corrected me.}  by $H_r$. A {\sl Hurwitz space\/} is a cover of $U_r$ that parametrizes all covers in a Nielsen class. \S\ref{dragging} explains how it comes from a representation of $H_r$ on a Nielsen class. 

\S\ref{equivalences} lists the equivalences we use, and then the action of $H_r$ on Nielsen classes that defines Hurwitz spaces. Then, \S \ref{genus} gives the formula for computing the genus of reduced Hurwitz spaces when $r=4$. 

\subsubsection{Dragging a cover by its branch points} \label{dragging}  Here is the way to think of forming a Hurwitz space. 
Start with  ${}_0\phi: {}_0W\to \prP^1_z$, with branch points ${}_0\bz$, classical generators ${}_0\bP$ and (branch cycles) ${}_0\bg\in \ni(G,\bfC)$.  

Drag ${}_0\bz$ and ${}_0\bP$, respectively,  to ${}_1\bz$  and ${}_1\bP$ along any path $B$ in $U_r$. 
With no further choices, ${}_t\sP\mapsto {}_0\bg$ by ${}_tP_i\mapsto g_i$, $i=1,\dots,r$,  forms a trail of covers  
${}_t\phi: {}_tW\to \prP^1_z$, $t\in [0,1]$, {\sl with respect to the same ${}_0\bg$ along the path indicated by the parameter.}

This produces a collection of $\prP^1_z$ covers of cardinality $|\ni(G,\bfC)^\dagger|$ over each  $\bz\in U_r$. This forces upon us a decision: when to identify two covers as equivalent.  
For $B$ closed, denote the homotopy class $[B]$ as $q_B\in H_r$. 

\begin{princ} For $B$ a closed path, we can identify branch cycles ${}_1\bg$ for the cover ${}_1f: {}_1W\to \prP^1_z$ lying at the end of the path, relative to the original classical generators ${}_0\bP$ from ${}_0\sP \mapsto ({}_0\bg)q_B^{-1}$. \end{princ}

Here are key points going back to \cite[\S4]{Fr77}. 
 \begin{edesc} \label{keys} 
\item \label{keysa}  {\sl Endpoint of the Drag}: A cover at the end of $B$ is still in $\ni(G,\bfC)^\dagger$; it depends only on the homotopy class of $B$ with its ends fixed.  
\item \label{keysb}  {\sl $H_r$ orbits}: (Irreducible) components of spaces of covers in $\ni(G,\bfC)^\dagger$ $\leftrightarrow$ braid ($H_r$) orbits. 
\end{edesc}  

Whatever the application, we must be able to identify the Galois closure (\S\ref{alggeom}) of the cover. The key  ambiguity is in labeling $\bw^\star=\{w_1^\star,\dots,w_n^\star\}$, points lying over $z_0$. Changing that labeling changes $T: G\to S_n$. A slightly subtler comes from changing $z_0$. There is a distinction between them. Changing $z_0$ to $z^*_0$ is affected by rewriting the $z_i$-loops as \begin{equation} \label{baseptch} \lambda^*\circ\lambda\circ\rho\circ \lambda^{-1}\circ(\lambda^*)^{-1},  \text{ with $\lambda^*$ a path from }z_0^* \text{ to }z_0.\end{equation}

\subsubsection{Braids and equivalences} \label{equivalences}  In this paper we primarily need the $H_r$, $r=r_\bfC$, action on $\ni(G,\bfC)$.   

\begin{edesc} \label{Hrgens} \item  \label{Hrgensa} Two elements generate this  $H_r$ action: 
 $$\begin{array}{rl} q_i: & \bg\eqdef (\row g r)  \mapsto (\row   g 
{i-1},g_ig_{i+1}g_i^{-1},g_i,g_{i+2},\dots, g_r); \\ 
\sh:  &\bg\mapsto (g_2,g_3,\dots,g_r,g_1) \text{ and } H_r  \eqdef \lrang{q_2,\sh} \text{ with }\\ &\sh\,q_i\ \sh^{-1} = q_{i\np1},\   i=1,\dots,r\nm 1. \end{array}$$ 
\item \label{Hrgensb} From braids, $B_r$, on $r$ strings, $H_r=B_r/\lrang{q_1\cdots q_{r\nm1}q_{r\nm1}\cdots q_1}$. \end{edesc}    

The case $r=4$ in \eql{Hrgens}{Hrgensa} is so important in examples, that in  {\sl reduced\/} Nielsen classes, we refer to $q_2$ as the {\sl middle twist}. As usual, in notation for free groups modulo relations,  \eql{Hrgens}{Hrgensb} means to mod out by the normal subgroup generated by the relation $q_1\cdots q_{r\nm1}q_{r\nm1}\cdots q_1=R_H$. 

\begin{princ} \label{princHS} From \eqref{Hrgens}, we get a permutation represention of $H_r$ on $\ni(G,\bfC)^\dagger$. Given $\dagger$, that gives a cover $\Phi\eqdef \Phi^{\dagger}: \sH(G,\bfC)^\dagger \to U_r$: The Hurwitz space of $\dagger$-equivalences of covers. 

Elements in $\lrang{q_1\cdots q_{r\nm1}q_{r\nm1}\cdots q_1}$ have this effect :  \begin{equation} \label{inneraction} \begin{array}{c} \bg\in \ni(G,\bfC)\mapsto g\bg g^{-1}\text{ for some }g\in G. \text{ Indeed, for} \\ 
 \bg\in \ni(G,\bfC), \{(\bg)q^{-1}R_Hq\mid q\in B_r\}=\{g^{-1}\bg g\mid g\in G\}.\end{array} \end{equation}  \end{princ}

Denote the subgroup of the normalizer, $N_{S_n}(G)$, of $G$ in $S_n$ that permutes a given collection, $\bfC$, of conjugacy classes, by $N_{S_n}(G,\bfC)$.  Circumstances dictate when we identify  covers $\phi_u: {}_uW\to\prP^1_z$, $u=0,1$, branched at ${}_0\bz$, obtained from any one cover using the dragging-branch-points principle. 
Two equivalences that occur on the Nielsen classes:   
\begin{edesc} \label{eqname}  \item  \label{eqnamea} {\sl Inner}: $\ni(G,\bfC)^\inn\eqdef \ni(G,\bfC)/G$ corresponding to \eqref{inneraction}; and \item  \label{eqnameb} {\sl Absolute}: $\ni(G,\bfC)^\abs\eqdef \ni(G,\bfC)/N_{S_n}(G,\bfC)$.   
\end{edesc} One might regard {\sl Inner\/} (resp.~{\sl Absolute}) equivalence as {\sl minimal\/} (resp.~{\sl maximal}). Act by $H_r$ on either  equivalence (denoted by a $\dagger$ superscript). 

\begin{defn}[Reduced action] \label{redaction} A  cover $\phi: W\to \prP^1_z$ is {\sl reduced\/} equivalent to $\alpha\circ \phi: W\to \prP^1_z$ for $\alpha\in \PSL_2(\bC)$. \end{defn} 

Also,  $\alpha$ acts  on $\bz\in U_r$ by acting on each entry. Def.~\ref{redaction} extends to any cover  $\Phi^\dagger: \sH(G,\bfC)^\dagger\to U_r$, giving a {\sl reduced\/} Hurwitz space cover:
\begin{equation} \label{reducedcover}  \Phi^{\dagger,\rd}: \sH(G,\bfC)^{\dagger,\rd} \to U_r/\PSL_2(\bC)\eqdef J_r.\end{equation}

\subsubsection{Genus formula: $r=4$} \label{genus}  Identify $U_4/\PSL_2(\bC)$ with $\prP^1_j\setminus\{\infty\}$. 
A reduced Hurwitz space of 4 branch point covers is a natural $j$-line cover.  As in \S\ref{compsr=4}, that completes to $\overline \sH(G,\bfC)^{\dagger,\rd} \to \prP^1_j$ ramified over $0, 1, \infty$. 

\begin{defn} \label{cusp-group} Denote $\lrang{q_1q_3^{-1},\sh^2}$ by $\sQ''$. The {\sl cusp group\/}, $\Cu_4$, is the group that $\sQ''$ and $q_2$ generate. \end{defn} 

Cusps on the projective non-singular completion of the Hurwitz space over $\infty\in \prP^1_j$ have a purely combinatorial definition.\footnote{Nielsen classes allow precise definitions of cusps for all $r\ge 4$. For $r>4$, the cusp group is generated just by $q_j$, $j=2$. Relating this to cusps on, say, Siegel Upper-half spaces, hasn't yet been elaborated.}
\begin{equation}\label{cusp-rednc} \begin{array}{c} \text{As in \eqref{cusps-comps}, they correspond to $\Cu_4$ orbits on} \\ \text{ {\sl reduced Nielsen classes\/} $\ni(G,\bfC)^{\dagger,\rd}\eqdef \ni(G,\bfC)^\dagger/\sQ''.$} \end{array}\end{equation} 
\cite[\S4.2]{BFr02} proves the Thm.~\ref{genuscomp} formula using the \eqref{cusp-rednc} definitions. 

\begin{thm} \label{genuscomp} Suppose a component, $\overline{\sH'}$, of $\overline \sH(G,\bfC)^{\dagger,\rd}$ corresponds to an $H_r$ orbit, $O$, on the Nielsen classes $\ni(G,\bfC)^{\dagger,\rd}$. Then, points of ramification, respectively over $0,1,\infty$, of $\overline {\sH'}\to \prP^1_j$ correspond to $$\text{disjoint cycles of $\gamma_0=q_1q_2$, $\gamma_1=q_1q_2q_1$, $\gamma_\infty=q_2$ acting on $O$.}$$ 

The genus, $g_{\overline {\sH'}}$, of $\overline{\sH'}$, a la Riemann-Hurwitz,  appears from  $$2(|O|+g_{\bar \sH'}-1)=\ind(\gamma_0)\np \ind(\gamma_1)\np \ind(\gamma_\infty).$$ \end{thm} 

\begin{rem} Notice that $\gamma_1\gamma_2\gamma_3=1$ (product-one) is a conseqence of the Hurwitz braid relation $$q_1q_2\cdots q_{r\nm1}q_{\nm1}q_{r\nm2}\cdots q_1=R_H$$ combined for $r=4$ with modding out by $q_1=q_3$. Also, that immediately gives $\gamma_0^3=1$ in its action on reduced Nielsen classes. Hint: Use also the braid relations $q_iq_{i\np1}q_i=q_{i\np1}q_iq_{i\np1}$. \end{rem} 

\begin{rem}[Braid orbits] \label{braidorbits1} Thm.~\ref{genuscomp} shows that identifying braid orbits, $O$ in the Nielsen class is crucial. As in Rem.~\ref{braidorbits2} on the \MT\ from the level 0 Nielsen class \eql{mainexs}{mainexsc} as $\ell$ changes in \cite{FrH20} or \cite[\S 5]{Fr20}. \end{rem} 

\subsection{Main Example} \label{mainex} 
We now do one example -- using the definitions in \eqref{cusp-rednc} -- to illustrate the Thm.~\ref{genuscomp} genus calculation.  \cite{BFr02} is our main source for the theory and other examples illustrating, purposely chosen to show on one full example of a \MT. Something  one might care about if these were modular curves; though they are not. 

Level 0 of the \MT\ for that example is designated $\ni(A_5,\bfC_{3^4})$. It  has just one braid orbit, unlike that of \S\ref{l=2level0MT} which has two. 

\subsubsection{A lift invariant}  \label{l=2level0MT}   Before using general Frattini covers (Def.~\ref{frattdef}) in \S\ref{lfratquot}, suppose $\psi: H\to G$ is a {\sl central\/} Frattini cover: $\ker(\psi)$ in the center of $H$. 
The most famous Frattini {\sl central\/} extension (kernel in the center of the covering group) arises in {\sl quantum mechanics\/} from the spin cover, $\psi: \Spin_n\to O_n(\bR)$, $n\ge 3$, of the orthogonal group.  Regard $\ker(\psi)$ as  $\{\pm 1\}$. The natural permutation embedding of $A_n$ in $O_n$  induces the $$\text{Frattini cover $\psi: \Spin_n \to A_n$, abusing notation a little.}$$ 

A braid orbit $O$ of  $\bg=(\row g r)\in \ni(A_n,\bfC)$, with $\bfC$ conjugacy classes consisting of odd-order elements, passes the (spin) lift invariant test if the natural (one-one) map $\ni(\Spin_n,\bfC) \to \ni(A_n,\bfC)$ has image containing $\bg$.  Each $g_i$ lifts to a same-order element $\tilde g_i\in \Spin_n$. 

\begin{defn}[Lift invariant]  \label{Anliftdef}  Then, $s_{\Spin_n/A_n}(O)\eqdef \prod_{i=1}^r \tilde g_i\in \ker(\psi)$.  
Generally, for $\ell$-perfect $G$ and $\ell'$ conjugacy classes $\bfC$:\footnote{You can drop both assumptions, as in \cite[App.]{FrV91}, but the definition is trickier.}  for $s_{H/G}$ on a braid orbit $O$ on $\ni(G,\bfC)$ substitute  $A_n\rightarrow G$ and $\Spin_n\rightarrow H$ in $s_{\Spin_n/A_n}(O)$. 
\end{defn}

One result: If covers in  $\ni(A_n,\bfC)$ have genus 0, then $s_{\Spin_n/A_n}(O)$ depends only on $\ni(A_n,\bfC)$, not on $O$, and there is an explicit computation for it. 
\begin{exmpl} \label{gen0} For $n=4$, there are two classes of 3-cycles, $\C_{\pm 3}$, but just one for $n\ge 5$. For $\bg\in \ni(A_n,\bfC_{3^{n\nm1}})$, $n\ge 5$, $n\nm1$ repetitions of $\C_3$, $s_{\Spin_n/A_n}(\bg)=(-1)^{n\nm1}$. For $n=4$, the only genus 0 Nielsen classes of 3-cycles are $\ni(A_4,\C_{+3^3})$ and $\ni(A_4,\C_{-3^3})$, and the lift invariant is -1.

 The short proof of \cite[Cor.~2.3]{Fr10} is akin to the original statements I made to Serre for \cite{Se90a}.  
\end{exmpl} 

\subsubsection{$A_4$, $r=4$} This is the most natural case not included in Ex.\ref{gen0}, and it is level 0 for $\ell=2$ of our  illustration for generalizing Serre's \OIT\  \sdisplay{\cite{FrH20}}.   The Nielsen class is $\ni(A_4,\bfC_{+3^2-3^2})^{\inn,\rd}$ with  conjugacy classes a rational union. 
From the \BCL\ (Thm.~\ref{bclthm}) the  Hurwitz space has moduli definition field $\bQ$.  Here is what to expect.  

\begin{edesc} \label{A44} \item \label{A44a} The Hurwitz space has two components, labeled $\sH_0^{\pm}$, that we will see clearly using the $\sh$-incidence matrix \eqref{shincentries} and \sdisplay{\cite{BFr02}}. 
\item \label{A44b} The Hurwitz spaces have fine moduli, but neither component has fine {\sl reduced\/} moduli (criterion of \cite[Prop.~4.7]{BFr02}).\footnote{Both components have moduli definition field $\bQ$ as in Def.~\ref{moddeffield} for many reasons. The easiest: If they were conjugate, $G_\bQ$ would preserve their degrees over $\prP^1_j$.} 
\item \label{A44c} The Spin lift invariant Def.~\ref{Anliftdef} separates the components and each  component has genus 0 and a characteristic cusp type.    
\item \label{A44d} Neither component is a modular curve, but we can compute their arithmetic and geometric monodromy as  $j$-line covers. \end{edesc} 

Comment on \eql{A44}{A44b}: Fine moduli for inner Hurwitz spaces here comes from $A_4$ having no center. Checking fine moduli on a {\sl reduced\/} space braid orbit $O$ has two steps \cite[\S4.3.1]{BFr02}: $\sQ''$ must act as a Klein 4-group (called b(irational)-fine moduli);  and neither $\gamma_0$ nor $\gamma_1$ has fixed points (on $O$). 

Comments on \eql{A44}{A44c}: Thm.~\ref{Anlift} uses the spin lift invariant. Thm.~\ref{level0MT} uses a {\sl Heisenberg\/} lift invariant for $\ell\ne 2$ prime. 

\subsubsection{The $\sh$-incidence matrix}  \label{shincex} 
Subdivide $\mapsto \ni(A_3,\bfC_{\pm 3^2})^{\inn,\rd}$ using 
sequences of conjugacy classes $\C_{\pm 3}$;   $q_1q_3^{-1}$ and $\sh$ switch these rows: 
$$\begin{tabular}{cccc} &[1] +\,-\,+\,- &[2] +\,+\,-\,- &[3] +\,-\,-\,+ \\
&[4] -\,+\,-\,+ &[5] -\,-\,+\,+ &[6] -\,+\,+\,-
\end{tabular} $$

The rest of this example displays the two $H_r$ orbits on the Nielsen classes, and the geniuses of their corresponding Hurwitz space components.  Here is the \sh-incidence matrix notation for cusps, labeled $O_{i,j}^k$: $k$ is the cusp width, and $i,j$
corresponds to a labeling of orbit representatives.  \sdisplay{\cite[\S 2.10]{BFr02}} says much more about the $\sh$-incidence matrix, which works for all $r\ge 4$, an example of which we now present. \sdisplay{\cite{BFr02}} says more about  other examples on which it has been used. 

It has appeared where braid orbits would have been otherwise difficult to either compute or to display.\footnote{The point is that the display is illuminating. I remain leary of relying on \GAP\ or some other computer program without a corroborating proof.}  Its entries are \begin{equation} \label{shincentries} \begin{array}{c} \text{$|O\cap (O')\sh|$, with $(O,O')$ cusp orbits. Read cusp widths}\\ \text{ by adding entries in a given row of each block.} \end{array} \end{equation}  
 
 Consider $g_{1,4}=((1\,2\,3), (1\,3\,4), (1\,2\,4), (1\,2\,4))$. Its $\gamma_\infty$ orbit $O_{1,4}^4$ is what Thm.~\ref{level0MT} calls {\sl double identity\/} (repeated elements in positions 3 and 4).\footnote{Its shift gives a cusp of type o-$2'$ \eqref{cusptype}.} There are also two other double identity cusps with repeats in positions 2 and 3, denoted $O_{3,4}^1$ and $O_{3,5}^1$. The following elements are in a Harbater-Mumford component \eqref{HMrep}. 
 $$\begin{array}{rl} \text{H-M rep.} \mapsto \bg_{1,1}=& ((1\,2\,3),
(1\,3\,2), (1\,3\,4), (1\,4\,3))\\
\bg_{1,3}=& ((1\,2\,3), (1\,2\,4), (1\,4\,2), (1\,3\,2))\\
\text{H-M rep.} \mapsto  \bg_{3,1}=& ((1\,2\,3), (1\,3\,2), (1\,4\,3), (1\,3\,4))
\end{array} $$ 

\begin{table}[h] \label{sh-incA4}
\begin{tabular}{|c|ccc|}  \hline $\ni_0^+$ Orbit & $O_{1,1}^4$\ \vrule  &
$O_{1,3}^2$\ \vrule & 
$O_{3,1}^3$ \\ \hline $O_{1,1}^4$ 
&1&1&2\\ 
$O_{1,3}^2$ &1 &0&1  \\ $O_{3,1}^3$ &2&1&0 \\  \hline \end{tabular} 
\ 
\begin{tabular}{|c|ccc|} \hline $\ni_0^-$ Orbit & $O_{1,4}^4$\ \vrule  &
$O_{3,4}^1$\ \vrule & 
$O_{3,5}^1$ \\ \hline $O_{1,4}^4$ 
&2&1&1\\ 
$O_{3,4}^1$ &1 &0&0  \\ $O_{3,5}^1$ &1&0&0 \\  \hline \end{tabular} \end{table}

\begin{prop} \label{A43-2} On $\ni(\Spin_4,\bfC_{\pm3^2})^{\inn,\rd}$
(resp.~$\ni(A_4,\bfC_{\pm3^2})^{\inn,\rd}$)
 $H_4/\sQ''$ has one (resp.~two) orbit(s). So,  $\sH(\Spin_4,\bfC_{\pm3^2})^{\inn,\rd}$ 
 (resp.~\!$\sH(A_4,\bfC_{\pm3^2})^{\inn,\rd}$) has one (resp.~two)  component(s),  
 $\sH_{0,+}$ (resp.~$\sH_{0,+}$ and $\sH_{0,-}$). 

Then, $\sH(\Spin_4,\bfC_{\pm3^2})^{\inn,\rd}$ maps one-one to 
$\sH_{0,+}$ (though changing $A_4$ to $\Spin_4$ give different moduli).
The compactifications of $\sH_{0,\pm}$ both have genus 0 from Thm.~\ref{genuscomp} (Ex.~\ref{exA43-2}).\end{prop} 

The diagonal entries for 
$O_{1,1}^4$ and $O_{1,4}^4$ are nonzero. In detail, however, $\gamma_1$
(resp.~$\gamma_0$) fixes 1 (resp.~no) element of $O_{1,1}$, and neither of $\gamma_i$, $i=0,1$, fix any
element of $O_{1,4}^4$. Ex.~\ref{exA43-2} uses \eqref{shincentries} for the cusp widths.  

\begin{exmpl}[Compute the genus] \label{exA43-2} Use  $(\gamma_0,\gamma_1,\gamma_\infty)$ from the $\sh$-incidence
calculation in Prop.~\ref{A43-2}. Denote their restrictions to lifting
invariant $+1$  (resp.~-1) orbit  
 by $(\gamma_0^+,\gamma_1^+,\gamma_\infty^+)$ (resp.~$(\gamma_0^-,\gamma_1^-,\gamma_\infty^-)$).

Read indices of $+$ (resp.~$-$) elements from the $\ni_0^+$ (resp.~$\ni_0^-$) matrix block: 
Cusp  widths over
$\infty$ add to the degree $4+2+3=9$ (resp.~$4+1+1=6$) to give
$\ind(\gamma_\infty^+)=6$ (resp.~$\ind(\gamma_\infty^+=3$).  

As $\gamma_1^+$ (resp.~$\gamma_1^-$) has 1
(resp.~no) fixed point and $\gamma_0^\pm$ have no fixed points, $\ind(\gamma_1^+)=4$
(resp.~$\ind(\gamma_1^-)=3$) and 
$\ind(\gamma_0^+)=6$ (resp.~$\ind(\gamma_0^+)=4$). 
From 
 \begin{center} $2(9+g_+-1)=6+4+6=16$ and\\ $2(6+{g_-}-1)=3+3+4=10.$\end{center} the genus of $\bar \sH_{0,\pm}$ is $g_{\pm}=0$. 
\end{exmpl}

\section{$\ell$-Frattini covers and \MT s} \label{modtowdef}  \S\ref{univfratpre} is an overview of  the universal Frattini cover of a finite group $G$. \S\ref{lfratquot} introduces the universal $\ell$-Frattini cover of $\ell$-perfect $G$.  

\S\ref{lfratquot-ab} focuses on ${}_\ell\tilde \psi_{G_\ab}: \tG \ell {}_\ab\to G$, the abelianized $\ell$-Frattini cover. That gives \MT s,  a source of $\ell$-adic representations attached to $(G,\ell,\bfC)$ with  $\bfC$ consisting of generating $\ell'$ conjugacy classes in $G$. 

\S\ref{nonemptyMT} explains the cohomological obstruction to \MT\ levels being nonempty. That includes taking {\sl lattice\/} quotients of ${}_\ell\tilde \psi_{G_\ab}$. \S~\ref{exA43-2cont} continues Ex.~\ref{exA43-2} as a case included in \sdisplay{\cite{CaD08}}. \S\ref{archfrattconj}  starts the motivation from \cite{Se68} for the eventually $\ell$-Frattini Definition. 

\subsection{Universal $\ell$-Frattini covers} \label{univfratpre} 
To realize the relation between the \RIGP\ and diophantine questions about classical spaces, we have an aid in a geometric approach to groups that naturally correspond to points on such spaces. We use these definitions:  
\begin{edesc} \item Def.~\ref{lperfectdef}: Finite group $G$ is $\ell$-perfect.  
\item Def.~\ref{frattdef}: Frattini cover $\psi: H\to G$ of profinite groups:  $\ell$-Frattini if $\ker(\phi)$ is a pro-$\ell$ group. \end{edesc} 

There is a Universal profinite group, $\tilde G$,  for the Frattini covering property. Further, for each prime $\ell$ dividing $|G|$, there is a universal $\ell$-Frattini cover ${}_\ell\tilde \psi: \fG \ell \to G$ for the Frattini property with kernel an $\ell$ group, and by modding out by the commutator of $\ker({}_\ell\tilde \psi)$, an abelianized version ${}_\ell\tilde \psi_\ab: \fG \ell {}_\ab \to G$. A characteristic sequence  $\{\tfG \ell k {}_\ab\}_{k=0}^\infty$ of quotients of $\fG \ell {}_\ab$ canonically defines a series of moduli space covers of $\sH(G,\bfC)^\rd$ when the elements of $\bfC$ are $\ell'$ (prime to $\ell$). 

This is elementary and in  \sdisplay{\cite{Fr95}}, \cite[App.~B]{Fr20} and reviewed in \S\ref{lfratquot}. With the assumptions of \S\ref{outlineMTs} -- $G$ is {\sl $\ell$-perfect\/}  and $\bfC$ consists of $r=r_\bfC$  $\ell'$ conjugacy classes of $G$, with  $\ni(G,\bfC)$  nonempty -- form a profinite version of Nielsen classes, $\ni(\fG \ell {}_\ab,\bfC)$. 

Classes of $\bfC$ lift to classes of same order elements -- so we don't change the notation -- of the universal abelianized $\ell$-Frattini cover ${}_\ell\tilde \psi_\ab: \fG \ell {}_\ab\to G$, whose kernel is a finite rank $\bZ_\ell[G]$ module.

Then $H_r$ extends to this Nielsen class. A \MT\  for $(G,\bfC,\ell)$ consists of a profinite $H_r$ orbit. Denote the collection of these by $\sF_{G,\bfC,\ell}\eqdef \sF_{G,\bfC,\ell,{}_\ell\tilde \psi_\ab}$ (\S\ref{nilatqt}). The homomorphism ${}_\ell\tilde \psi_\ab$ has at least one proper $\ell$-Frattini lattice quotient $L^\star \to G^\star \mapright {{}_\ell\tilde \psi^\star} G$ of ${}_\ell\psi_\ab$ (Def.~\ref{latquotdef}) attached to each quotient of the characteristic $\ell$-Frattini module attached to $(G,\ell)$. 

To each lattice quotient the same definition, conjectures and variants on properties applies. \S\ref{prodladic} explains how $\ell$-adic representations appear from these \MT s.   In  cases  the \MT s, $\sF_{G,\bfC,\ell,{}_\ell\tilde \psi^\star}$, for proper quotients are variants on classical spaces. \sdisplay{\cite{CaD08}} expands on this. 

A cohomological condition, Thm.~\ref{obstabMT}, checks precisely for when $\sF_{G,\bfC,\ell, {}_\ell\tilde \psi^\star}$ is nonempty.\footnote{The criterion is from the {\sl lift invariant\/},   Def.~\ref{Anliftdef}.}  \sdisplay{\cite{Fr20}}  relates Frattini covers and the Inverse Galois Problem. This paper describes Serre's original \OIT\ in \S\ref{pre95} in the discussion of \sdisplay{\cite{Se68}} and \sdisplay{\cite{Fr78}}, updated from the original papers with references to \cite{Fr05}, \cite{GMS03} and \cite{Se97b}. Especially, this gives background  on the problems that connected our Hurwitz space approach  to the \OIT. 

\cite{Fr20} is complete on the Universal Frattini cover itself, and especially the role of  the {\sl lift invariant}.  Nontrivial lift invariants arise from what group theorists call representation covers of $G$. They are also a detectible subset of central (in the notation of Def.~\ref{frattdef}, when $\ker(\psi)$  is in the center of $H$) Frattini covers. This set of ideas gives the main information we require to understand the \MT\ levels,  components, their cusps and why they are appropriate for generalizing the OIT.   

\subsection{Universal Frattini cover} \label{lfratquot} We starts by putting a nilpotent tail on any finite group, producing for any $G$, even $G=A_5$ myriad extensions by, say, 2-groups, none of which have ever been realized as Galois groups. Let $\psi_i:H_i\to G$, $i=1,2$, be Frattini covers (Def.~\ref{frattdef}). 

\begin{lem} \label{fiberprod} A minimal (not necessarily unique) subgroup $H\le H_1\times_G H_2$ that  is surjective to $G$, is a Frattini cover of $G$ that factors surjectively to each $H_i$. Thus, Frattini covers of $G$ form a projective system. From their definition, taking a Frattini cover of a group preserves the rank. \end{lem} 

\begin{proof} The projection $\pr_i: H\le H_1\times_G H_2 \to H_i$, makes $\pr_i(H)$ a subgroup of $H_i$ mapping surjectively to $G$. As $\psi_i$ is Frattini, $\pr_i(H)=H_i$, $i=1,2$. \end{proof}

Also, Frattini covers of perfect groups are perfect.  Key for Frattini covers is that $\ker(\psi)$ is nilpotent \cite[Lem.~20.2]{FrJ86}${}_{1}$ or \cite[Lem.~22.1.2]{FrJ86}${}_{2}$. 

Write $\ker(\psi)=\prod_{\ell ||G|} \ker(\psi)_{\ell}$ indicating the product is over its $\ell$-Sylows. For each $\ell$, quotient by all Sylows for primes other than $\ell$ dividing $\ker(\psi)$. Thus form ${}_\ell\psi: {}_\ell H \to G$. 

The fiber product of the ${}_\ell H\,$s over $G$ equals $H$.  \cite[App.~B]{Fr20}  discusses elementary structural statements about the construction of $\tilde G$ below. Some version of these appear in \cite[Chap.~22]{FrJ86}${}_2$. 

\begin{defn} \label{univFrattDef} This produces a profinite cover, the {\sl Universal Frattini cover}, $\tilde \psi_G: \tilde G \to G$. Similarly,   $$\tilde \psi_{G{}_\ab}: \tilde G/ [ \ker( \tilde \psi_{G}), \ker( \tilde \psi_{G})] \eqdef \tilde G_\ab \to G \text{ (resp. }{}_\ell \tilde  \psi_{G_\ab}: \fG \ell{}_\ab \to G)$$ is the Universal  {\sl Abelianized\/} Frattini (resp.~$\ell$-Frattini) cover of $G$. \end{defn} 

Then, $\tilde \psi_G$ is a minimal projective object in the category of profinite groups covering $G$.  So, given any profinite group cover $\psi: H \to G$, some homomorphism $\tilde \psi_{G,H}$ to $H$ factors through $\psi$. If $\psi$ is a Frattini cover, then $\tilde \psi_{G,H}$ must be a Frattini cover, too.

If $\rk(G)=t$, construct $\tilde G$ using a pro-free group, $\tilde F_t$, on the same (finite) number of generators. Sending its generators to generators of $G$ gives a cover, $\tilde F_t\to G$. Then, $\tilde G\le \tilde F_t$ is minimal (closed) among covers of $G$.

This, though is nonconstructive.  It uses the Tychynoff Theorem: a nested sequence of closed subgroups of $\tilde F_t$ covering $G$ has non-empty intersection covering $G$.  That also explains why it wasn't sufficient to replace $\tilde F_t$ by the free (rather than pro-free) group on $t$ generators. 

The following quotients of $\tilde G$ are more accessible.  They result from decomposing the (pro-)nilpotent kernel $\ker(\tilde \psi)$ into a product of its $\ell$-Sylows.

\begin{defn} \label{univFrattdef-ab} For each prime $\ell| |G|$, there is a profinite Frattini cover ${}_\ell \tilde \psi_{G}: {}_\ell \tilde G \to G$ with $\ker({}_\ell\tilde \psi_{G})$ a  profree pro-$\ell$ group of finite rank, $\rk(\fG \ell)$.  There are similar such covers with ${}_\ell\tilde \psi_{G_\ab}$ replacing ${}_\ell\tilde \psi_{G}$. \end{defn} 

The Frattini subgroup of an $\ell$-group $H$ is the (closed) subgroup generated by $\ell$-th powers and commutators from $H$. Denote it by ${}_{\text{\rm fr}}H$. 

Consider the kernel of the short exact sequence  $\ker_0 \to \fG \ell \to G$. Recover a cofinal family of finite quotients of $\fG \ell$ as follows. Mod out by successive $\ell$-Frattini subgroups of  {\sl characteristic kernels\/} of $\fG \ell$: \begin{equation} \label{charlquots} \ker_0> {}_{\text{\rm fr}} \ker_0\eqdef  \ker_1 \ge \dots \ge  {}_{\text{\rm fr}}\ker_{k{-}1} \eqdef \ker_k \dots\end{equation}   Denote  $\fG \ell/\ker_k$ by $\tfG \ell k$, and $\ker_k/\ker_{k'}$ by ${}_\ell M_{k,k'}$ or  $M_{k,k'}$ for $k'\ge k$. 
\begin{equation} \label{charZlGmod} \begin{array}{c} \text{Especially, ${}_\ell M_G\eqdef {}_\ell M_{0,1}$ is the characteristic $\bZ/\ell[G]$ module.} \\ \text{Denote its dimension, $\dim_{\bZ/\ell} ({}_\ell M_G)$ by ${}_\ell m_G$.}\end{array} \end{equation}

Given $(G,\bfC,\ell)$, define the Nielsen classes ${}_\ell\ni(G,\bfC)$ of a \MT\  in a profinite way that extends that of an ordinary Nielsen class (Def.~\ref{NielsenClass}). 

Denote the free group $\pi(U_{\bz_0}, z_0)$ modulo inner automorphisms, by $\sG_{\bz_0}$. Then, consider $\psi_{\bg}: \sG_{\bz_0} \to G$ given by mapping classical generators \eqref{generatorspi}, $\sP$,  to the branch cycles $\bg\in \ni(G,\bfC)$ as given in \S \ref{dragging}. 

Now form all homomorphisms $\psi_{\tilde \bg}: \sG_{\bz_0} \to \fG \ell$ through which $\psi_{\bg}$ factors, indicating images of the classical generators, $\sP$, by $\tilde \bg$, that satisfy this additional condition: 
\begin{equation} \text{$\tilde g_i$ has the same order as $g_i$, $1,\dots, r$.} \end{equation} 
From Schur-Zassenhaus,  as $\fG \ell\to G$ has kernel an $\ell$-group, this defines the conjugacy class of $\tilde g_i$ uniquely. With no loss, also label it $\C_i$, $i=1,\dots,r$. 

This makes sense of writing $\ni(\fG \ell,\bfC)^\dagger$ (or $\ni(\fG \ell {}_\ab,\bfC)^\dagger$) with $\dagger$ any one of the equivalences we have already discussed in \S\ref{equivalences}. As previously $H_r$ acts on the Nielsen classes. To define the Nielsen class levels, mod out successively, as in Def.~\ref{univFrattDef}, on $\fG \ell$ by the characteristic kernels of \eqref{charlquots}. Then, $H_r$ acts compatibly on these canonical towers of Nielsen classes:  

\begin{equation} \label{HurSpSeq} \begin{array}{c} \text{forming a Hurwitz space sequence $\bH(\fG \ell,\bfC)=\{\sH(\tfG \ell k,\bfC)^\dagger\}_{k=0}^\infty$,} \\ \text{  with a natural map from level $k\np1$ to level $k$.} \\ 
\text{Similarly for $\bH(\fG \ell {}_\ab,\bfC)=\{\sH(\tfG \ell k {}_\ab,\bfC)^\dagger\}_{k=0}^\infty$, the Hurwitz} \\ \text{ space tower of the maximal lattice referred to in \S \ref{outlineMTs}.}  \end{array}\end{equation}

\subsection{Using $\fG \ell {}_\ab$ and $\tfG \ell 1$} \label{lfratquot-ab}   We now have a general situation for any $(G,\bfC,\ell)$ in which natural $\ell$-adic representations arise from points on a tower of Hurwitz spaces relating the \RIGP\ and the \OIT. \S \ref{lfratlatqt}  adds the notion of ($\ell$-Frattini) lattice quotients (of the maximal one). \S\ref{nilatqt} produces their associated \MT s, whose levels can be (variants on) classical spaces.  Then, \S \ref{normlgp} notes that $G$ with normal $\ell$-Sylow  gives very small such quotients. 

\subsubsection{$\fG \ell {}_\ab$ lattice quotients} \label{lfratlatqt} 
Assume the standard properties \eqref{abquot-def} for $(G,\ell,\bfC)$: $G$ is $\ell$-perfect and $\bfC$ consists of $\ell'$ conjugacy classes. 

Consider any short exact sequence $ L^\star\to G^\star \mapright{\psi^\star} G$ \begin{equation} \label{quotseq} \text{with $\ker(\psi^\star)=L^\star$ a $\bZ_\ell$ lattice  and $\psi^\star$ an $\ell$-Frattini cover.} \end{equation}  Take $G_k=G^\star/\ell^{k}\ker(\psi^\star)$ for the analog of  \eqref{HurSpSeq}: 
a tower of Hurwitz spaces $\bH(G^\star,\bfC)$ from the Nielsen class sequence $\{\ni(G_k,\bfC)^\inn\}_{k=0}^\infty$. 
\begin{equation} \label{charqtmodules} \begin{array}{c} \text{Since $\psi^\star$ is $\ell$-Frattini,  there is a surjection  $\tfG \ell {}_\ab\mapright{\mu^\star} G^{\star}$} \\ \text{inducing a surjection ${}_\ell M_G\to \ker(G_1\to G_0)\eqdef M^{\star}_\psi=M^\star.$} \end{array}\end{equation}  

\begin{defn} \label{latquotdef} Refer to $\psi^\star$ as a {\sl lattice quotient\/} of ${}_\ell\tilde \psi_{G_\ab}$ -- with {\sl target\/} $M^\star$. Speak of a \MT\ on it or on its corresponding Nielsen class sequence.\footnote{When the context is clear we use shortenings of the name $\ell$-Frattini lattice quotient to $\ell$-lattice quotient or just lattice quotient.} \end{defn} Consider  two lattice quotients of $\tfG \ell {}_\ab\mapright{{}_j\mu^\star} {}_jG^\star  \mapright{{}_j\psi^\star} G$ as above, with targets ${}_j M^\star$, and respective Hurwitz space towers ${}_j\bH, j=1,2$.  

\begin{lem} \label{latquotmaps}  Assume restricting ${}_2\mu^\star$ to $\ker({}_1\psi^\star)$ surjects onto $\ker({}_2\psi^\star)$. If Main \RIGP\  conj. \ref{mainconj} holds for each \MT\ on ${}_2\bH$,  then it does so for each \MT\ on  ${}_1\bH$. \end{lem}

\begin{proof} The assumptions give $\mu_{2,1}: {}_1M^\star \to {}_2 M^\star$, a surjection.  Consider a \MT, ${}_1\bH$,  on  $\{{}_1\sH_k\}_{k=0}^\infty$. Denote the corresponding Nielsen class orbit for ${}_1\sH_k$ by ${}_1O_k\le \ni(G_{k,1},\bfC)^\dagger$. 

Inductively, from $\ker(\mu_{2,1})$, form a $\bZ/\ell[G_{k,1}]$ module $M(k)^\star$ so that  with $G_{k,1}/M(k)^\star=G_{k,2}$ this produces a braid orbit ${}_2O_k\le \ni(G_{k,2},\bfC)^\dagger$ for a \MT\ on ${}_2\bH$. 

These maps are natural and they map $K$ points to $K$ points, for $K$ any number field. By assumption $K$ points disappear at high levels of any \MT\ on  ${}_2\bH$. Therefore they must also on ${}_2\bH$. \end{proof} 

Statements \eqref{charprops} on ${}_\ell M_G$ -- using considerable modular representation theory -- give some sense of tools at our disposal for using Lem.~\ref{latquotmaps}. They are respectively \cite[Chap. 3, Prop.~1.26]{Fr20} (or \cite[Proj.~Indecomp.~Lem.~2.3]{Fr95} with help from \cite{Se88})  and \cite[Chap.~3, Prop.~1.27]{Fr20} (or \cite[Prop.~2.7]{Fr95}). 
\cite{Fr20} collects these tools under four {\sl $\ell$-Frattini principles}. That material also has detailed explanations of the cohomology involved.  

\begin{edesc} \label{charprops}  \item  \label{charpropsa} It is indecomposable  (if not then its summands would be obvious examples of $M^\star$) and $\dim_{\bZ/\ell} (H^2(G, M_G))=1$ (see Lem.~\ref{charqtquest}). 
\item \label{charpropsb} Describing it requires having explicitly only the projective indecomposables belonging to the {\sl principal\/} block  representations.\footnote{$\bZ/\ell[G]$ decomposes as a sum of indecomposable 2-sided ideals (blocks) corresponding to writing 1 as a direct sum of primitive central idempotents. The block \lq\lq containing\rq\rq\ the identity representation is the principle block \cite[\S6.1]{Be91}.}\end{edesc} 

Consider any (non-trivial) $\bZ/\ell[G]$ quotient $M'$ of ${}_\ell M_G$, with kernel $K_{M'}$. Since any quotient of ${}_\ell\tilde \psi_\ab$ mapping through $G$ is a Frattini cover, therefore giving ${}_\ell^1\psi_{M'}: \tfG \ell 1/K_{M'}\eqdef \tfG \ell 1 {}_{M'}\to G$ is an $\ell$-Frattini cover. 
We say it is {\sl unique\/} if  the following $\bZ/\ell$ module has dimension 1 (see Rem.~\ref{uniqlfratt}):
\begin{equation} \label{ext2Lstar} H^2(G,M')=\Ext^2_{\bZ/\ell[G]}(\one, M')\text{ \cite[p.~70]{Be91}}.  \end{equation} 

\begin{lem}  \label{charqtquest} As above, there is a short exact sequence \begin{equation} \label{findLeq}  L_{M'}\to {}_\ell \tilde G_{M',\ab} \longmapright{{}_\ell \tilde \psi_{M'} }{40} G\footnote{Referencing \eqref{findLeq} just by $M'$  is a simplification of notation, since $M'$ can appear, in cases, as a quotient of ${}_\ell M_G$ several ways. Usually this won't be a problem.}\end{equation}  
satisfying \eqref{findLprop} for $k\ge 0$:\footnote{\sdisplay{\cite{CaD08}} gives a good set of these -- appropriate analogs of Ex.~\ref{Dlk+1} and Thm.~\ref{level0MT} -- sufficient to test the conjectures.}
\begin{edesc} \label{findLprop} \item \label{findLpropa} ${}_\ell \tilde \psi_{M'} $ factors through ${}_\ell^1\psi_{M'}$
\item \label{findLpropb} $\ell^k\L_{M'}/\ell^{k\np1}\L_{M'}\cong M'$ as a $\bZ/\ell[G]$ module.\end{edesc}  

If $M'$ is an indecomposable $\bZ/\ell[G]$ module, then $L_{M'}$ is an indecomposable $\bZ_\ell[G]$ module \cite[Thm.~1.9.4]{Be91}. \end{lem}

\begin{proof} For \eqref{findLeq}, inductively form ${}_\ell^k\psi_{M'}: \tfG \ell k {}_{M'}\to G$ as a quotient of ${}_\ell^k\psi_\ab: \tfG \ell k {}_\ab\to G$. 

Use the universal property of ${}^{k\np1}_\ell \psi_\ab: \tfG \ell {k\np1} {}_\ab \to G$ for $\ell$-Frattini covers of $G$ with exponent $\ell^{k\np1}$ kernel: This factors through ${}_\ell^k\psi_{M'}$ as $$\psi'': {}^{k\np1}_\ell \psi_\ab: \tfG \ell {k\np1} {}_\ab  \to \tfG \ell k {}_{M'}.$$  Continue inductively, assuming we have  ${}^k K_{M'}$, the kernel of $\tfG \ell k {}_{\ab} \to  \tfG \ell {k\nm1} {}_{M'}$. Then, on $\tfG \ell {k\np1} {}_\ab$, mod out by $\ell\psi''^{-1}({}^kK_{M'})\eqdef {}^{k\np1}K_{M'}$ to form  $\tfG \ell {k\np1}  {}_{M'}\to G$.  

For the final profinite group cover given by ${}_\ell \tilde \psi_{M'}$, take the projective limit of these group covers of $G$.   \end{proof} 

\begin{rem} \label{uniqlfratt} The uniqueness definition in \eqref{ext2Lstar} extends to consider all related coefficients, $\bZ/\ell^k$ and $\bZ_\ell$.  We think ${}_\ell^1\psi_{M'}$ in Prop.~\ref{latquotmaps} is unique for Frattini covers of $G$ with kernel $M'$ if and only if \eqref{findLeq}  is unique with 
properties  \eql{findLprop}{findLpropa} and \eql{findLprop}{findLpropb}. The only if part is easy, but at this time we don't have a complete proof of the other direction.  \end{rem} 
\subsubsection{Nielsen classes of $\ell$-Frattini lattices} \label{nilatqt} Apply the formation of Nielsen classes, as in Thm.~\ref{BCYCs} with the assumptions of \S\ref{outlineMTs}, now though to a profinite version, $\ni(\fG \ell {}_\ab,\bfC)$. Then $H_r$ extends to this Nielsen class. 

A \MT\  for $(G,\bfC,\ell)$ consists of a profinite $H_r$ orbit. Denote the collection of these by $\sF_{G,\bfC,\ell}\eqdef \sF_{G,\bfC,\ell,{}_\ell\tilde \psi_\ab}$. The same constructions works for  \begin{equation} \label{proplatticequotient} \begin{array}{c} \text{ any  proper $\ell$-Frattini lattice quotient $L^\star \to G^\star \longmapright {{}_\ell\tilde \psi^\star} {30}G$, with} \\ \text{ the corresponding $^kG^\star$ from modding out on $G^*$ by $\ell^k L^\star$.}\end{array}\end{equation} 

\S\ref{prodladic} explains how $\ell$-adic representations appear from these \MT s.   For many proper lattice quotients, elements of $\sF_{G,\bfC,\ell,{}_\ell\tilde \psi^\star}$ are variants on classical spaces as started in Ex.~\ref{Dlk+1}, \S\ref{normlgp} and expanded on in \sdisplay{\cite{CaD08}}. 

\subsubsection{Normal $\ell$-Sylow and other cases} \label{normlgp}  We get a fairly small $\ell$-Frattini lattice quotient  with $G=N\xs H$, a normal $\ell$-Sylow, $N$ ($(|H|,\ell)=1$). Then, the construction above  for ${}_\ell M_G$ feels much less abstract.

That is because $\fG \ell$ is just $\tilde N\xs H$ with $\tilde N$ the pro-$\ell$, pro-free group on the same number of generators as has $N$. The point is the action of $H$ extends to $\tilde N$ \cite[Prop.~22.12.2]{FrJ86}${}_2$.\footnote{That makes it sound explicit, but extension of that action is abstract.}

Denote the Frattini subgroup of $N$ by ${}_{\text{\rm fr}}N$ (as in \eqref{charlquots}) and its quotient $N/{}_{\text{\rm fr}}N$ by $N'=(\bZ/\ell)^t$ ($t=\text{\rm rank}(N)$). Then, $N\to N'$ extends to $N\xs H\to N'\xs H$, and then to $\mu': {}_\ell \tilde N\xs H \to N'\xs H$. 

Both Lem.~\ref{Zlquot} (which follows from the above) and Ex.~\ref{lcentext} apply to the \S\ref{05-to-now} main example.  

\begin{lem} \label{Zlquot} As  ${}_\ell \tilde N\xs H$ is $\ell$-projective, $\mu'$ extends to  $$\text{ $\mu:(\bZ/\ell^2)^t\xs H\to (\bZ/\ell)^t\xs H$, a Frattini cover.}$$ 

Conclude:  ${}_\ell M_{G}$ has $\ker(\mu)=(\bZ/\ell)^t\eqdef M'$ as a quotient. So, from Lem.~\ref{charqtquest}, ${}_\ell\tilde \psi_\ab: \tfG \ell {}_\ab \to G$ has \eqref{findLeq} lattice quotient with $L_{M'}=(\bZ_\ell)^t$. \end{lem} 

Ex.~\ref{lcentext} -- giving what appears to be a very simple $\ell$-Frattini lattice quotient -- shows we can isolate out the role of the Schur multiplier of $G$, as giving a lift invariant. A list of its appearances is in \eqref{centfratuse}. Explicit examples of this are in Prop.~\ref{A43-2} and in Thm.~\ref{level0MT} in \S\ref{05-to-now}. These appearances extend all the way up the levels of that \MT\ (not in this paper).   

\begin{exmpl}[Schur multipliers] \label{lcentext} Again $G$ is $\ell$-perfect. Consider when $\psi_H: H\to G$ is a {\sl central\/} $\ell$-Frattini cover. Then, $\ker(\psi_H)$ is a quotient of ${}_\ell M_G$. The $\ell$-Frattini quotient lattice in \eqref{charqtquest} then has lattice kernel in the center of $ {}_\ell \tilde G_{M',\ab}$. These \MT s based on $\ell'$ classes $\bfC$ have two features. 

The components at higher levels (if they pass the obstruction test of Thm.~\ref{obstabMT}) would have the same underlying spaces as at level 0. As in the discussion, however,  following Def.~\ref{moddeffield}, the {\sl moduli definition field test\/} for realization of representing covers at that level would not directly work. These would definitely not be fine moduli spaces. \end{exmpl} 

\begin{rem}[\MT\ levels and being centerless] \label{nocentMT} If $G$ is centerless and $\ell$-perfect, then $\tfG \ell k$ and  $\tfG \ell k {}_\ab$ are also \cite[Prop.~3.21]{BFr02}. The significance with inner Hurwitz spaces:  this is the criterion for fine moduli. Ex.~\ref{lcentext} shows this may not hold for $\ell$-Frattini lattice quotients attached to $G$. \end{rem} 

\subsection{Test for a nonempty \MT} \label{nonemptyMT} Again, $G$ is $\ell$-perfect and we have $\ell'$ conjugacy classes $\bfC$.   \S\ref{obstcomps} gives the main criterion for non-empty \MT s. Notice, a' la Rem.~\ref{nocentMT} that, even if $G$ is centerless, Thm.~\ref{obstabMT} deals with a group that has a center.

Subsections  \S\ref{Ancomps} and \S\ref{exA43-2cont}  give examples of obstructed braid orbits (Hurwitz space components). They also give examples of computing for reduced \MT\ levels for $r=4$ with genuses $> 1$.  

\subsubsection{The obstructed component criterion} \label{obstcomps} This section shows just how important is modular representations in understanding geometric properties of \MT\ levels. \cite{Fr06} (and \cite{Fr20}) show how this works for cusp analysis. Denote the maximal central $\ell$-Frattini extension by ${}_\ell \alpha:{}_\ell G^\natural \to G$.\footnote{This is a cover with finite kernel.} 

\begin{lem} If each level of $\bH(\fG \ell {}_\ab,\bfC)$ is nonempty, then there is at least one $\MT$ on  $\bH(\fG \ell {}_\ab,\bfC)$.  \end{lem} 

\begin{proof} Producing such a \MT\ is equivalent to forming  
\begin{equation} \label{exMTbg} \begin{array}{c} \text{branch cycles $\bG\eqdef \{{}_k\bg\in \ni(\tfG \ell k {}_\ab,\bfC)\}_{k=0}^\infty$ such that} \\ \text{the natural map $\tfG \ell {k\np 1} {}_\ab \to \tfG \ell {k} {}_\ab$ maps ${}_{k\np1}\bg $ to ${}_{k}\bg $, $k\ge 0$.} \end{array}\end{equation}  Use the axiom of choice (Tychonoff Theorem) as in \S\ref{toughestpt}. Some nonempty chain hits each \MT\  level. \end{proof} 

\begin{defn}  Taking the braid orbits of the elements in $\bG$ gives a \MT\ through the braid orbit of ${}_0\bg$.  If  ${}_0\bg \in O$ we say this is a \MT\ through $O$. If there is no \MT\ through $O$, then $O$ is {\sl obstructed}. \end{defn} 

\begin{thm} \label{obstabMT} There  is a \MT\ on $\bH(\fG \ell {}_\ab,\bfC)$ through $O$ if and only if the natural map $\ni({}_\ell G^\natural,\bfC) \to \ni(G,\bfC)$ has image ${}_0\bg$.\footnote{Computing Schur multipliers is hard. So,  a criterion for a component, assuring without computation, that the lift invariant is always trivial is valuable. The generalization, Def.~\ref{gl'thm},  of Harbater-Mumford components, Def.~\ref{HMrep}, gives exactly that.} 

Now replace $\fG \ell {}_\ab \to G$ by any $\ell$-Frattini lattice $\psi^\star: G^\star\to G$ and ${}_\ell G^\natural \to G$ by the maximal central $\ell$-Frattini extension of $G$ that is a quotient $G^*$. Then, the same statement guarantees a \MT\ on $\psi^\star $ through $O$. \end{thm} 

\begin{proof} Assume an $\ell$-Frattini cover $H''\to G$ that factors through $H'\to G$, with $M'=\ker(H''\to H')$ an irreducible $\bZ/\ell[H']$ module.  \begin{equation} \label{appone} \begin{array}{c} \text{\cite[Obst.~Lem.~3.2]{FrK97}: then $\ni(H'',\bfC)\to \ni(H',\bfC)$ } \\ \text{surjects on all braid orbits, unless  $M'=\one_{H'}$. In that case} \\ \text{it surjects only on braid orbits with lift invariant 1.}\end{array}\end{equation}   

\cite[Lem. 4.9]{Fr06} shows \eqref{appone} holds with $G_{k\np1}$ (resp.~$G_k$) replacing $H''$ (resp.~$H'$) and with the maximal quotient of $\ker( G_{k\np1}\to G_k)$  on which $G_k$ acts trivially replacing $G^\natural$. Then, \cite[Lem.~4.14]{Fr06} shows this holds for the sequence of Hurwitz spaces  defined by $\fG \ell {}_\ab \to G \to G_0$ using $G^\natural$.

The last paragraph uses $\ell$-Poincar\'e duality, a' la \cite[I.4.5]{Se97a}. Except, instead of a pro-$\ell$ group, $\fG \ell {}_\ab \to G$  has $G$ at its head \cite[Prop.~3.2]{We05}. We have only to adjust to the appropriate test on the central extension when applied to the $\ell$-Frattini lattice quotient.  
\end{proof} 

\subsubsection{Obstructed $A_n$ components} \label{Ancomps}  Use the {\sl spin cover}, $\alpha_n: \Spin_n\to A_n$, $n\ge 4$ and its lift invariant  (\S\ref{l=2level0MT}). Denote the genus of covers in $\ni(A_n,\bfC_{3^r})$ -- adjusting for the two conjugacy classes $\C_{\pm}$ for $n=4$ -- by $\textbf{g}_{n,r}$. 
\begin{edesc} \label{An} \item \label{Ana}  $\textbf{g}_{n,n\nm1}=0$: There is one Hurwitz space component with  lift invariant $(-1)^{n\nm1}$ starting with $-1$ at $n=4$. 
\item \label{Anb}  $\textbf{g}_{n,r}>0$: The two Hurwitz space components, $\sH_{n,r}^{\pm}$, are separated by their lift invariants. 
\item \label{Anc}  In $\sH(A_n,\bfC_{3^{r}})^\inn \to \sH(A_n,\bfC_{3^{r}})^\abs$, all $r\ge n\nm1$, 
each image component has only one preimage. \end{edesc} 

Here is the rephrasing of this in obstructed components for the full $\ell$-Frattini lattice (quotient) $\fG \ell {}_\ab\to G$. The unique component in \eql{An}{Ana} is obstructed if and only if $n$ is even. The component $\sH_{n,r}^+$ (resp.~$\sH_{n,r}^-$) in \eql{An}{Anb} is unobstructed (resp.~obstructed).   

\S \ref{exA43-2cont} finishes two issues with Ex.~\ref{exA43-2}. Here we consider both the {\sl full\/} $\ell$-Frattini lattice $\fG  2 {}_\ab\to G$, and the {\sl minimal\/} $\ell$-Frattini lattice (\S\ref{normlgp}): \begin{equation} \label{minlatqt} \text{$(\bZ_2)^2\xs \bZ/3\to (\bZ/2)^2\xs \bZ/3$, the case $\ell=2$ of Thm.~\ref{level0MT}.}\end{equation}   
For each lattice quotient $L^\star$, we want a \MT\ level value ($k_{L^\star}$), so any reduced Hurwitz space component at that level has genus $> 1$.\footnote{Putting aside the problem of an explicit Falting's result: At which tower level the finitely many $K$ (number field) points actually disappear?}

\subsubsection{Finishing Ex.~\ref{exA43-2}} \label{exA43-2cont} 
The Main \MT\ Conj.~\ref{mainconj} \cite[\S5]{Fr06} has the crucial hypothesis for giving a lower bound on the genus of high levels of a \MT. It is the existence of $\ell$-cusps -- basically that $\ell$ divides the cusp widths -- on each  component \eqref{cusptype}. 

Here $\ell=2$: Each component has 2-cusps:  respectively $O^4_{1,1}$ and $O^4_{1,4}$. 
Both inner reduced components have genus 0 (Prop.~\ref{A43-2}). So, we need the argument of \sdisplay{\cite{Fr06}} to conclude Main \RIGP\ Conj.~\ref{mainconj}  -- as in discussing Falting's in \S\ref{toughestpt} -- for each lattice $L^\star$ to get the (higher) level, $k_{L^\star}$,  at which component genuses rise beyond 1. 

\cite[Thm.~9.1]{BFr02} did exactly that for the full lattice  of $(A_5,\bfC_{3^4},\ell=2)$, listing the genus's at level 1 as 9 and 12. The characteristic module for both $A_4$ and $A_5$ for $\ell=2$ is a copy of $(\bZ/2)^5$; action of $A_5$ giving that of $A_4$ by restriction to the subgroup \cite[Prop.~2.4]{Fr95}. 

The rest of \cite[\S9.1]{BFr02} outlines this for the full lattice for $A_4$. \cite{FrH20} does this for the complete series of groups in Thm.~\ref{level0MT}, including $A_4$, and the minimal lattice. Here, as often, the smaller lattice quotient gives more components, and trickier issues in bounding the genus. 

Now consider the lattice in Lem.~\ref{Zlquot} and obstruction for the two \\ $(A_4,\bfC_{\pm 3^2},\ell=2)$ level 0 components. The lift invariant separated the two components. Yet, neither is obstructed:  $(\bZ_2)^2\xs \bZ/3\to (\bZ/2)^2\xs \bZ/3$ does not factor through $\Spin_4\to A_4$.

\subsection{Archetype of the $\ell$-Frattini  conjectures} \label{archfrattconj}  This section reminds of many reasons for using  proper $\ell$-Frattini lattice quotients (\S\ref{lfratlatqt}) of $\fG \ell {}_\ab \to G$: Easier group theory; reflection on the main (\RIGP\ and \OIT) conjectures; and the classical connections. 

Many of our examples are the \S\ref{normlgp} type. Here though we give the example that arises in considering Serre's \OIT.  Simultaneously, it inspired the most significant Def.~\ref{evenfratt}. It also gives proper $\ell$-Frattini lattice quotients that aren't from \S\ref{normlgp}. 

\begin{defn} \label{evenfratt} Call a sequence of finite group covers $$\dots \to H_{k\np1}\to H_k\to \dots \to H_1 \to H_0=G$$ {\sl eventually Frattini\/} (resp.~eventually $\ell$-Frattini) if there is a $k_0$ for which $H_{k_0\np k} \to H_{k_0}$ is a Frattini (resp.~$\ell$-Frattini) cover for $k\ge0$. \end{defn} 

If the projective limit of the $H_k\,$s is $\tilde H$, then we say it is eventually Frattini since the same holds for any cofinal sequence of quotients. Note: any {\sl open\/} subgroup of $\tilde H$ will also be eventually Frattini (resp.~$\ell$-Frattini). 

Take $G=\PSL_2(\bZ/\ell)$, ${}_\ell M_G$ its characteristic $\ell$-Frattini module \eqref{charZlGmod}. Here is a relevant commutative diagram for this, for Serre's setup, with $\Ad_3$ indicating the $2\times 2$ trace 0 matrices. 

\begin{equation} \label{commsl2diag} \begin{tikzcd}
\Ad_3(\bZ_\ell)\arrow[d, "{\rm Id}"'] \arrow[rr, " "]  & &\SL_2(\bZ_\ell)  \arrow[d, "{}_\ell \tilde \alpha "'] \arrow[rr, "{}_\ell \tilde \psi"]  &  & \SL_2(\bZ/\ell)\arrow[d, "{}_\ell \alpha"']\\
\Ad_3(\bZ_\ell)  \arrow[rr, " "]   & &\PSL_2(\bZ_\ell)  \arrow[rr, "{}_\ell \tilde \psi"]  &   &  \PSL_2(\bZ/\ell) \end{tikzcd} \end{equation}

The cover ${}_\ell\alpha$ is a Frattini extension, of degree 2 (say, \cite[Chap.~6, Lem.~3.1]{Fr20}). Though not for $\ell\ge 3$ an $\ell$-Frattini extension, it connects the upper and lower rows of  \eqref{commsl2diag} on their respective Frattini conclusions. 

\begin{prop} \label{PSLFrat} The natural cover $\SL_2(\bZ/\ell^{k\np 1})\to \SL_2(\bZ/\ell)$  is an $\ell$-Frattini cover for all $k$ if $\ell> 3$. For $\ell=3$ (resp.~2), $$\begin{array}{c}\text{$\SL_2(\bZ/\ell^{k+1})\to \SL_2(\bZ/\ell^{k_0\np1})$, $k\ge k_0$ where $k_0=1$ (resp.~2), } \\ \text{is the minimal value for which these are Frattini covers.}\\ \text{For all $\ell$,  $\PSL_2(\bZ_\ell)\to \PSL_2(\bZ/\ell)$ is eventually $\ell$-Frattini.} \end{array}$$

For $\ell>3$, $\Ad_3(\bZ/\ell)$ is a quotient of ${}_\ell M_{\PSL_2(\bZ/\ell)}$, but it is not for $\ell=2$ or 3. Further, for no $\ell$ is it the whole module. Therefore, for $\ell>3$, the lower row of \eqref{commsl2diag}  is a proper $\ell$-Frattini lattice quotient of $\fG \ell {}_\ab$. 
\end{prop} 

\begin{proof} The first sentence is \cite[Cor. 22.13.4]{FrJ86}${}_2$,  from \cite[Lem.~3, IV-23]{Se68} which has exercises  that the same statement and proof applies to $\SL_d(\bZ/\ell)$.  Here we do those exercises for just $\ell=3$. 

First, we trim the treatment of \cite[Cor. 22.13.4]{FrJ86}${}_2$. Use $A=\smatrix 1 1 0 1$ and $B=\smatrix 1 0 1 1$. Following \cite[p.~532]{FrJ86}${}_2$, add $C=\smatrix 1 {-1} 1 {-1}$ to give three independent generators of $\Ad_3(\bZ/\ell)$, all with square 0:  every $u\in \Ad_3(\bZ/\ell)$ is a sum of square 0 elements. 

Our induction hypothesis is $H\le \SL_2(\bZ/\ell^{k\np1}) \to \SL_2(\bZ/\ell^k)$ maps surjectively. We need to show,  for  $u\in \Ad_3(\bZ/\ell)$, there is $h\in H$ of form $1 \np \ell^k u$. For this, the induction gives $h_0\in H$ and $v\in \Ad_3(\bZ/\ell)$ with $h_0=1\np u\ell^{k\nm1} \np v\ell^k$ for some $v\in \Ad_3(\bZ/\ell)$. Here are the remaining steps. 
\begin{edesc} \label{SL2induct} \item Binomially expand  $h=(h_0)^\ell$ to see it is $1\np u\ell^k \mod \ell^{k\np1}$ unless $k=1$ when it is $1\np u\ell\np \ell(u\np v\ell)^2(\bullet) \np (u\np v\ell )^\ell$ with $\bullet\in \bZ$.  
\item \label{SL2inductb}  If $u^2=0$ (and $k=1$), the result is $1 \np \ell u\mod \ell^2$, if $(u\np v\ell)^\ell\equiv 0\mod \ell^2$. For $\ell>3$ this is numerically easy to see. 
\item \label{SL2inductc}  Write $u\in \Ad_3(\bZ/\ell)$ as a sum of squares $u=\sum_{i=1}^t u_i$. From \eql{SL2induct}{SL2inductb} with $h_i=1\np u_i\ell\in H$, then $\prod_{i=1}^t h_i=1\np u\ell \mod \ell^2$. \end{edesc}  
Then, \eql{SL2induct}{SL2inductc} concludes the induction argument, for the first sentence. 

Now we do the conclusion for $\ell=3$.  \cite[IV-28, Exer.~3]{Se68} asks to show that $\SL_2(\bZ/3^2)\to \SL_2(\bZ/3)$ is not Frattini. We say it purely cohomomologically. 
Then, $\mu\in H^2(\SL_2(\bZ/3), \Ad_3(\bZ/3))$ defines the cohomology class of this extension \cite[p. 241]{Nor62}. 

For any cohomology group, $H^*(G,M)$, with $M$ a $\bZ/\ell[G]$ module, restrict to an $\ell$-Sylow $P_\ell\le G$. 
This gives an isomorphism onto the $G$ invariant elements of $H^*(P_\ell,M)$  \cite[III.~Prop.~10.4]{Br82}. So,  $\mu$ splits if $\mu_\ell$ splits. 

There is an element, $g_3$, of order 3 in $\SL_2(\bZ)$ --  $\PSL_2(\bZ)$ is well-known to be freely generated by an element of order 3 and an element of order 2 -- and so in $\SL_2(\bZ/3^2)$. 
This element of order 3 -- given, say, by $A'=\smatrix {1}{-3}{1}{-2}$ -- generates a 3-Sylow in $\SL_2(\bZ/3)$. Conclude that  $\mu_3$ splits. 

Denote the conjugacy class of $A'$ by  $\C_3$. Its characteristic polynomial is $x^2+x+1$. Any lift of any non-trivial element in $$\ker(\SL_2(\bZ/3^2)\to \SL_2(\bZ/3))=\Ad_3(\bZ/3)$$ is an element of order $3^2$ in this Frattini cover $$\ker(\SL_2(\bZ/3^3)\to \SL_2(\bZ/3))=(\bZ/3^2)^3 \to (\bZ/3)^3.$$  So, the extension $\SL_2(\bZ/3^3)\to \SL_2(\bZ/3^2)$ certainly does not split. 

\cite[IV-28, Exer.~1.b]{Se68} states that $\SL_2(\bZ/3^{k\np 2})\to \SL_2(\bZ/3^2)$, $k\ge 0$, is Frattini.   Take  $v\in \ker(\SL_2(\bZ/3^2)\to \SL_2(\bZ/3))$. 

For $\tilde v\in \ker(\SL_2(\bZ/3^3)\to \SL_2(\bZ/3))$ lifting $v$,  $\tilde v^3$ identifies with $v$, but in $\ker(\SL_2(\bZ/3^3)\to \SL_2(\bZ/3^2))$. From that stage, any subgroup mapping onto $\SL_2(\bZ/3)$ has the kernel in it. 

We are done except for showing $\Ad_3(\bZ/\ell)$ is not the whole characteristic $\ell$-Frattini module: Rem.~\ref{Ad3}. \end{proof}

\begin{rem} \label{Ad3} Take $G=\SL_2(\bZ/5)$. \cite[\S II.F]{Fr95} shows ${}_5 M_{\SL_5(\bZ/5)}$ has  Loewy display $M=\Ad_3(\bZ/5) \to \Ad_3(\bZ/5)$. (Applying Shapiro's Lemma, inducing the identify from the action of a $D_5$ in $A_5$ on $\one$,  $\dim_{\bZ/\ell}(H^2(G,M))=1$  (see Lem.~\ref{charqtquest}). So ${}_5 M_{\SL_5(\bZ/5)}$ is not $\Ad_3(\bZ/5)$. This argument can work for all $\ell\ge 5$ \cite[Chap.~6 \S1.6]{Fr20} based on \cite[\S 2.2.2]{Fr02b}. 

\cite[\S 22.14]{FrJ86}$_2$ notes diagram \eqref{commsl2diag} isn't the universal $\ell$-Frattini cover of $\PSL_2(\bZ/\ell)$. That, here, is beside the point: an $\ell$-Frattini lattice quotient can only be the whole universal $\ell$-Frattini cover when ${}_\ell M_G$ has dimension 1, since the lattice and all its finite index subgroups have bounded rank. The $\ell$-Sylow of the universal $\ell$-Frattini cover is a pro-free pro-$\ell$ group. From Schreier's Theorem \cite[Prop.~17.5.7]{FrJ86}$_2$, if its rank exceeds 1, the rank of its open subgroups grows with their indices. \end{rem}

\section{Monodromy and $\ell$-adic representations}  \label{l-adicreps}  \S\ref{prodladic} gives notation for how $\ell$-Frattini lattice quotients give the $\ell$-adic representations of the title. 
 \S\ref{MTsOIT} formulates the \MT\ main \OIT\ conjectures based on Def.~\ref{evenfratt} -- an eventually $\ell$-Frattini sequence. (In Prop.~\ref{PSLFrat}, $\ell=3$ is an example from Serre's case of the \OIT.)  

\S\ref{uselfratt} shows how we use eventually $\ell$-Frattini, including its Hilbert's Irreducibility aspects. 
In \S\ref{RIGPFaltings}, the \RIGP\ and the \OIT\ problems interact on a \MT. That includes discussing Serre's most difficult case: a model for how Falting's Theorem enters.

\subsection{$\ell$-adic representations} \label{prodladic} As usual,  $K$ is a  number field. Also, as usual, assume $\bfC$ is a rational union of conjugacy classes. From the \BCL\  Thm.~\ref{bcl}, the moduli definition field of the Hurwitz spaces -- though maybe not \MT\ levels  on them -- is $\bQ$.   

\subsubsection{Geometric monodromy of a \MT} \label{notMT} Consider the Nielsen classes referenced by an $\ell$-Frattini lattice quotient, $L^\star \to G^\star \mapright{\tilde \psi^*} G$ of $\fG \ell {}_\ab \to G$ in  \S \ref{lfratlatqt}. We denoted the collection of $H_r$ orbits -- the actual \MT s -- by $\sF_{G,\bfC,\ell,{}_\ell\tilde \psi^\star}$. 
 Indicate one of these, say $O$, as a sequence of $H_r$ orbits $\{O_k\}_{k=0}^\infty$ naturally mapped from level $k\np1$ to level $k$ analogous to \eqref{exMTbg}. 

From  this (\S\ref{equivalences} or \eqref{reducedcover}) comes a canonical Hurwitz space tower,  \begin{equation} \label{GClMtower} \text{$\{\sH(G_k,\bfC)\}_{k=0}^\infty$;  the $(G,\bfC, \ell, L^\star)$ tower.}\end{equation}  

Denote  the collection of projective sequences of points $$\text{$\bar \bp'=\{\bp_k'\in \sH(G_k,\bfC)^\dagger\}_{k=0}^\infty$ lying above $\bp_0'$ on the tower by  $\bS_{\bp_0'}$.}\footnote{For the \RIGP\  we often take $\dagger$ equivalence to be inner, reduced. Here though in comparing with the \OIT\ we need absolute equivalence, too.} $$ 
Abusing notation slightly, also denote the similar sequences of points above $J'\in J_r$ by $\bS_{J'}$. 
\begin{edesc} \label{qtresults} \item \label{qtresultsa} That the extension $ L^\star\to G^\star \mapright{\psi^\star} G$ is $\ell$-Frattini is precisely why the  construction of this Nielsen class tower is canonical. 
\item \label{qtresultsb} A $\bar\bp'\in \bS_{\bp_0'}$  represents a surjection in $\Hom_{\bZ_\ell[G]}(H_1(\hat W_{\bp'_0},\bZ_\ell),L^\star))$. 
\end{edesc} 

Ex.~\ref{Dlk+1} is a particular example of this situation. All the spaces above have canonical polarizations, and therefore quasi-projective structures. That implies if, $\bp_0\in  \sH(G_0,\bfC)(K)$ (resp.~$J'\in J_r(K)$), then $G_K$ acts on $\bS_{\bp_0'}$ (resp.~$\bS_{J'}$). It also acts on the \MT s on the Hurwitz space tower. 

Generalizing Ex.~\ref{Dlk+1} starts with \eqref{monactMT}  and completes in \S\ref{limgptie}.  Use generic notation, $W_{\bp}\to \prP^1_z$, for a cover represented by $\bp\in \sH(G_0,\bfC)^\dagger$. 

\begin{edesc} \label{monactMT} \item  \label{monactMTa} Running over $\bp\in \sH(G_0,\bfC)^\dagger$, create a total family of varieties with fibers $\Pic^u(W_{\bp})$, $u=1,2$, over  $\sH(G_0,\bfC)^\dagger$. 
\item  \label{monactMTb} Interpret elements of \eql{qtresults}{qtresultsb} as subspaces, with an appropriate $G$ action, on the $\bZ_\ell[G]$ Tate module, $T_{\bp}$,  of $\Pic(W_{\bp})$.\footnote{We are sloughing off capturing that $G^\star$ is a Frattini extension of $G$.} 
\item  \label{monactMTc}  Identify $H_r$ action on Nielsen classes, and their orbits as \MT s.
\end{edesc} 

\subsubsection{Using $G\mapsto G^\lm$} \label{limgptie}  \S\ref{MTsOIT} formulates the Main \OIT\ \MT\ conjectures starting from \eql{qtresults}{qtresultsb}. As in Thm.~\ref{level0MT}, it starts with the display from the $\sh$-incidence matrix (as in \S\ref{shincex}). It relies on identifying geometric monodromy groups attached to the levels of a \MT. 

Suppose for each \MT\ we are given $\bp'$, a base point on a component of $\sH(G_0,\bfC)^\dagger$.  How can we \begin{equation} \label{monactMTd} \text{ translate \eql{monactMT}{monactMTc} into $H_r$ acting on each Tate module, $T_{\bp'}$?} \end{equation}

The point of the Nielsen class $(G^\lm,\bfC)$ from $(G,\bfC)$ is to produce a recognizable group within which we can see the Hurwitz monodromy action, and thereby label the geometric (and arithmetic) monodromy groups attached to a specific \MT. Even in recognizing this in Serre's case, the process is illuminating. What makes it work is using the relation between inner and absolute classes given in \eqref{inabsseq} in \S\ref{earlyOIT-MT}.\footnote{This is counterintuitive: The \RIGP\ ends up about inner equivalence. How can it be that the \OIT\ benefits from absolute equivalence?} 

A forerunner application in \eqref{Gconds} \cite{FrV92} uses this Inner-Absolute sequence  for a presentation of $G_\bQ$. That goes from $(G,\bfC)$ to a new $(G^*,\bfC^*)$ using its absolute (rather than inner) Hurwitz space. The result: $N_{S_N}(G^*,\bfC^*)$ -- as in \eqref{eqname} -- contains every outer automorphism of $G^*$. 

Changing the Nielsen class to that of the limit group $G^\lm$, moves all the $H_r$ action into outer automorphisms of $G^\lm$ as in \eqref{sl2H4}. Substituting $G\mapsto G^\lm$,  keeps $\bfC$ the same, using that $G^\lm$ covers $G$. Then, canonically lift the classes $\bfC$ to $G^\lm$, and relate corresponding absolute and inner Nielsen classes. 

Then, in applications  the image of the $H_r$ action interprets inside the symplectic group given as automorphisms of $H_1(W_\bp,\bC)$, with its canonical pairing.  In this process, the underlying inner (reduced) Hurwitz spaces in a \MT\ all come out to be $J_r$. So, we can consider \MT\ fibers over  $J'\in J_r(\bar \bQ)$, just as he did on projective systems of order $\ell^{k\np1}$ on the elliptic curve with $j$-invariant $j'\in \prP^1_j\setminus \{\infty\}$.  

\begin{exmpl} \label{ncmodcurves-cont} Ex.~\ref{ncmodcurves} included an example  for the modular curves denoted  $X_1$ and $X_0$ as a special case of of the mapping from inner to absolute Hurwitz  spaces as  in \eqref{inabsseq}.  \cite[Chap.~6, \S 3.3.1]{Fr20} shows how it works with limit group $G^\lm_{k\np1}=(\bZ/\ell^{k\np1})^2\xs \bZ/2$, corresponding to $\bfC=\C_{2^4}$. 
 
\begin{edesc} \label{sl2H4} \item \label{sl2H4a} We construct $\ni(G^\lm_{\ell^{k\np1}}, \bfC_{2^4})$ from $\ni(D_{\ell^{k\np1}},\bfC_{2^4})$, with the $H_4$ action easily read off from this. 
\item \label{sl2H4b} Apply $\sh,q_2\in H_4$ directly to  $\ni(G^\lm_{\ell^{k\np1}}, \bfC_{2^4})$. They  interpret on  $\SL_2(\bZ/\ell^{k\np1})$ from the action on generators of $\SL_2(\bZ_\ell)$.
\end{edesc} 
From this, the geometric monodromy of $$\bar \sH(G^\lm_{\ell^{k\np1}}, \bfC_{2^4})^{\inn,\rd}\to \prP^1_j=\bar \sH(G^\lm_{\ell^{k\np1}}, \bfC_{2^4})^{\abs,\rd}\text{ is }\SL_2(\bZ/\ell^{k\np1})/\{\pm 1\}.$$ 
That the arithmetic monodromy group is $\GL_2(\bZ/\ell^{k\np1})/\{\pm 1\}$ interprets  using the {\sl Weil pairing\/} along the Hurwitz space giving $$H^1(W_\bp, \bZ_\ell)\times H^1(W_\bp, \bZ_\ell)\to H^2(W_\bp, \bZ_\ell)\text{ coming from roots of 1}.$$ This is the extension of constants of \sdisplay{\cite{Fr78}}. 
\MT s gets it from the Heisenberg lift invariant as in Thm.~\ref{level0MT}  \cite[Chap.~6, Prop.~3.7]{Fr20}. 
\end{exmpl}

\subsection{\MT s and the \OIT\ generalization}  \label{MTsOIT}  
Let us backtrack to \eql{monactMT}{monactMTc}. This amounts to finding projective sequences, $O$, of braid orbits   
\begin{equation} \label{projbraidorbits} \begin{array}{c}  \text{$\{\ni_k\le \ni(G_k,\bfC)\}_{k=0}^\infty$, $k\ge 0$: }\\ \text{$\psi_{k\np1,k}: \ni_{k\np1}\mapsto \ni_k \mod \ker(G_{k\np1}\to G_k)$, $k\le 1$.}\end{array}\footnote{Ex.~\ref{ncmodcurves} has at all levels just one braid orbit, while other examples don't. So, $O$ is one of many possible \MT s.} \end{equation} Restricting the braid action (\S\ref{equivalences}), this gives a projective sequence of finite morphisms of reduced and absolutely irreducible components: 

\begin{equation} \label{projhurworbits} \cdots \to \sH_{k\np1} \longmapright {\Psi_k\np1,k} {40} \sH_{k} \to \cdots \to \sH_1 \longmapright {\Psi_1,0} {30}  \sH_0 \longmapright {\Psi_0} {30} J_r.\end{equation} 

Ex.~\ref{mainex} and Thm.~\ref{level0MT} show the main computational techniques for distinguishing braid orbits: 
\begin{equation} \begin{array}{c} \text{the {\sl lift invariant\/} and the {\sl shift-incidence matrix}. Unless }  \\
\text{there at least 2 \HM\ components, these have distinguished } \\ \text{all components with their (moduli) definition fields.} \end{array} \footnote{\S\ref{95-04} has discussion of \HM-components, Def.~\ref{HMrep}, and their generalization. These  components are transparent to obstruction (\S\ref{obstcomps}), but pose problems if there is more than one at a given level.} \end{equation}  

Again, using the \BCL\ criterion for $(G,\bfC)$ we assume $G_\bQ$ acts on the set of \MT s. Consider the $G_\bQ$ orbit ${}_\bQ O$ of $O$, with the level $k$ component orbit denoted ${}_\bQ\sH_k$. For each $k$, (as in \S\ref{alggeom}) consider the respective geometric and arithmetic monodromy groups of $$\text{$\psi_k: \sH_k\to J_r$ and of ${}_\bQ\psi_k: {}_\bQ\sH_k\to J_r$.}$$ Denote the respective projective limits of these sequences  by $\sG_O$ and ${}_\bQ\hat \sG_O$. 

Use the notation, $\bar  \bS_{\bp'}$ (resp.~$\bS_{J'}$) for projective sequences of points, on $O$, over $\bp_0'$ (resp.~$J'$)  from \S\ref{notMT}
Consider the {\sl decomposition group\/}, $\sG_{\bar \bp_{J'}}$, of $\bar\bp_{J'}\in \bS_{J'} $: the projective limit of the groups of the Galois closures of $\bQ_{\sH_k}(\bp'_k)/\bQ_{\sH_k}$. For all $\bar\bp_{J'}$ these groups are conjugate inside ${}_\bQ\hat \sG_O$. Denote their isomorphism class by the symbol $\sG( \bS_{J'})$. 

Recall Def.~\ref{evenfratt} for an eventually $\ell$-Frattini sequence. 

\begin{guess}[Main \OIT\ Conj.] \label{OITgen}  With the notation above: 
 \begin{edesc} \label{OITintro} \item \label{OITintroa} The geometric monodromy $G_O$ of $O$ is an eventually $\ell$-Frattini  sequence; and  
 \item \label{OITintrob}  for each $J'$, $\sG( \bS_{J'})\cap \sG_O$ is eventually $\ell$-Frattini. 
 \end{edesc} 
 \end{guess} 

 Thm.~\ref{level0MT}  is evidence that \eqref{OITintro} is true, beyond Serre's modular curve \OIT\ case:  as in Prop.~\ref{PSLFrat}.  List \eqref{serreles} uses \eqref{OITintro} language  on Serre's \OIT.

\subsection{Using eventually $\ell$-Frattini} \label{uselfratt} For $K$ a number field and $\hat\psi: \hat W\to \prP^1_z$ a finite Galois cover $K$, with arithmetic monodromy group $G_{\hat \psi}$, denote the decomposition group of $z'\in \prP^1_z(K)$ by $D_{z'}$.  
Here is a  form of {\sl Hilbert's Irreducibility 
Theorem\/} (\HIT)  \cite[Chap.~12--13]{FrJ86}${}_2$. 
\begin{equation} \label{HIT}   \text{For a dense set of  $z'\in \prP^1_z(K)$, $G_{\hat \psi}=D_{z'}$.} \end{equation}  Dense can mean by almost any measure, including Zariski dense, or $p$-adically dense for $p$ a prime. 

\begin{prop} Use the $k_0$  from Def.~\ref{evenfratt}.  Combining \HIT\  for the cover ${}_\bQ\psi_{k_0}$ and \eql{OITintro}{OITintroa} implies, for a dense set of  $J'\in J_r(K)$, ${}_\bQ\hat \sG_O=\sG(\bS_{J'})$. \end{prop} 

Starting from \sdisplay{\cite{Se68}},  \cite[Chap.~6 \S3]{Fr20}  reshapes it for \MT s. \cite{Se68}  concentrates on two types of fibers over  $j'\in U_j\eqdef\prP^1_j(\bar \bQ)\setminus \{\infty\}$.  \begin{edesc} \label{serretype}  \item \label{serretypea}  Those called complex multiplication (\CM) coming from $j$-invariants of elliptic curves with nontrivial rings of endomorphisms (Def.~\ref{CMdef}). 
\item \label{serretypeb} Those for $j'$ that are not algebraic integers. 
\end{edesc} A promised Tate paper never materialized; it suggested {\sl all\/} non-\CM\ fibers  (not just those in \eql{serretype}{serretypeb}) would give fibers over $j'$ of  
\begin{equation} \label{faltingstype} \begin{array}{c} \text{$\GL_2$ type: $G_{\bQ(j')}$ acts on lines of the elliptic curve 1st  }\\ \text{$\ell$-adic cohomology as an open subset of $\GL_2(\bZ_\ell)$.}\end{array}\end{equation} Later Serre papers asked how \eql{serretype}{serretypeb} varies with $\ell$. For example:  

\begin{prop} For $j'\in U_j(\bQ)$ and almost all $\ell$, $G_{\bQ(j')}=\GL_2(\bZ_\ell)$. $$\text{Also, there exist $j'$ for which equality holds for all $\ell$.}$$ \end{prop} See Thm.~\ref{OITsf} and Ogg's examples at the end of \sdisplay{\cite{Fr78}}. That is appropriate for the part of Thm.~\ref{level0MT} we display, too.  In our case there are more than two types of fibers over  $j'\in U_j(\bar\bQ)$. Yet, for almost all $\ell$, they are all the expected general type. 

\subsection{\RIGP\ and Faltings} \label{RIGPFaltings}  Up to this point, the central object has been a Nielsen class.  \S\ref{l'RIGP} asks a question without referring to any specific conjugacy classes, nor anything about \MT s. Yet, the answer forces existence of \MT s.  \S\ref{toughestpt} explains how Falting's Theorem engages the most difficult aspect of Serre's case of the \OIT.  

\subsubsection{$\ell'$ \RIGP} \label{l'RIGP} As in \S \ref{lfratlatqt}, consider $L^\star \to G^\star \mapright{\tilde \psi^*} G$, an $\ell$-Frattini lattice quotient  of $\fG \ell {}_\ab \to G$. Prop.~\ref{bdrefreal} considers the \RIGP\ for the collection $\sG_{G,\ell,L^\star}\eqdef \{\psi_k: G_k\to G\}_{k=0}^\infty$ of corresponding $\ell$-Frattini covering groups. 

Call a regular realization $\ell'$ if it occurs with $\bfC$ consisting of $\ell'$ classes.   

\begin{prop} \label{bdrefreal} For any $\ell$-perfect $G$, assume we have $\bQ$ regular realizations of each $G_k$ (say, in $\ni(G,\bfC_k)$, $k\ge 0$).  Assume also that $r_{\bfC_k}$ is bounded independent of $k$ by $B$.  

Then,  there exists a \MT\ for $\tilde \psi^\star$ and some specific $\ell'$ classes $\bfC$ with $\bQ$ points at every level \sdisplay{\cite{FrK97}}.  \end{prop} 
That is, if the Main \RIGP\  \MT\ Conjecture is true, $\ell'$ regular realization of all groups in $\sG_{G,\ell}$ requires increasingly large numbers of branch points.

The easiest case, $D_\ell$ ($\ell$ odd), has  two contrasting situations. 
\begin{edesc} \label{ell'G} \item \label{ell'Ga}  There are regular realizations of all the $\sG_{D_\ell,\ell,\bZ_\ell}$. 
\item \label{ell'Gb}  Beyond the few cases known by Mazur's Theorem (see \sdisplay{\cite{Fr78}}), no one has produced any $\ell'$ realizations of the  $\sG_{D_\ell,\ell,\bZ_\ell}$. 
\end{edesc} Regular realizations for  \eql{ell'G}{ell'Ga}  are given in 1st year algebra. The branch point number   increases with $k$. 

We already discussed \eql{ell'G}{ell'Gb}:  $\ell'$ regular realizations correspond to cyclotomic torsion points on hypelliptic Jacobians in \S\ref{dihedtohyper}. 

Also, there are versions -- Thm.~\ref{genRIGPbdr} -- of both parts of \eqref{ell'G} for all the finite groups we have considered.

\subsubsection{$\GL_2$ toughest point} \label{toughestpt} As in \eql{qtresults}{qtresultsb}, start from a point at level 0, $\bp'_0$, corresponding to a Galois cover, $\hat \phi_{\bp_0}: \hat W_{\bp_0}\to \prP^1_z$.  Denote the complete collection of projective sequences over  $\bp'_0$ on, say, a \MT\    by $\bS_{\bp_0'}$. 

To simplify relating the \RIGP\ and the analog of Serre's case, assume $\bQ_{\sH_0}=\bQ$ and $k_0=0$ in \eqref{RIGPOITNCs}.  This gives a  $G_\bQ$ action on $\bS_{\bp_0'}$; each element gives a copy of the $\bZ_\ell[G]$ module  $L^\star$ \S\ref{notMT} as a  subspace on $H^1(\hat W_{\bp_0},\bZ_\ell)$. The $\ell$-adic action maps $G_\bQ$ into the group permuting such subspaces. 

For the $\GL_2$ type given by \eql{serretype}{serretypeb}, Serre used Tate's $\ell$-adic elliptic curve with $j$-invariant not an $\ell$-adic integer (discussion of {\eql{serreles}{serrelesb}). From {\sl wild ramification} on the $\bZ/\ell$ torsion on the Jacobian of the Tate elliptic curve at $j'$ comes  a crucial piece of $X_1(\bZ/\ell)\to \prP^1_z$  geometric monodromy. 

Using the Frattini property of \eql{OITintro}{OITintroa}, and the entwining  of the two Nielsen classes of Ex.~\ref{ncmodcurves}, that concluded Serre's result for such $j'$.      

For general \MT s, the analog of Serre's toughest $\GL_2$ case -- those not in \eql{serretype}{serretypeb} -- comes from $\bp_0'\in \sH_0(\bQ)$. If fine moduli holds, 
$$\text{this automatically gives  an $\ell'$  $\bQ$ regular realization in $\ni(G,\bfC)$.}$$

Completing Serre's \OIT\ awaited Falting's Theorem \cite{Fa83}.  For $K$ a number field, Conj.~\ref{mainconj} says high \MT\  levels have no $K$ points (\S\ref{95-04}). Both \sdisplay{\cite{Fr06} and \cite{CaTa09}} used Faltings in their versions for  $r=r_\bfC=4$. 

Faltings says, when the genus of a \MT\ level exceeds 1,  it has only finitely many rational points. If they exist (off the cusps) at each level -- applying the Tychonoff Theorem -- some subset of them would be part of a projective system of $K$ points on the \MT.  

This contradicts Weil's Theorem on Frobenius action on the first $\ell$-adic cohomology. That is, the Frobenius on the corresponding subspace ($\equiv L^\star$ as in \eql{qtresults}{qtresultsb})  would be trivial, not of absolute value $q^{\frac 1 2}$) when you reduce the tower modulo a {\sl good prime\/} $p|q$: $p$ is good if $p$ doesn't divide the order of $G$ (a conseqence of \cite{Gr71}). 

\cite{Fr06}  used properties of cusps  in the genus formula of Thm.~\ref{genuscomp}, as in Ex.~\ref{mainex}. See the discussion with $\ell$-lattice quotients in \S~\ref{exA43-2cont} for a particular case. While this is explicit for $r=4$ as to the level where the genus rises, Faltings isn't explicit for rational points disappearing at higher levels. 

The general case using the entwining of Nielsen classes of \eqref{RIGPOITNCs} -- as in the example below Thm.~\ref{level0MT} --  benefits from  both \cite{Fr06} and \cite{CaTa09}. 

Completing the conjectured \OIT\ requires a useful listing of the eventually $\ell$-Frattini sequences in the geometric monodromy of the \MT. Also, finding which are achieved  in \eql{OITintro}{OITintrob} as decomposition groups for some values of $j'\in \prP^1_z(\bar \bQ)$. \cite{Fr20} considers this, for being as explicit as Serre's \OIT\  --   as in the steps in \eqref{unfinishedC3^4} -- only for the case of Thm.~\ref{level0MT}. 

Consider an $\ell'$ \RIGP\ realizations of an $\ell$-Frattini cover $G_k\to G$ over a number field $K$.  This corresponds to a  $K$ point $\bp_k$ on a \MT. 
$$\begin{array}{c} \text{Conclusion: showing the $J'\in J_r(K)$ below $\bp_k$ satisfies the \OIT\ is akin to}\\ \text{Serre's toughest \OIT\ case, requiring extending Falting's Theorem.}\end{array}$$ 

\section{The path to {\bf M}(odular){\bf T}(ower)s}  \label{overview} \S\ref{history} emphasizes historical examples from the case $r=4$ (four branch point covers) that motivated this paper as a prelude to \cite{Fr20}. Considerable motivation came from interacting with Serre. This plays on modular curve virtues, as being upper half-plane quotients, even though congruence subgroups do not define the \MT\ levels except when $G$ is close to dihedral. 

\S\ref{S1title} elaborates on \lq\lq What Gauss told \dots.\rq\rq\ a hidden history that has obscured nonabelian aspects of {\bf R}(iemann)'s{\bf E}(xistence){\bf T}(heorem). The universal Frattini cover of  $G$ allows launching into such non-abelian aspects, though \cite{Fr20} takes a middle road.

Applications extending the \OIT\ and using $\ell$-adic monodromy come at the book's end. The \RIGP\ and its interpretations by rational points on Hurwitz spaces comes at the book's beginning. The middle joins these, using the universal Frattini cover $\tilde G$, and braid orbits on Nielsen classes.

Suppose, you accede to taking on a serious simple group (say, $A_5$, and a very small $\ell$-Frattini cover (for the prime $\ell=2$).  For example, as in \cite[Chap.~9]{Se92},  the short exact sequence $0 \to \bZ/2 \to \Spin_5\to A_5\to 1$. Then you might want to recognize the sequence for the abelianized $2$-Frattini cover \begin{equation} \label{A52}  0\to (\bZ_2)^5\to {}_2\tilde A_{5,\ab}\to A_5\to 1 \text{ \cite[Prop.~2.4]{Fr95}} \end{equation} (and its characteristic quotients) as quite a challenge at the present time.\footnote{ $\Spin_5\to  A_5$ is the smallest nontrivial quotient of ${}_2^1 A_5\to A_5$, as in \eqref{charlquots}.} Even according to Conj~\ref{MTconj} if you only had to find {\sl any\/} number field $K$ for which all those characteristic quotients have \RIGP\ realizations over $K$. The Prop.~\ref{bdrefreal}  question without the bound $B$.

 \S\ref{explainingFrattini}, the longest in this paper, explains the logic of the book, with extended abstracts on several papers on \MT s, starting from 1995.

\subsection{Historical motivations} \label{history} \S\ref{guides} alludes to my interactions with Serre on two seemingly disparate topics, at far separated periods. It is meant to show how those topics came together. The brief \S\ref{ack} acknowledges two referees.  It also connects to recent work that harkens to Serre's \cite{Se92} where he notes there was an alternative \RIGP\ approach. 

\subsubsection{Guides} \label{guides} 
This section is about on I learned from four examples. 

\begin{edesc} \label{mainexs} \item \label{mainexsa} Abel's spaces  as level 0 of a \MT\ classically denoted $ \{X_0(\ell^{k\np1})\}_{k=0}^\infty$. 
\item \label{mainexsb} The \MT\ from the Nielsen class $\ni(A_5,\bfC_{3^4})$ and the prime $\ell=2$. 
\item \label{mainexsc} The \MT\ system, from the Nielsen classes $\ni((\bZ/\ell)^2\xs\,\bZ/3, \bfC_{3^4})$, running over primes $\ell$, as my foray into an \OIT\ beyond Serre's.\footnote{What to do about $\ell=3$, is tricky. (See \sdisplay{\cite{FrH20} starting with \eqref{3vs2}.)}} 
\item \label{mainexsd} Relating $ \{X_1(\ell^{k\np1})\}_{k=0}^\infty$ and  $ \{X_0(\ell^{k\np1})\}_{k=0}^\infty$ as a special case of a general relation between {\sl inner\/} and {\sl absolute\/} Hurwitz spaces. 
\end{edesc} 

In \eql{mainexs}{mainexsa} and \eql{mainexs}{mainexsd}, there is a parameter $k$ indicating a tower level. Since in these two cases, the tower levels are traditionally related to a power of a prime $\ell$, I assume $k$ as a level requires no more explanation. 

My initial relation with the \OIT, during my first decade as a mathematician, was based on \eql{mainexs}{mainexsa} on  several stages using practical problems with considerable literature, on which this paper elaborates. The two series in \eql{mainexs}{mainexsd} are reasonably considered the mainstays of modular curves. 

My interactions with Serre on \cite{Se90a} and \cite{Se90b}, before they were written, related to my review of \cite{Se92} (see \cite{Fr94}), before it appeared  caused me to go deeply into \eql{mainexs}{mainexsb}. \cite{Fr90} ,that Serre saw me present in Paris in 1988, had him write to me asking -- essentially -- for the lift invariant formula for the families of genus 0 covers in the Nielsen classes $\ni(A_n,\bfC_{3^{n-1}})$, $n\ge 4$.

For that reason, I have alluded to \cite{Fr12} in the discussions. Especially, applied to a {\sl braid orbit\/}  on a Nielsen class, 
$$\text{the idea of the {\sl braid lift invariant\/}  from a {\sl central  Frattini cover}.} $$  The comparison between general and central Frattini covers of a finite group appears in many places to interpret \MT s.   

I aimed with \eql{mainexs}{mainexsc} to show commonalities and differences from the source of Serre's \OIT\ \eql{mainexs}{mainexsa}. Especially I refer to what works for modular curves. Also,  to what I learned that applies even to modular curves, though not previously observed, or the approach is different/illuminating. That is the concluding topic of \cite{Fr20}. Therefore I am brief on it here, merely recording some of its results that show what is new from an example that goes beyond Serre's \OIT. 

\subsubsection{Settings for the \RIGP\ and acknowledgements} \label{ack} 
If \RIGP\ realizations of $G$ exist, where are they? From \cite{FrV92} (and related), such must correspond to $\bQ$ points on Hurwitz spaces, with Main \RIGP\ Conj.~\ref{mainconj} a tool to understand that. We have three small subsections. 

{\sl Missing topic from \RIGP\ vs \IGP\ discussions: } Why  has the \RIGP\ (combined with Hilbert's Irreducibility Theorem) been more successful in finding \RIGP\ realizations of groups?  \cite{NScW00} proves that all solvable groups are Galois groups, but not with \RIGP\ realizations. 

Something akin to \RIGP\ realizations have now been formed for supersolvable groups \cite{HWi20}.
\footnote{A subseries of groups $G>G_1>G_2>\dots>G_t>\{1\}$ each of index a prime in the previous, but unlike solvable, each is normal in $G$.}  They use spaces that may not be rational varieties, but satisfy  {\sl weak-weak approximation}:  (say, over $\bQ$)  weak approximation for finite sets of primes outside some specific finite set of primes. 

\cite[\S3.5 and \S3.6]{Se92} shows that  varieties, $Z$, with a weak-weak approximation in place of $\prP^1_z$ suffice for Hilberts irreducibility in this sense:\footnote{A unirational variety -- image of a projective space -- was conjectured by Colliot-Thelene to have weak-weak approximation $\implies$ Noether covers $\afA^n\to \afA^n/G$, $G\le S_n$, would have an Hilbertian property.}   

If $\phi: W\to Z$ is a Galois cover with group $G$ over $\bQ$, then a dense set in  $Z(\bQ)$ has  irreducible $\phi$ fibers (producing $G$ as a Galois group) \cite{Ek90}. 

Instead of Noether covers, they embed $G$ into $\SL_N(\bQ)$. Instead of directly using weak-weak approximation, they use  {\sl unramified Brauer-Manin conditions}.  Prob.~\ref{weakweak} gives two questions for this approach.

\begin{prob} \label{weakweak}  Imitate Lem.~\ref{bpstoinfty} with $\tilde B\xs H$ where $B$ instead of abelian is an $\ell$-group, and $\tilde B$ is the lattice quotient from Lem.~\ref{Zlquot}. Assume $H$ is weak-weak style regularly realized. Also, how would you relate to Nielsen classes in this weak-weak regular style?
\end{prob}

{\sl What about homological stability:} It played a big role in \cite{FrV91} and \cite{FrV92} where its use was subtle, in relating absolute and inner classes (a footnote in \S\ref{limgptie}), and it will be in \cite{Fr20}. This is the story of making changes in one of the two parameters in a Nielsen class, $\ni(G,\bfC)$. \S\ref{limgptie}  discussed fixing $\bfC$ and changing $G$ to $G^\lm$. 

Homological stability is a consideration when fixing $G$ and letting $\bfC$ change by increasing the multiplicity of appearance of  (all) classes in $\bfC$,  thereby getting a sequence of new conjugacy classes in $G$. 

\cite[App.]{FrV91} found  there is a $k_0$ with dimensions of $\{H^0(\sH(G,\bfC_k))\}_{k\ge k_0}^\infty$ all the same:  the order of a quotient of  the Schur multiplier of $G$.\footnote{Saying this precisely requires the lift invariant definition without the assumption on $\bfC$ in Def.~\ref{Anliftdef}.} 

\S\ref{Ancomps} for $A_n$ is a special case. Actually, \cite{FrV92} used the distinguished component, wherein the lift invariant is trivial.  We didn't write out the full theorem into print at the time. Since then \cite{Sa19} has. 

We suspected stable homotopy would continue to higher cohomology, giving a stable $H^1(\sH(G,\bfC_k)$ on the components with trivial lift invariant -- under the same high multiplicity condition on the $\bfC_k\,$s. This would make its $\ell$-adic cohomology useful. That has been fulfilled by \cite{EVW19} when $G$ is dihedral.\footnote{Actually, their condition was that  $G$ is a group of order congruent to 2 mod 4 and they have used the conjugacy classes $\bfC$ to be supported in the unique class of involutions.}

{\sl This paper had two referees:} Both helped the author make the best of its first versions. The first referee -- known to me -- more familiar with the area made suggestions, including assuring inclusion of topics related to work of himself and his cowriters. The second referee had serious reorganization suggestions that shortened the original abstract, then expanded the introduction to point to the definitions and \S \ref{explainingFrattini} abstracts that would help guide the reader: {\bf Thank you to both!}

\subsection{What Gauss told Riemann about Abel's Theorem}  \label{S1title}  The title is the same as that of the paper \cite{Fr02}.  It must be shocking that most upper half-plane quotients -- $j$-line covers ramified over $0,1,\infty$ -- are not modular curves, or that they are related to practical problems. 

One lesson from ${}_\ell \sX_0\eqdef \{X_0(\ell^{k\np1})\}_{k=0}^\infty$  came from Galois. He computed the geometric monodromy of $X_0(\ell)\to \prP^1_j$ (over $\bC$; $k=0$). finding it to be $\PSL_2(\bZ/\ell)$ which is simple for $\ell\ge 5$. As an example of his famous theorem:  {\sl radicals\/} don't generate the algebraic functions describing modular curve covers of the $j$-line. 

\subsubsection{Early  \OIT\ and \MT\ relations} \label{earlyOIT-MT}  Recall Ex.~\ref{Dlk+1}. 
The \MT\ description of $X_0(\ell)$ in \eql{mainexs}{mainexsa}:  they are {\sl absolute\/} -- see \eql{eqname}{eqnameb} -- reduced Hurwitz spaces $\sH(D_{\ell^{k\np1}},\bfC_{2^4})^{\abs,\rd}$. There were two stages in this recognition. 

{\sl Stage 1}: \cite[\S2]{Fr78} starts the observation of the close relation between Serre's \OIT\ and the description of {\sl Schur covers\/} $f\in \bQ(w)$. Those are rational functions $f\in \bQ(w)$ for which: 
\begin{equation} \label{schurcover} \!\!\!\text{$f: \prP^1_w(\bZ/p) \to \prP^1_z(\bZ/p)$ is 1-1 for $\infty$-ly many $p$.}\end{equation}
Similarly, over any number field $K$ replace $\bQ(w)$ by  $K(w)$, and $\bZ/p$ by residue class fields $\sO_K/\bp=\bF_\bp$ :
\begin{equation} \label{schurcovernf} \!\!\!\text{$f$ is  1-1 on $ \prP^1_w(\bF_\bp)$ for $\infty$-ly many $\bp$.}
\end{equation}The Galois closure of such an $f$ may only be defined over a proper extension $\hat K/K$.\footnote{Extension of constants as in Ex.~\ref{ncmodcurves-cont}.} Indeed,  for $f$ to be a Schur cover over $K$, we must have $\hat K\not= K$.   

\newcommand{\exc}{{\text{\rm Exc}}}
Use the permutation notation of Rem.~\ref{permnot},  $T_f: \hat G_f\to S_n$ and the respective stabilizers of 1 by $\hat G_f(1)$, $G_f(1)$. 

To prevent accidents, define the exceptional set (for the Schur property): $$\exc_{f,K}=\{\bp \mid f \text{ is one-one on $\infty$-ly many extensions of }\bF_\bp\}.$$

If $f_i$, $i=1,2$, are exceptional over $K$, then  \begin{equation} \label{capexc} \text{so is $f_1\circ f_2$, if $|\exc_{f_1,K}\cap \exc_{f_2,K}| = \infty$.}\end{equation} That is, $f_1\circ f_2=f$ is a decomposition of $f$ over $K$.  The condition that $f$ is {\sl indecomposable\/}  over $K$ is that $T_f: \hat G_f\to S_n$ is {\sl primitive}: There is no group $H$ properly between $\hat G_f(1) $ and  $\hat G_f$.

\begin{prob}[Schur Covers]  Explicitly describe  \!{\sl exceptional \/} \!\!$f$ {\sl indecomposable\/} over $K$. \end{prob}  

Thm.~\ref{schurcoverfiber} connects exceptional $f$ with the \OIT: exceptionality relates the arithmetic and geometric monodromy of the covers from $f$.  

\begin{thm} \label{schurcoverfiber} The following is equivalent to exceptionality. 

There exists $\hat g \in \hat G_f(1)$, such that each orbit of $\lrang{G_f(1),\hat g}$ on $\{2,\dots,n\}$   breaks into (strictly) smaller orbits under $G_f(1)$ \cite[Prop.~2.1]{Fr78}. \end{thm}  

Thm.~\ref{schurcoverfiber} holds for essentially any cover (absolutely irreducible over $K$; the  sphere need not be the domain)  \cite[Prop.~2.3]{Fr05}. This application of a wide ranging Chebotarev density theorem is a case of  {\sl monodromy precision\/}. Usually a {\sl Chebotarev density\/} application in the $\implies$ direction isn't so precise. Yet, here instead of saying $f$ is {\sl almost\/} one-one, it implies it is exactly one-one for $\infty$-ly many $\bp$. 

\cite{Fr05} expanded on situations giving monodromy precision. This is as an improvement on the intricate industry refining the appropriate error term in the Riemann hypothesis over finite fields. 

To relate to the \OIT\ for rational $f$, it turned out enough to concentrate on two cases with $\ell$ prime: $\deg(f)=\ell$, or $\ell^2$. The next step was to describe those Nielsen classes that produce the exceptional $f$. \cite[Thm.~2.1]{Fr78} lists the Nielsen classes (of genus 0 covers) for $\deg(f)=\ell$ that satisfy these conditions, noting a short list of 3-branch point cases, with a main case of $r=4$ branch points forming one connected family.  

Thereby, it identifies $X_0(\ell^{k\np1})$ (resp.~$X_1(\ell^{k\np1})$)  as reduced absolute (resp.~inner) Hurwitz spaces as in \eql{mainexs}{mainexsd}. This was a special case of \cite[\S3]{Fr78}, the {\sl extension of constants\/} rubric for covers, by going to their Galois closure. 

\begin{center} This sequence encodes the moduli interpretation Inner to Absolute: \end{center}  \begin{equation} \label{inabsseq} \text{ } \sH(G,\bfC)^\inn\to \sH(G,\bfC)^{\abs} \to U_r\end{equation}  and its expansion to total spaces (over $U_r\times \prP^1_z$) as in \cite[Thm.~1]{FrV91}.  This  relates $G_f$ (geometric) and $\hat G_f$ (arithmetic) monodromy, as in \S\ref{whyrationalfuncts}. 

\cite[\S2]{Fr78} shows, for prime degree $\ell$ rational functions, identifying Schur covers is essentially equivalent to the theory of complex multiplication. Further, from that theory, we may describe $\exc_{f,K}$ as an explicit union of arithmetic progressions, thereby allowing testing the condition \eqref{capexc}.  

Describing prime-squared degree exceptional rational functions interprets the $\GL_2$ part of Serre's \OIT, as in \S\ref{uselfratt} \cite[\S6.1--\S6.3]{Fr05}. We state the main point, over $\bQ$, again using precision. This also fits the inner-absolute Hurwitz space relation above by using the limit Nielsen class $\ni((\bZ/\ell)^2\xs \bZ/2,\bfC_{2^4})$ Ex.~\ref{ncmodcurves-cont}, still a modular curve case. The Nielsen class collection $\{\ni((\bZ/\ell^{k\np1})^2\xs \bZ/2,\bfC_{2^4})\}_{k=0}^\infty$ identifies with a modular curve tower \cite[Prop.~6.6]{Fr05}. 

For a given $\ell$, from Serre's (eventual, Thm.\ref{OITsf}) version of the \OIT\, we conclude this. If the elliptic curve $E$ (say, over $\bQ$) has a $\GL_2$ $j$-invariant, $j_E=j_0$, then the corresponding degree $\ell^2$ rational function $f_{j_0,\ell^2}$ has arithmetic/geometric monodromy group quotient $\GL_2(\bZ/\ell)/\{\pm1\}=G(\bQ_{j',\ell^2}/\bQ)$ for all primes $\ell \ge \ell_0$ for some $\ell_0$ dependent on $j_0$ \cite{Se97b}. 
For a Galois extension $\hat L/\bQ$, use $\Fr_{L,p}$ for the Frobenius (conjugacy class) at $p$. 

\begin{prop} \label{excGL2} Given any $\ell\ge \ell_0$ as above, $\exc_{f_{j',\ell^2}}$ is the set of $p$ so that the group  $\lrang{-1,\Fr_{\bQ_{j',\ell^2},p}}$ acts irreducibly on $(\bZ/\ell)^2$. 
This is always infinite.  \end{prop}

\begin{proof} The classical Chebotarev density theorem implies $\exc_{f_{j',\ell^2}}$ is infinite if any element of $\GL_2(\bZ/\ell)$ acts irreducibly on $(\bZ/\ell)^2$.  For example, on the degree 2 extension $\bF_{\ell^2}$ of $\bZ/\ell=\bF_\ell$, multiply by a primitive generator $\alpha$ of $\bF_{\ell^2}/\bF_\ell$ to get an invertible $2\times 2$ matrix with no invariant subspace. \end{proof} 

The $\GL_2$ case is vastly different from the $\CM$ case in that the exceptional set described in Prop.~\ref{excGL2} is definitely not a union of arithmetic progressions. \cite[\S6.3.2]{Fr05} relates to \cite{Se81} on using the (conjectural) Langlands program to consider these exceptional sets. 

Guralnick-M\"uller-Saxl \cite{GMS03} show that -- excluding those above -- other  indecomposable Schur covers by rational functions,  are {\sl sporadic\/}. That is,  they correspond to points on a finite set of Hurwitz space components.

\begin{defn}[Named Nielsen Classes] \label{namedclasses} Use the respective names \CM\ and $\GL_2$ for the Nielsen classes $\ni(D_\ell,\bfC_{2^4})^{\dagger,\rd}$ and $\ni((\bZ/\ell)^2\xs \bZ/2,\bfC_{2^4})^{\dagger,\rd}$ (or to the whole series with $\ell^{k\np1}, k\ge 0$ in place of $\ell$), with $\dagger$ referring to either inner or absolute equivalence, and the relation between them. \end{defn} 

{\sl Stage 2\/} discussion, and its relation to the \OIT, couldn't happen until there was a full formulation of \MT s, the topic of \S\ref{explainingFrattini}. Still, a transitional phase after Stage 1 occurred with dihedral groups and the space of hyperelliptic jacobians as in \S\ref{dihedtohyper} and \S\ref{Frat-Groth}.   

The  \MT\ free  question of Prop.~\ref{bdrefreal} shows how \MT s for each finite group $G$ and $\ell$-perfect prime of $G$ generalizes what the same question applied to dihedral groups posed for hyerelliptic jacobians. This -- expanded in \sdisplay{\cite{CaD08}} -- shows the \MT\ project, through the \RIGP, tying to classical considerations.

\subsubsection{Competition between algebra and analysis} \label{algvsanal} The full title of \cite[\S7]{Fr02} is {\sl Competition between algebraic and analytic approaches}.  This subsection consists of brief extracts from it and \cite[\S10]{Fr02}. \S7 was gleaned partly from \cite{Ne81}, and my own observation of \cite{Ahl79} and \cite{Sp57}. \S10 was a personal \lq\lq modern\rq\rq\  dealing with the common \lq appreciation\rq\ of mathematical genius.

{\sl Riemann's early education} \cite[\S7.1]{Fr02}:  Riemann was suitable, as no other German mathematician then, to effect the first synthesis of the \lq\lq French\rq\rq\ and \lq\lq German\rq\rq\ approaches in general complex function theory. 

{\sl Competition between Riemann and Weierstrass} \cite[\S7.2]{Fr02} and \cite[p.~93]{Ne81}:   In 1856 the competition between Riemann and Weierstrass became intense, around the solution of the {\sl Jacobi Inversion problem}. Weierstrass consequently withdrew the 3rd installment of his investigations, which he had in the meantime finished and submitted to the Berlin Academy. 

{\sl Soon after Riemann died} \cite[\S7.3]{Fr02} and \cite[p.~96]{Ne81}:  
After Riemann's death, Weierstrass attacked his methods often and even openly. Curiously, the unique reference in Ahlfor's book was this: 

\begin{quote} Without use of integration R.L.Plunkett proved the continuity of the derivative (BAMS65,1959). E.H.~Connell  and P.~Porcelli proved existence of all derivatives (BAMS 67, 1961). Both proofs lean on a topological theorem due to G. T. Whyburn. \end{quote} This is an oblique reference to to Riemann's use of {\sl Dirichlet's Principle\/} for constructing the universal covering space of a Riemann surface. 
 
There is a complication in analyzing Neuenschwanden's thesis that this resulted in mathematicians accepting Riemann's methods. How does this event resurrect the esteem of Riemann's geometric/analytic view? 

{\sl Final anecdote} \cite[\S 10]{Fr02}: While at the Planck Institute in Bonn, to give talks in the early 21st Century, I visited Martina Mittag, a humanities scholar, who had earlier visited UC Irvine. In private conversation she railed that mathematicians lacked the imagination of humanities scholars. Yet, she was vehement on the virtues of Einstein. 

I explained that Einstein was far from without precedent; that we mathematicians had geniuses of his imaginative. My example was Riemann: I called him the man who formed the equations that gave Einstein his scalar curvature criterion for gravity: his thesis, and admittedly not my expertise. \lq\lq Mike,\rq\rq\ she said, \lq\lq You're just making that up! Who is Riemann?\rq\rq  

I took the {\sl R\/} book in her (German) encyclopedia series from the shelves on her walls, without the slightest idea of what I would find. Opening to Riemann, I found this [in German] in the first paragraph: 

\begin{quote} Bernhard Riemann was one of the most profound geniuses of modern times. Notable among his discoveries were the equations that Einstein later applied to general relativity theory. \end{quote}

In modern parlance, what Gauss explained to Riemann was what -- when I was young -- were called the {\sl cuts\/}. These are always displayed with pictures that are impossible -- not hard, but rather {\sl cannot\/} exist, as in \cite[Chap.~4, \S 2.4]{Fr80} under the title \lq\lq Cuts and Impossible Pictures.\rq\rq 

The pictures usually apply to covers that are cyclic, degree 2 or 3, as in \cite[p.~243]{Con78}, which, though, is  excellent in many ways for students not comfortable with algebra. What Riemann learned, again in modern parlance, is that you don't need -- explicitly -- the universal covering space, nor a subgroup of  its automorphism group, to produce  covers.\footnote{As in \S\ref{compsr=4} remarks that for $r=4$, you aren't using congruence subgroups to form the reduced Hurwitz spaces for $r=4$.}   

\subsubsection{Profinite: Frattini and Grothendieck} \label{Frat-Groth} 
I gleaned \S\ref{algvsanal} from reading long ago (from \cite{Sp57}),  that Riemann's $\theta\,$s, in a sense defined Torelli space, the period matrix cover of the moduli space, $\sM_\geng$, of curves of genus $\geng$. This codifies the integrals that Riemann used to introduce one version of moduli of curves. \cite{Sp57}  does explain fundamental groups. Yet, it always relies on universal covering spaces. 

One famous theorem is that the universal covering space of $\sM_\geng$ is a (simply-connected, Teichm\"uller) ball: It is contractible. Many beautiful pictures of fundamental domains come from this. 

This paper (and \cite{Fr20}) uses Hurwitz spaces, to investigate the place of one group, $G$, at a time. Yet, it considers the full gamet of its appearances through varying conjugacy classes defining covers with $G$ as monodromy. In doing that it replaces universal covering spaces with the {\sl Universal Frattini\/} cover $\tilde G\to G$, and its abelianized version $\tilde G_\ab \to G$ (as in \S\ref{univfratpre}).  

It introduces new tools, albeit with a topological component: the {\sl lift invariant\/} (as in Thm.~\ref{obstabMT}) and the {\sl $\sh$-incidence matrix\/} (as in \S\ref{shincex} with applications such as Prop.~\ref{A43-2} et. al.). These tie directly to group cohomology/modular representations. Thus by-passing relevant but famously difficult problems like dealing with non-congruence subgroups. 

The relation between $\sM_\geng$ and Hurwitz spaces starts by realizing that the latter generalizes the former, adding data that divides the former into smaller pieces. Using that division effectively does not require you must know all finite groups (or even all simple groups). 

\cite[App.~B]{Fr20} has a section on how each problem appears to have its own appropriate finite groups, based on a well-known paradigm -- {\sl the Genus 0 problem\/}. This is commentary on \cite{Fr05}, a guide inspired by solving problems like {\sl Schur's}, {\sl Davenport's\/} and {\sl Schinzel's}, that came from the middle of the 20th century or before.

There are $\infty$-ly many spaces $\sM_\geng$ (resp.~$\sH(G,\bfC,\ell)$) indexed by $\geng$ \\ (resp.~$(G,\bfC,\ell)$)  in each case. The indexing seems more complicated in the Hurwitz case. Yet, in the former case, you suspect they should all fit together. Grothendieck's famous {\sl Teichm\"uller\/} group attempted to gather their presence together into one profinite group with an hypothesis that he was describing $G_\bQ$.

That created quite an industry. Still, \cite{FrV92} showed that the Hurwitz space approach was up to the challenge of describing properties of $G_\bQ$ that most mathematicians can understand. For example \eqref{FrVSns}.  

\begin{thm} \label{FrV91} We may choose a(n infinite) Galois algebraic extension  $L/\bQ$  so that  $G_\bQ$ has a presentation (see also Conj.~\ref{FrVConj}): 
\begin{equation} \label{FrVSns} 1\to  F_{\omega}=G_L \to G_\bQ \to \Pi_{n=2}^\infty  S_n \eqdef \sS_\infty \to 1 \end{equation} 
That is $G_\bQ$, has a product of $S_n\,$s as a quotient (the Galois group of $L/\bQ$) with the kernel a pro-free group on a countable set of generators. 
\end{thm} 

This overview result hid that these were practical techniques giving a new context connecting classical problems to the \RIGP. We illustrated that first with Thm.~\ref{bdrefreal} connecting involution realizations of dihedral groups with torsion points on hyperelliptic jacobians -- as developed in \S \ref{dihedtohyper} and Prop.~\ref{dihcase} with its direct connection to Conj.~\ref{Tormain}. 

\begin{guess} \label{Tormain} Torsion Conjecture: There is a negative conclusion to statement \eqref{impossdih} on existence of a $\bQ$ cyclotomic point of order $\ell^{k\np1}$, for each $k$,  among all   hyperelliptic jacobians of any fixed dimension $d$. 
  
$B$-free Conjecture: Without any bound $B$, for each $\ell^{k\np1}$ there is a $\bQ$ cyclotomic point on some hyperelliptic jacobian, corresponding to a $(D_{\ell^{k\np1}},\bfC_{2^r})$ ($r$ dependent on $\ell^{k\np1}$) \RIGP\ involution realization.  \end{guess} 

Despite the last part of Conj.~\ref{Tormain}, no one has found those \RIGP\ {\sl involution} realizations beyond $r=4$ and $\ell=7$. The theme of \cite[\S7]{Fr94} -- using this paper's notation -- still seems reasonable. For any prime $\ell\ge 3$, as in \S\ref{earlyOIT-MT}, and given a choice, you should rather 
$$\text{regularly realize the {\sl Monster\/} than the collection $\{D_{\ell^{k\np1}}\}_{k=0}^\infty$.}\footnote{referring to the famous {\sl Monster\/} simple group.}$$  

Thm.~\ref{bdrefreal}  generalizes this consideration to all $\ell$-perfect finite groups  (for example the $A_5$, $\ell=2$ case of \eqref{A52}). Then, the negative conclusion using the Torsion conjecture generalization of Conj.~\ref{Tormain} as in \sdisplay{\cite{CaD08}}.  This leads to progress, Thm.~\ref{truer=4}, on Conj.~\ref{mainconj}.
 
\begin{guess}[Main \RIGP\ Conj.] \label{mainconj} High {\bf MT} levels have {\sl general type\/} and no $\bQ$ points.\footnote{{\sl General type\/} means that some multiple of the  {\sl canonical bundle\/} gives a projective embedding of the variety.} \end{guess} 

\begin{thm}  \label{truer=4} Conj.~\ref{mainconj} is true for $r=4$, where $\sH_k\,$s are upper half-plane quotients. Thm.~\ref{genuscomp} is a tool for showing the genus rises with $k$. \end{thm}  

For $K$ a number field, concluding in Conj.~\ref{mainconj} that high \MT\ levels have no $\bQ$ points is of significance only if there is a uniform bound on the definition fields of the \MT\ levels. Therefore distinguishing between towers with such a uniform bound, and figuring the definition field as the levels grow if there is no uniform bound, is a major problem. 

Our approach allows us to compute, and to list properties of \MT\ levels. This is progress in meeting  Grothen\-dieck's objection that jacobian correspondences impossibly complicate  generalizing Serre's \OIT.

In our \cite[\S5]{Fr20} example, that complication is measured by the appearance of distinct Hurwitz space components. The {\sl lift invariant\/} accounts for most. Still, others pose a problem at this time -- we know them, but not their moduli definition fields, as {\sl Harbater-Mumford\/} components. 

That problem occurs because there is more than one with the same 0 lift invariant as discussed around Prop.~\ref{A43-2} and Thm.~\ref{level0MT}.  As in Thm.~\ref{genuscomp}, we know their braid orbits on the Nielsen class; modular curves and complex multiplication are not a guide.  

\subsection{The \TL\ of the \MT\ program} \label{explainingFrattini}  After a prelude we have divided this section into three subsections: 

\begin{itemize} \item \S\ref{pre95}, prior to 1995; 
\item  \S\ref{95-04}, the next decade of constructions/main conjectures; and   
\item  \S\ref{05-to-now}, progress on the Main \OIT\  \ref{OITgen} and \RIGP\ \ref{mainconj} conjectures. \end{itemize}

\subsubsection{Organization} \label{organization} Each \TL \  item connects to a fuller explanation of the history/significance of a paper's contributions. 

We trace the \RIGP\ literature, starting with the definition of  Nielsen classes (Def.~\ref{NielsenClass} and Thm.~\ref{BCYCs}), then going to  \MT\ conjectures as in \S\ref{preimRIGPOIT}. 

Notation reminder:
$\ni(G,\bfC)$ referencing (unordered) conjugacy classes $\bfC = \{\row \C r\}$ of a finite group $G$. {\bf R}(iemann)-{\bf H}(urwitz) \eqref{RH} gives the genus $\geng\eqdef \geng_\bg$ of a sphere cover corresponding to $(\row g r)\in \ni(G,\bfC)$.
 
The {\sl Branch Cycle Lemma\/} (\BCL) ties moduli definition fields (Def.~\ref{moddeffield}) of covers (and their automorphisms) to branch point locations. A whole section in \cite[Chap.~2 \S4]{Fr20} is taken with the \BCL\ for good reasons. There is nothing else quite like it in most moduli space thinking.   

Especially, it gives the precise moduli definition field of Hurwitz {\sl families\/} for $\ni(G,\bfC)$ for any equivalences (as in \S\ref{equivalences}). While this is key for number theory (on the \RIGP, and generalizing Serre's \OIT), we emphasize one easily stated corollary of Lem.~\ref{branchCycle}.

\begin{thm} \label{bclthm} The total space of an inner (even as a reduced) Hurwitz space, $\sH(G,\bfC)^{\inn}$ together with its extra structure as a moduli space of $\prP^1_z$ covers,  is a cyclotomic field given in the response to \eql{mtneed}{mtneedc} as \eqref{bcl}.\footnote{Total space means to includes a representing family of covers $\sT\to \sH(G,\bfC)^{\inn}\times \prP^1_z$ in the case of fine moduli, where there is one, and standard generalizations of this.} In particular, an inner Hurwitz (moduli) space  structure is defined over over $\bQ$ if and only if $$\text{$\bfC^u = \bfC$  for all $u\in (\bZ/N_{\bfC})^*$: $\bfC$ is a {\sl rational union}.}\footnote{Here $\bfC^u$  means to put each element of $\bfC$ to the power $u$.}$$ 
\end{thm} 

For this reason we use rational unions of conjugacy classes in all examples.  Individual \MT s have an attached prime (denoted $p$ in the early papers, but $\ell$ here because of the latest work).\footnote{It is insufficient to say that the Hurwitz space is defined over a given field. Examples both old and new,  \cite[Chap.~2 \S4.3]{Fr20}, include those  with  components isomorphic as covers of $U_r$, but inequivalent as moduli carrying family structures.}  

We support the somewhat abstract description of a \MT\ in \S\ref{outlineMTs} with many examples.  \MT s come with what we call the usual \MT\ conditions: 
\begin{edesc} \label{MTconds} \item \label{MTcondsa} Each has an attached group $G$, and a collection of $r$ conjugacy classes, $\bfC$ in $G$ with $\ell'$ elements (of orders prime to $\ell$). 
\item \label{MTcondsb} Further, $G$ is $\ell$-perfect: $\ell$ divides $|G|$, but $G$ has no surjective homomorphism to $\bZ/\ell$.
\end{edesc}  For $G$ a dihedral group, with $\ell$ odd and $r=4$, we are in the case of {\sl modular curve towers}. So, \MT s generalizes modular curves towers. Since there are so many $\ell$-perfect groups, the generalization is huge. 

Conditions to form a \MT\ begin with an $\ell$-Frattini lattice quotient \S\ref{lfratlatqt} of the universal (abelianized) $\ell$-Frattini cover ${}_\ell \tilde \psi: \fG \ell {}_\ab \to G$ of $G$ \S\ref{univfratpre}. Given such a lattice quotient, $\ell'$ conjugacy classes $\bfC$ of $G$ and a braid orbit $O$ on a Nielsen class $\ni(G,\bfC)$, there is a succinct cohomology test, Thm.~\ref{obstabMT}, for existence of a \MT\ -- and therefore of its corresponding projective sequence of nonempty absolutely irreducible varieties Ð above $O$. 

When we started this project, we considered our Main \RIGP\ conjecture \ref{mainconj} only on the maximal lattice quotient. With, however, solid evidence  from  $(A_5,\bfC_{3^4}, 
\ell=2)$, and related examples (\S\ref{l=2level0MT}), among good reasons for opening up the territory to all lattice quotients,   we could then include diophantine connections to classical spaces. Our main examples do just that, thus benefitting the classical spaces, too. 

\MT\ data passing the lift invariant test gives an infinite (projective) system of nonempty levels. Using reduced equivalence of $\prP^1_z$ covers, as in Ex.~\ref{mainex}, each level, $\sH_k'$, has a projective normalized compactification $\bar \sH'_k$. This $r\nm 3$ dimensional algebriac variety covers the compactification of $J_r$, a quotient of an open subset of $\prP^r$ by $\SL_2(\bC)$. For $r=4$, this is the classical $j$-line, $\prP^1_j$.

Here is a paraphrase of Main \RIGP\ Conj. \ref{mainconj}: 
\begin{edesc} \label{MainConj} \item \label{MainConja} High tower levels have general type; and 
\item  \label{MainConjb}  even if all levels have a fixed definition field $K$, finite over $\bQ$, still $K$ points disappear (off the cusps) at high levels. 
\end{edesc} Bringing particular \MT s alive plays on {\sl cusps}, as do modular curves. Cusps already appeared in Thm.~\ref{genuscomp} as the disjoint cycles of $\gamma_\infty$ (corresponding to the points over $\infty$ on the $j$-line). \S\ref{05-to-now}  of our TimeLine precisely compares  \MT\ cusps with those of modular curve towers, and consequences of this. 

It also discusses two different methods giving substantial progress on the Main Conjectures.  The argument of \S\ref{RIGPFaltings} -- dependent on Falting's Theorem -- shows  why \eql{MainConj}{MainConja} implies \eql{MainConj}{MainConjb} when $r=4$. 

The $\sh$(ift)-incidence matrix graphical device,  used in the table above Prop.~\ref{A43-2}, displays these cusps, and the components -- corresponding to blocks in the matrix -- in which they fall.

Several papers emphasize that cusps for Hurwitz spaces often have extra structure -- meaningful enough to suggest special names for them -- that come from group theory in ways that doesn't appear in the usual function theory approach to cusps. The use of the names {\sl Harbater-Mumford\/} and {\sl double identity\/} cusps in Thm.~\ref{level0MT} are examples of these.  

We emphasize two connections between  abelian varieties  and \MT s:  The {\bf S}(trong) {\bf T}(orsion) {\bf C}(onjecture)  \sdisplay{\cite{CaD08}} and $\ell$-adic representations \S\ref{prodladic}. Here is the main  \MT\ device for applications to the \RIGP\ and the \OIT.  
$$\begin{array}{c} \text{Give meaningful labels to cusps on a component}\\ \text{of the space of covers of $\prP^1_z$ defined by a Nielsen class.}\end{array}$$ 

\subsubsection{Lessons from Dihedral groups Ð Before '95} \label{pre95} 

This section goes from well-known projects to their connection with the \MT\ program. The references to Serre's work was around two very different types of mathematics: His \OIT, with its hints of a bigger presence of Hilbert's Irreducility Theorem, and his desire to understand the difficulty of regularly realizing the \Spin\ cover of $A_n$.  
This section concludes with \sdisplay{\cite{DFr94}} when the project divided into two branches.  

The arithmetic concentrated in the hands of Pierre D\`ebes and his collaborators Cadoret, Deschamps and Emsalem. The structure of particular \MT s -- based on homological algebra and cusp and component geometry of the spaces  -- follows my papers and my relation to Bailey and Kopeliovic with influence from the work discussed with Liu-Osserman and Weigel. Effect of quoted work of Ihara, Matsumoto and Wewers, all present at my first \MT\ talks, is harder to classify. 

\Sdisplay{\cite{Sh64}} 
 I studied this during my two year post-doctoral 67--69 at IAS (the {\sl Inst.~for Advanced Study}). Standout observation: Relating a moduli space's properties to objects represented by its points, through the {\sl Weil co-cycle condition\/}. That lead to {\sl fine moduli\/} conditions on absolute, inner and reduced Hurwitz spaces  (resp.~\cite[Thm.~5.1]{Fr77} and \cite[\S4.3]{BFr02}).
 
Having fine moduli gives positive solutions for a group $G$ toward the \RIGP\  from rational points on inner moduli spaces. Use the notation for the stabilizer of the integer 1 in the representation $T$ given by Rem.~\ref{permnot}. Here are those respective conditions (as in \eqref{eqname}): 
\begin{edesc}  \label{fmodabsinn}  \item   \label{fmodabsinna}  Absolute equivalence:  Given $T: G \to S_n$,  as a subgroup of $G$, $G(T,1)$ is its own normalizer.
\item  \label{fmodabsinnb}  Inner equivalence: $G$ is centerless. \end{edesc} 

Refer to $(G,T)$ satisfying \eql{fmodabsinn}{fmodabsinna} as {\sl self-normalizing}. Fine moduli gives a (unique) total space $\sT$, over $\sH\times\prP^1_z$, with $\sH$ the Hurwitz space. This represents covers corresponding to $\bp\in \sH$: $$\text{$\sT_\bp \to \bp\times \prP^1_z$ from pullback of $\bp\times \prP^1_z$ to $\sT$.}$$ 
 
The fine moduli condition, with the addition of reduced equivalence to Nielsen classes (Def.~\ref{redaction}) is in \cite[\S4.3]{BFr02}, as in our example Ex.~\ref{exA43-2}. 
 
Results: The \BCL\  Thm.~\ref{bcl} and early uses starting with the solution of {\sl Davenport's problem\/} as in  \cite[\S1]{Fr12}, for problems not previously considered as moduli-related. Later: A model for producing \lq\lq automorphic functions\rq\rq\  supporting the Torelli analogy through $\theta$ nulls on a Hurwitz space (\S\ref{Frat-Groth}) as in \cite[\S6]{Fr10}.

\Sdisplay{\cite{Se68}} Serre gave one lecture on his book during my 2nd year (1968-1969) post-doctoral at IAS. His amenuenses were writing his notes. I asked questions and interpreted the hoped for theorem -- a little different than did Serre -- as this. For each fixed $\ell$ as $j'\in \bar\bQ$ varies, consider the field, ${}_\ell\sK_{j'}$, generated over $\bQ(j')$ by  coordinates of any projective sequence of points $${}_\ell\bx'\eqdef \{x_k'\in X_0(\ell^{k\np1})\}_{k=0}^\infty \mid \dots \mapsto x_{k\np1}'\mapsto x_k'\mapsto \dots \mapsto j'.$$ 

Denote the  Galois closure of ${}_\ell\sK_{j'}/\bQ(j')$ by ${}_\ell\hat\sK_{j'}$, and its Galois group, the {\sl decomposition group\/} at $j'$, by ${}_\ell \hat G_{j'}$.\footnote{Potential confusion of notation: $j$ here is not an index, but the traditional variable used for the classical $j$-line.} Imitating the notation of the arithmetic monodromy group of a cover in Def.~\ref{arithmon}, denote the arithmetic monodromy group of the cover 
$${}_\ell \phi_{j,k}: X_0(\ell^{k\np1})\to \prP^1_j \text{ by ${}_\ell \hat G_{{}_\ell \phi_{j,k}}\eqdef {}_\ell \hat G_{j,k} $ and ${}_\ell \hat G_j$ its projective limit.} $$ 
Similarly, without the $\hat{}$, the projective limit of the geometric monodromy is ${}_\ell G_j$. For $j'\in \bar \bQ$, 
 in a natural way ${}_\ell \hat G_{j'}\le {}_\ell \hat G_j$.\footnote{The points, $\{j=0,1\}$  of ramification of the covers are special. We exclude them here; a more precise result (due to Hilbert) includes them as \CM\ points, too.} 
 
Such fields don't vary smoothly: they birfurcate into two very distinctive types: \CM\ \eql{serretype}{serretypea} from the theory of complex multiplication which takes up a great part of \cite{Se68}; and $\GL_2$. The last rightly divides into two types itself:   \eql{serretype}{serretypeb} for $j'$ not an algebraic integer on which Serre's book gets a grasp, and \eqref{faltingstype}. \S\ref{toughestpt} explains the extreme difference between these two, despite that in both cases ${}_\ell \hat G_j$ is open in $\GL_2(\bZ_\ell)$.

Using the Def.~\ref{evenfratt} of eventually $\ell$-Frattini in Prop.~\ref{PSLFrat} and \S\ref{MTsOIT}, what made an impact on our approach from \cite{Se68} was this.   

\begin{edesc} \label{serreles} \item \label{serrelesa} For each fixed $\ell$, ${}_\ell G_j$ is eventually $\ell$-Frattini. Further, for $\ell> 3$, it is $\ell$-Frattini (right from the beginning). 
\item \label{serrelesb} If for some prime $p$, $j'\in \bar \bQ$ is not integral at $p$, then the intersection ${}_\ell \hat G_{j'}\cap {}_\ell G_j$ is open in ${}_\ell G_j$. 

\item \label{serrelesc} If a given $j'$ has complex multiplication type (Def.~\ref{CMdef}), then the intersection of ${}_\ell \hat G_{j'}$ with ${}_\ell G_j$  is eventually $\ell$-Frattini. 

\item \label{serrelesd} From either \eql{serreles}{serrelesb} or \eql{serreles}{serrelesc}, you have only to get to a value of $k'$ with ${}_\ell \hat G_{j',k'}$ within the $\ell$-Frattini region to  assure achieving an open subgroup of the respective $\GL_2$ or $\CM$ expectation.  \end{edesc} 

The group ${}_\ell G_j$ in \eql{serreles}{serrelesa} is $\PSL_2(\bZ_\ell)$. Serre frames his result differently, so his group is $\SL_2(\bZ_\ell)$. The distinction is between the monodromy group view from a \MT\ (\S\ref{notMT} ) versus going to the $\ell$-adic representation view (\S\ref{limgptie}). Both are necessary for progress on the Main \OIT\ Conj.~\ref{OITgen}.   

We interpret \eql{serreles}{serrelesb} as giving an $\ell$-adic germ representating the moduli space -- through Tate's $\ell$-adically uniformized elliptic curve --- around the (long) cusp we call Harbater-Mumford on $X_0(\ell)$.  This is a model for  gleaning $G_{\bQ_\ell}$ action for $j$ $\ell$-adically \lq\lq close to\rq\rq\ $\infty$.  

Suppose $K$ is a complex quadratic extension of $\bQ$. The technical point of complex multiplication is the discussion of 1-dimensional characters of $G_K$ on the $\bQ_\ell$ vector space -- Tate module, or 1st $\ell$-adic \'etale cohomomology -- of an elliptic curve with complex multiplication by $K$. On the 2nd $\ell$-adic \'etale cohomomology it is the cyclotomic character;  on the 1st there is no subrepresentation of any power of the cyclotomic character. 

Only a part of abelian extensions of $K$ are cyclotomic -- generated by roots of 1, a result that generalizes to higher dimensional complex multiplication in \cite{Sh64}. Much of \cite{Se68} is taken with \eql{serreles}{serrelesc}. The groups there are primarily the (abelian) ideal groups of classical complex multiplication.  

As \cite{Ri90} emphasizes, Serre's book is still relevant, especially for the role of abelian characters, those represented by actions on Tate modules (from abelian varieties), and those not. 

We reference this discussion in many places below. The full (and comfortable) completion of Serre's \OIT\ awaited replacement of an unpublished Tate piece by ingredients from Falting's Thm. \cite{Fa83} (as in \cite{Se97b} and the more complete discusion of \S\ref{RIGPFaltings}). 

\Sdisplay{\cite{Fr78}} This was the forerunner of the always present relation between absolute, $\sH(G,\bfC)^\abs$, and inner, $\sH(G,\bfC)^\inn$, Hurwitz spaces \eqref{eqname}.  The latter naturally maps -- via the equivalence -- to the former. \cite[\S3]{Fr78} -- {\sl Determination of arithmetic monodromy from branch cycles\/} -- was based on the idea I informally call {\sl extension of constants}.   

The definition field of an absolute cover in a Hurwitz family (represented by $\bp\in \sH(G,\bfC)^\abs$) might require a definition field extension from going to the Galois closure of the cover. That extension comes from coordinates of any $\hat \bp\in \sH(G,\bfC)^\inn$  above $\bp$. \cite[Thm.~1]{FrV91}  gives the standard codification of this relation (also for reduced Hurwitz spaces). 

\cite[\S2]{Fr78} was a special case of it, where $G=D_\ell$, $\ell$ odd, and $\bfC=\bfC_{3^4}$ is 4 repetitions of the involution conjugacy class. In this case, it was describing the pair of fields $(\bQ(\bp),\bQ(\hat \bp)$), for $\bp\in \sH(G,\bfC)^{\abs}$ and $\hat \bp\in \sH(G,\bfC)^\inn$ over it. As in \S\ref{earlyOIT-MT}, this was the main case in describing prime degree ($\ell$) rational functions having the Schur cover  property \eqref{schurcover}. 

More precisely, consider the cover $f_\bp: W_\bp\to \prP^1_z$, from $\bp\in \sH(G,\bfC)^\abs$ with $\hat f_\bp$ its Galois closure. Here $W_\bp$ is isomorphic to $\prP^1_w$ over $\bQ(\bp)$ (because $\ell$ is odd). If  $\bQ(\bp)\not =\bQ(\hat \bp)$,  with extension of constants from \eql{serreles}{serrelesc} $$\text{$f_\bp$ is a Schur cover over $\bQ(\bp)$ (Thm.~\ref{schurcoverfiber}).}$$

If $\bQ(\bp) =\bQ(\hat \bp)$, then $\hat f_\bp$ is a (4 branch point)  \RIGP\  {\sl involution\/} realization of $D_\ell$ over $\bQ(\bp)$. \sdisplay{ \cite{DFr94}}, as in Prop.~\ref{dihcase}, completely classifies involution realizations of dihedral groups. This qualifies -- as stated in  Ques.~\ref{C-realvshyp} -- as the \lq\lq easiest\rq\rq\  case of untouched problems on the \RIGP; justifying why the \RIGP\ generalizes Mazur's results on modular curves. 

From each elliptic curve over $\bQ$ with non-integral $j$-invariant, the $\GL_2$ part of the \OIT\ \eql{serreles}{serrelesb}, gives explicit production of Schur cover rational functions \eqref{schurcover} of degree $\ell^2$, for infinitely many primes $\ell$. As with the \CM\ case, the distinction is measured by the difference between $\bQ(\bp)$ and $\bQ(\hat \bp)$ with $\hat\bp$ on the inner space over $\bp$ in the absolute space. 

When they are different, the degree $\ell^2$ rational function over $\bQ(\bp)$ decomposes, over $\bQ(\hat \bp)$, into two rational functions of degree $\ell$, with no such decomposition over $\bQ(\bp)$ (\cite{GMS03} and \cite[Prop.~6.6]{Fr05}). This is a phenomenon that cannot happen with {\sl polynomials\/} of degree prime to the characteristic, a fact exploited for the Schur cover property (as in  \cite{Fr70}). 

Conj.~\ref{OITgen} -- expressing our best guess for what to expect of an \OIT\ from a \MT, is the result of thinking how the relation  between these two different Schur covers compares with Serre's \OIT. Especially considering what is possible to prove at this time, both theoretically and explicitly. 

For example, \CM\ cases are famously explicit. Just as, in Schur covers given by  polynomials (cyclic and Chebychev), the nature of the exceptional set $\exc_{f,K}$ in \eqref{capexc} is a union of specific arithmetic progressions (in ray class groups). Example use: You can decide if {\sl compositions\/} of polynomials exceptionals are exceptional. 

\begin{defn} \label{CMdef} A $j'\in \bar\bQ$ is {\sl complex multiplication\/} if the elliptic curve with $j$ invariant  $j'$ has a rank 2  endomorphism ring. Then, that ring identifies with a fractional ideal in a complex quadratic extension $K/\bQ$.  \end{defn} 

The main point: $G_K$ will respect those endomorphisms. Thereby it will limit the decomposition group of a projective system of points on the spaces $\{X_0(\ell^{k\np1})\}_{k=0}^\infty$. Originally, as one of Hilbert's famous problems, Kronecker and Weber used this situation to describe the abelian extensions of complex quadratic extensions of $\bQ$. 

\cite[IV-20]{Se68} concludes the proof that for $j'$ non-integral (so not complex multiplication),  the Tate curve shows there is no decomposition of the degree $\ell^2$ rational function. It even gives the following result. 

\begin{thm}[\OIT\ strong form] \label{OITsf} Suppose $j'\in \bar \bQ$ is not a complex multiplication point. Then, not only is it a $\GL_2$ point for any prime $\ell$, but the decomposition group $G_{j'}$ is actually $\GL_2(\bZ_\ell)/\{\pm1\}$ (rather than an open subgroup of this) for almost all primes $\ell$. \end{thm} 

Falting's theorem \cite{Fa83} (as in \cite{Se97b})  replaces the unpublished result of Tate. The use of Faltings in both versions of the $r=4$ Main \MT\ conjecture for \MT s mean that both have inexplicit aspects, though the results are different on that (see  \sdisplay{\cite{CaTa09}} and \sdisplay{\cite{Fr06}}). 

So even today, being explicit on Thm.~\ref{OITsf} in the Schur covering property for the $\GL_2$ case still requires non-integral j-invariant \cite[\S6.2.1]{Fr05}. \cite[IV-21-22]{Se68} references Ogg's example \cite{O67} (or \cite[\S 6.2.2]{Fr05}), to give  $j'\in \bQ$ with the decomposition group $G_{j'}$ equal $\GL_2(\bZ_\ell)/\{\pm1\}$ for {\sl all\/} primes $\ell$. 

\Sdisplay{\cite{Ih86}} A similar title with \cite{Fr78} may be  Shimura's influence. Both played on interpreting braid group actions, a monodromy action that captures {\sl anabelian\/} data (from curves, rather than from abelian varieties). They also both thought on complex multiplication.  

Ihara's paper has a moduli interpretation for how to generate the field extension using Jacobi sums derived from Fermat curves , giving the second commutator quotient of $G_\bQ$; versus the first commutator quotient given (Kronecker-Weber) cyclotomic values.  Abstract result from it: An interpretation of Grothendieck-Teichm\"uller on towers of Hurwitz spaces \cite{IM95} (albeit,  Hurwitz spaces without fine moduli properties).

\Sdisplay{\cite{Se90a}}  At the top of \S\ref{overview},  example \eql{mainexs}{mainexsb} started my interaction over this approach to the \OIT.  That expanded quickly into using the Universal Frattini cover to construct the original \MT s.  

For simplicity assume a finite group $G$ is $\ell$-perfect \eql{MTconds}{MTcondsa}. Then, the lift invariant for a prime $\ell$ described below comes from considering {\sl central\/} Frattini extensions with $\ell$ group kernels. Using this gave precise statements on components of Hurwitz spaces. 

The sequence \eqref{FrVSns} is an easy-to-state result based on this tool,  giving a presentation of $G_\bQ$. It also produced a simply-stated conjecture. 
Assume for $K\subset \bar\bQ$ that $G_K$ is a projective profinite group.\footnote{Shafarevich's conjecture is the special case that $K$ is $\bQ$ with all roots of 1 adjoined.}

\begin{guess}[Generalizing Sharafavich's Conjecture] \label{FrVConj} Then, $K$ is Hilbertian if and only if $G_K$ is profree.\footnote{That $G_K$ profree implies it is Hilbertian is a consequence of a version of Chebotarev's field crossing argument. The \cite{FrV92} result starts with the assumption that $K$ is {\bf P}(seudo){\bf A}(lgebraically){\bf C}(losed).}  \end{guess} 

That sounds good. But, it didn't lead to more general use of Nielsen classes. So, \cite{Fr20} revamps how \eqref{FrVSns} arises}, recasting it as a classically motivated connector between the \RIGP\ and the \OIT\  (as in \S\ref{RIGPFaltings}). Historical support is below and in the next two discussions.

Now consider a cover $\phi_\bg: \prP^1_w\to \prP^1_z$ representing $\bg$ as given by the conditions \eqref{bcycs}. Then, consider constructing $Z \to \hat W \to \prP^1$ with $\hat W$ the Galois closure of $\phi_\bg$, and $Z\to \prP^1_z$ Galois with group $\Spin_n$.  Result: There is  an unramified $Z\to \hat W$ if and only if $\bg$ is in the image of $\ni(\Spin_n,\bfC_{3^r})$. 

\cite[Thm.~B]{Fr10} says, for $r\ge n$, the two braid orbits on $\ni(A_n,\bfC_{3^r})$ are distinguished by their lift invariants. See \sdisplay{\cite{Fr02b}}  and  \sdisplay{\cite{We05}}. 

This example, including using the same naming of the same order lift class, $\C_3$, of elements of order 3 in both $A_n$ and $\Spin_n$, has many of the ingredients that inspired the use of the Universal Frattini cover $\tilde G$. 

The conjugacy class $\C_3$ has the same cardinality in $\Spin_n$ as it has in $A_n$. If we included, even once, a product of two disjoint 2-cycles as an element of the Nielsen class $\ni(A_n,\bfC)$, this would kill the lift invariant. These examples are writ large in the Main Theorem of \cite{FrV92}, which \cite{Fr20} has  revamped and expanded.

Central Frattini extensions affected three kinds of results.  
\begin{edesc} \label{centfratuse} \item Describing components on a Hurwitz space $\sH(G,\bfC)$ assuming, if a class appears in $\bfC$ it does so with high multiplicity.  
\item Describing, as in Thm.~\ref{obstabMT}, the obstruction to there being a nonempty \MT\ supported by the Nielsen class $\ni(G,\bfC)$. 
\item As in  Ex.~\ref{exA43-2cont}, a tool for classifying cusps. \end{edesc}

\Sdisplay{\cite{Se90b}} A combination of this paper with \cite[\S 6]{Fr10} makes use of the lift invariant for any Nielsen class of odd-branched Riemann surface cover of the sphere in say, the Nielsen class $\ni(A_n,\bfC)$. 

It is a formula for the parity of a uniquely defined half-canonical class on any cover $\phi: W \to \prP^1_z$ in the Nielsen class that depends only on the spin lift invariant generalizing $s_{\Spin_n/A_n}$ defined above.  
From this \cite[\S6.2]{Fr10} produces {\sl Hurwitz-Torelli\/} automorphic functions on certain Hurwitz space  components through the production of even $\theta$-nulls. 

\Sdisplay{\cite{Se92} and \cite{Fr94}}  
Serre didn't use the braid monodromy (rigidity) method.  Fried makes the connection to braid rigidity through Serre's own exercises. The difference shows almost immediately in considering the realizations of Chevalley groups of rank $>1$. Serre records just three examples of Chevalley groups of rank $>1$ having known regular realizations at the publication date of his book. 

\cite{FrV91} and \cite{FrV92} constructed, for each finite $G$, a centerless covering group $G^*$ with infinitely many  collections of rational conjugacy classes $\bfC$ of $G^*$, also having a (faithful) representation $T^*: G^*\to S_{n^*}$. These had the following additional properties.

\begin{edesc} \label{Gconds}  \item \label{Gcondsa}  \!$G^*(T^*,1)$ (stablilizing 1) is \!self-normalizing.\footnote{As in \eql{fmodabsinn}{fmodabsinna}.}  \item \label{Gcondsb}  $N_{S_{n^*}}(G^*,\bfC)/G^*$\footnote{\S\ref{equivalences} for $N_{S_n}(G,\bfC)$.} consists of all outer automorphisms of $G^*$. 
\item \label{Gcondsc}  $\sH(G^*,\bfC)^\inn$, the resulting inner Hurwitz spaces  are irreducible and have definition field $\bQ$.  
\end{edesc}  

This allowed using the Hurwitz spaces as part of a {\sl field-crossing\/} argument over any \begin{center} {\bf P}(seudo) {\bf A}(lgebraically) {\bf C}(losed) field $F\subset \bar \bQ$:\end{center} any absolutely irreducible variety over $F$ has a Zariski dense set of $F$ points.  The result was that if $F$ was also Hilbertian, then $G_F$ is profree (see Conj.~\ref{FrVConj}).  A particular corollary was the presentation of $G_\bQ$ in \eqref{FrVSns}. 

Condition \eql{Gconds}{Gcondsa} is sufficient to say that any $K\subset \bar \bQ$ point on $\sH(G^*,\bfC)^\inn$ (satisfying \eql{Gconds}{Gcondsc}) corresponds to a $K$ regular realization of $G^*$, and therefore of $G$. This is because  $G^*$ will have no center, the condition that the inner Hurwitz space is then a fine moduli space. 

In myriad ways we can relax these conditions. Still, to use them effectively over say $\bQ$ requires finding $\bQ$ points on $\sH(G^*,\bfC)^\inn$. The usual method is to choose $\bfC$ so that $\sH(G^*,\bfC)^\inn$ is sufficiently close to the configuration space $U_r$, that $\bQ$ points are dense in it. If $r=4$, we may use Thm.~\ref{genuscomp}. Compute the genus of $\sH(G^*,\bfC)^{\inn,\rd}$, check if it has genus 0, and a degree 0 divisor of $\bQ$ $\implies$ $\infty$-ly many $\bQ$ points. 

Soon after \cite{FrV91}, V\"olklein and Thompson -- albeit powerful group theorists -- produced abundant high rank Chevalley groups based on this method. Locating specific high-dimensional uni-rational Hurwitz spaces was the key. Examples, and the elementary uses of Riemann's Existence Theorem, abound in \cite{Vo96}.

The Conway-Fried-Parker-Voelklein appendix of \cite{FrV91} was a non-explicit method for doing that. \cite{Fr10}  shows what  explicit can mean. 

\Sdisplay{\cite{DFr94}}  Ques.~\ref{C-realvshyp} gave the formulation of the Main \MT\ conjecture for dihedral groups. Equivalently, if it is false, the \BCL\ (Thm.~\ref{bclthm}) implies there is an  even integer $r$ ($\le r^*)$ and for each $k\ge 0$,  a dimension ${r\nm 2}\over 2$ {\sl hyperelliptic Jacobian\/} (over $\bQ$) with a $\bQ(e^{2\pi i/\ell^{k\np1}})$ torsion point, of order $\ell^{k\np1}$,  on whose group $G_\bQ$ acts as it does on $\lrang{e^{2\pi i/\ell^{k\np1}}}$. 

The {\sl Involution Realization Conjecture\/} says the last is impossible: There is a uniform bound as $n$ varies on $n$  torsion points on any hyperelliptic Jacobian of a fixed dimension, over a given number field. (The only proven case, $r=4$, is the Mazur-Merel result bounding torsion on elliptic curves.) If a subrepresentation of the cyclotomic character occurred on the $\ell$-Tate module of a hyperelliptic Jacobian (see \cite{Se68}), the Involution Realization Conjecture would be  false. This led to formulating Main \RIGP\ \MT\  conj.~\ref{mainconj}. 

Still missing: For any  prime $\ell > 2$, find  cyclotomic $\ell^{k\np1}$ torsion points on any hyperelliptic Jacobians for all (even infinitely many) values of $k$  (\S\ref{dihedtohyper}). 

\subsubsection{Constructions and Main Conjectures from 1995 to 2004} \label{95-04}   \S\ref{lfratquot} has the universal $\ell$-Frattini cover ${}_\ell\psi: {}_\ell\tilde G\to G$  (Def.~\ref{univFrattdef-ab}).  It is the minimal profinite cover of $G$ with its $\ell$-Sylow a pro-free pro-$\ell$ group \cite[Prop.~22.11.8]{FrJ86}$_2$. 

For  $P$ an $\ell$-group, ${}_{\text{\rm fr}}P\eqdef P^\ell[P,P]$ is its Frattini subgroup. As usual, ${}_{\text{\rm fr}}P$ is the closed subgroup of  $P$ containing generators from the $\ell$ powers and commutators of $P$. Then,  $P\to P/{}_{\text{\rm fr}}P$ is a Frattini cover.

Recover a cofinal family of finite quotients of $\fG \ell$ by taking $\ker_0=\ker({}_\ell \psi)$ denoting the sequence of {\sl characteristic kernels\/} of $\fG \ell$ as in \eqref{charlquots}: \begin{equation} \label{charlquots2} \ker_0> {}_{\text{\rm fr}} \ker_0\eqdef  \ker_1 \ge \dots \ge  {}_{\text{\rm fr}}\ker_{k{-}1} \eqdef \ker_k \dots,\end{equation}  $\fG \ell/\ker_k$ by $\tfG \ell k$, and the characteristic modules $\ker_k/\ker_{k'}={}_\ell M_{k,k'}$, etc.  

\Sdisplay{\cite{Fr95}}   Assume generating conjugacy classes, $\bfC$, of $G$.\footnote{The group generated by all entries of $\bfC$ is $G$.}  Then, with $N_{\bfC}$ the least common multiple of the orders of elements in $\bfC$:  
\begin{equation} \label{congcond} \begin{array}{c} \text{If $\ell \not| N_{\bfC}$,  Schur-Zassenhaus implies the classes $\bfC$ lift canonically}\\ \text{ to classes of elements of the same orders in each group $\tfG \ell k {}$.}  \end{array}\end{equation} 

This opens by recasting modular curves as Hurwitz spaces of sphere covers for dihedral groups, referring to \sdisplay{\cite[\S2]{Fr78}}. Then, applying \eqref{charlquots2}, that any group can be used to constructed modular curve-like towers. \S\ref{remainsects}  discussed the worth of this, emphasizing these: 

\begin{edesc} \item It works  with any $\ell$-perfect group $G$, replacing a dihedral group $D_\ell$, $\ell$ odd and  conjugacy classes $\bfC$ satisfying \eqref{congcond}; and 
\item This recasts the \RIGP\ and  the \OIT\ using points on substantive moduli spaces, allowing formulating a relation between them. \end{edesc} 

Without \eqref{congcond}, there is no unique assignment of lifts of classes in $\bfC$ to the characteristic $\ell$-Frattini cover groups. Reason: If $g\in G$ has order divisible by $\ell$, then the order of any lift $\tilde g\in \tfG \ell 1$ is $\ell\cdot \ord(g)$.

Given \eqref{congcond}, we canonically form towers of Nielsen classes, and associated Hurwitz spaces, from \eqref{charlquots} and their abelianizations: 
\begin{equation} \label{MTHur}  \text{$\{\sH(\tfG \ell k {},\bfC)^\inn\}_{k=0}^\infty$ and the abelianized version $\{\sH(\tfG \ell k {}\ab,\bfC)^\inn\}_{k=0}^\infty$.}\end{equation}  Originally we called these the \MT s. Now we prefer that a \MT\ -- Def.~\ref{MTdef} -- is a projective sequence of {\sl irreducible\/} components (from braid orbits on the Nielsen classes) of their respective levels.  

\cite[Part II]{Fr95} describes the characteristic modules for $G=A_5$ and primes $\ell=2,3,5$ dividing $|A_5|=60$. Thereby, for these cases, it describes the tower of Nielsen classes attached to the abelianized version of \eqref{MTHur}. 

We then required three immediate assurances. 
\begin{edesc} \label{mtneed} \item   \label{mtneeda} That we could decide when we are speaking of a non-empty \MT. 
\item  \label{mtneedb} That $K$ points on the $k$th tower level correspond to $K$ regular realizations in the Nielsen class  $\ni((\tfG \ell k {},\bfC)$ (or  $\ni((\tfG \ell k {}_\ab,\bfC)$).  
\item \label{mtneedc} That we know the definition field of $\sH(\tfG \ell k {},\bfC)^\inn \to U_r$ and the rest of the structure around $\sH(\tfG \ell k {},\bfC)^\inn$ as a moduli space. \end{edesc} 

\begin{proof}[Comments]  Response to \eql{mtneed}{mtneeda}:  The first necessary condition is that $G$ is $\ell$-perfect. Otherwise, no elements of $\bfC$ will generate a $\bZ/\ell$ image. 

A much tougher consideration, though, was what might prevent finding elements $\bg\in G^r\cap \bfC$ satisfying product-one (as in \S\ref{analgeom}). Thm.~\ref{obstabMT}, using the  {\sl lift invariant\/} resolves that \sdisplay{\cite{Fr02b}}  and  \sdisplay{\cite{We05}}. 

Response to \eql{mtneed}{mtneedb}:  Originally I formed \MT s to show that talking about rational points on them, vastly generalized talking about rational points on modular curve towers. Especially, that the \RIGP\ was a much tougher/significant problem than usually accepted. 

To assure $K$ points on the $k$th level correspond to regular realizations of the Frattini cover groups, we needed the fine moduli condition that each of the $\tfG \ell k\,$s has no center. The most concise is \cite[Prop.~3.21]{BFr02}: 
\begin{equation} \label{modcond} \text{If $G$ is centerless, and $\ell$-perfect, then so is each of the $\tfG \ell k\,$s.} \end{equation} 

Response to \eql{mtneed}{mtneedc}: The \BCL\ of \cite[\S5.1]{Fr77} perfectly describes moduli definition fields (Def.~\ref{moddeffield}) of Hurwitz spaces. The result is more complicated for absolute classes, but both results also apply to reduced classes.  We have given detailed modern treatments of this now in several places including \cite[App.~B.2]{Fr12}; also see \cite[Main Thm.]{FrV91}~and \cite{Vo96}. 

Using $N_{\bfC}$ (\S\ref{introhursp}), denote the field generated over $K$ by a primitive $N_{\bfC}$ root of 1 by $\Cyc_{K,\bfC}$ . Recall: $$G(\Cyc_{\bQ,\bfC}/\bQ)=(\bZ/N_{\bfC})^*, \text{ invertible integers } \!\!\!\mod N_{\bfC}.$$ Use the subgroup, $G(\Cyc_{K,\bfC}/K)$, fixed on $K\cap \Cyc_{\bQ,\bfC}$, to define:  
\begin{equation} \label{bcl} \begin{array}{rl} &\bQ_{G,\bfC} \!\eqdef\{m\in (\bZ/N_{\bfC})^* \!\!\mid \{g^m\mid g \in \bfC\} \eqdef\, \bfC^m\, =\,\bfC\}. \\ 
&\bQ_{G,\bfC,T} \!\eqdef \{m\in (\bZ/N_{\bfC})^* \!\!\mid \exists h\in N_{S_n}(G,\bfC) \text{ with } h \bfC^m h^{-1} =\bfC\}. \end{array} \end{equation} 

\begin{lem}[Branch Cycle] \label{branchCycle} As above, then $\bQ_{G,\bfC}$ (resp.~$\bQ_{G,\bfC,T}$) is contained in any definition field of any cover in the Nielsen class $\ni(G,\bfC)^\inn$ (resp.~$\ni(G,\bfC,T)^\abs\eqdef \ni(G,\bfC)^\abs$ if $T$ is understood) \cite[p.~62--64]{Fr77}.\end{lem} 

Still, for $K$ points to exist, there must be a component with moduli definition field $K$ (Def.~\ref{moddeffield}). That is a much harder problem. \end{proof} 

For good reasons there can be more than one component (as in  \sdisplay{\cite{Se90b}} and Prop.~\ref{A43-2} on \cite[Thm.~B]{Fr10})  in any particular case. 

The \OIT\ contends with that at all levels \sdisplay{\cite{FrH20}}. In lieu of Main \RIGP\  Conj.~\ref{mainconj} (or  Conj.~\ref{MTconj}) for a given \MT\ consider two cases.    
\begin{edesc} \label{moddefbd} \item \label{moddefbda} Some number field $K$ is a  moduli definition field of all tower levels.  
\item  \label{moddefbdb} The moduli definition degrees rise with the tower levels.  \end{edesc} 

Recall Main \RIGP\ conj. \ref{MTconj}
\begin{guess} \label{MTconj} At high levels there will be no $K$ points on a \MT. Also, high levels will be algebraic varieties of general type. \cite{Fr95}.  \end{guess} 

\sdisplay{\cite{BFr02}} inspected the case $(A_5,\bfC_{3^4}, \ell=2)$. This found that significant cusps and the lift invariant revealed as much detail on this \MT\ as one would expect from a modular curve tower, despite interesting differences.

\begin{defn} \label{HMrep} An element $\bg\in \ni(G,\bfC)$ is a {\sl Harbater-Mumford\/} (\HM) representative if it has the form $(g_1,g_1^{-1},\dots,g_s,g_s^{-1})$ (so $2s=r$). A braid orbit $O$ is said to be \HM, if the orbit contains an \HM\ rep. \end{defn} 

The name {\sl Harbater-Mumford} comes from \cite{Mu72} which used a completely degenerating curve on the boundary of a space of curves. In a sense, \cite{Ha84}, for covers, consists of a \lq\lq germ\rq\rq\ of such a construction.

\cite[Thm.~3.21]{Fr95} got the most attention by showing that if $\bfC$ is a rational union (Thm.~\ref{bclthm}), then $G_\bQ$ permutes the \HM\ components. Further, it gave an explicit criterion, applying to any  $G$, for producing classes $\bfC$  so that $\sH(G,\bfC)$ has just one $\HM$ component. 

Thereby, it found for $G$, small explicit values of $r=r_\bfC$  with an attached \MT\ and $\bQ$ as a moduli definition field for all \MT\ levels, situation \eql{moddefbd}{moddefbda} with the \MT\ levels all \HM\ components. 

Motivated by this,
both \cite{DEm05} and \cite{W98} developed a Hurwitz space Deligne-Mumford stable-compactification. This put a Harbater degeneration on its  boundary. This allows a standard proof -- contrasting with the group theoretic use of {\sl specialization sequences\/} in Fried's result -- to show Harbater-Mumford cusps lie on Harbater-Mumford components.

Both compactifications suit the \MT\ construction extending the projective systems. Further,  from Grothendieck's famous theorem, other than primes $p$ dividing $|G|$,  the \MT\ system has good reduction $\mod p$. This topic thus applies to the full Frattini cover $\tilde G\to G$ (rather than abelianized version) in achieving \RIGP\ results over fields other than number fields.  \sdisplay{\cite{D06}, \cite{DDes04}  and \cite{DEm05}}  continues this discussion. 

\Sdisplay{\cite{FrK97}}  Suppose $G$ is a group with many known regular realizations. For example:  $A_n$  semidirect product some finite abelian group. (Say, a quotient of $\bZ^{n-1}$ on which $A_n$ acts through its standard representation; a special case of Prop.~\ref{splitembedding} in \sdisplay{\cite{CaD08}}). Consider, for some prime $\ell$ for which $G$ is $\ell$-perfect, if there are regular realizations of the whole series of $\tfG \ell k$, $k\ge 0$, over some number field $K$.  

The basic question: Could all such realizations have a uniform bound, say $r^*$, on numbers of branch points -- with no hypothesis on the classes $\bfC$. 

\begin{thm} \label{genRIGPbdr} Such regular realizations are only possible by restricting to $\ell'$ classes (elements of $\bfC$ with orders prime to $\ell$). If they do occur, there must exist a \MT\ over $K$ for some one choice of $r\le r^*$ classes, $\bfC$, with a $K$ point at every level \cite[Thm.~4.4]{FrK97}. \end{thm} 

The Thm.~\ref{genRIGPbdr} conclusion is contrary to Main \MT\ \RIGP\ Conj. \ref{MTconj};  proven for $r^*\le 4$ (\sdisplay{\cite{Fr06} and \cite{CaTa09}}).

\Sdisplay{\cite{BFr02}} This is a book of tools that has informed all later papers on \MT s. The thread through the book is checking phenomena on \MT s lying over one (connected) reduced Hurwitz space: For the Nielsen class $\ni(A_5,\bfC_{3^4})$; four repetitions of the conjugacy class of 3-cycles) and the prime $\ell=2$. It computes everything of possible comparison with modular curves about level one (and level 0) where we take the full (not just the abelianized, Def. \ref{univFrattDef}) universal $2$-Frattini cover of $A_5$. 

It shows the Main \MT\  Conjecture \ref{MTconj} for it: No $K$ points at high levels (K any number field). The inner space at level 0 has one component of genus 0. Level one has two components, of genus 12 and genus 9. This concludes with a conceptual accounting of all cusps, and all real points on any \MT\ over the level 0 space (none over the genus 9 component). 

Many points about \MT\ levels arise here; particular cases of them appear attached to the components at level 1. For example, a version of the spin cover (extending the domain of use of \cite{Se90a}) obstructs anything beyond level 1 for the genus 9 component. 
Also, the argument using a prime, $\ell$,  of good reduction, as in \S\ref{RIGPFaltings} appears first here.

\cite[\S2.10.2]{BFr02} introduces the {\sl shift-incidence matrix\/} applying it to an early version of Ex.~\ref{mainex}. This linear device detects braid orbits and organizes cusps. Recall the braid generators in \eqref{Hrgens}. 

Choose any one of the twists $q_v$  (for $r$ = 4 it suffices to choose $q_2$ on reduced Nielsen classes) and call it $\gamma_\infty$. Reference rows and columns of the matrix  by  orbits of $\gamma_\infty$ on reduced Nielsen classes as $\row O t$. Reduced Nielsen classes are special in the case $r=4$, as in  Thm.~\ref{genuscomp}.
$$\text{The $(i,j)$ entry of the matrix is then $|(O_i)\sh\cap O_j|$:}$$
\begin{center} Apply $\sh$ to all entries of $O_i$, intersect it with $O_j$; the $(i,j)$-entry is  the cardinality of the result. \end{center}
As in \S~\ref{shincex},  $\sh$-incidence  graphically displays a \MT\ component at a given level using a characteristic cusp distinguishing  that component.  

Since $\sh^2=1$ on reduced classes when $r=4$, for that case the matrix is symmetric. Braid orbits correspond to matrix blocks. \cite[Table 2]{BFr02} displays the one block and the genus calculation for $(A_5 , \bfC_{3^4})$. Then, \cite[\S8.5, esp.~Table 4]{BFr02} does a similar calculation for the level 1 \MT s, $({}_2^1 A_5, \bfC_{3^4})$.  (This and $\ni(A_4,\bfC_{3^4})$ are still the only cases in print going up to level 1 for the full universal  $\ell$-frattini cover.)

\begin{edesc} \item There are two kinds of cusps, \HM\ and near-\HM, with near-\HM\ having a special action under the complex conjugation operator. \item The components at  level 1  have resp.~genuses 12 and 9 ($>1$): Faltings kicks in here as in \S\ref{toughestpt}. 
\item Apply \cite{DFr90}: the genus 12 (resp.~9)  component has one (resp.~no) component of real points. \end{edesc} 

\cite{Fr20b} has more examples illustrating the theme of having cusp types -- based on using refined aspects of $G$ -- separate components. That includes those in \sdisplay{\cite{LO08} and and  \cite{Fr09}}. Recall the key issue.  

\begin{center} When Hurwitz spaces have several components,  identify moduli definition fields geometrically to recognize the $G_\bQ$ action on the components. \end{center}

\Sdisplay{\cite{Fr02b}   and   \cite{We05}}   
\cite[Chap.~9]{Se92} added material from \cite{Me90} on regularly realizing the $\psi_{A_n,\Spin_n}: \Spin_n\to A_n$ (see Ex.~\ref{gen0}) cover. Stated in my language he was looking at the Nielsen class extension $$\Phi_{A_n,\Spin_n}: \ni(\Spin_n, \bfC_{3^{n-1}})^\inn \to \ni(A_n, \bfC_{3^{n-1}})^\inn.$$

\begin{thm} \label{Anlift} \cite[Main Thm.]{Fr10}  For all $n$, there is one braid orbit for $\ni(A_n, \bfC_{3^{n-1}})^\inn$. For $n$ odd, $\Phi_{A_n,\Spin_n} $ is one-one, and the abelianized \MT\ is nonempty. For $n$ even, $\ni(\Spin_n, \bfC_{3^{n-1}})^\inn$ is empty. \end{thm} 

Here is the meaning for $n$ even. Lift the entries of  $\bg \in  \ni(A_n, \bfC_{3^{n-1}})^\inn$ to same order entries in $\Spin_n$, to get $\hat \bg$. Then, the result does not satisfy product-one: $\hat g_1, \dots, \hat g_{n-1} = -1$ (the {\sl lift invariant\/} in this case).

I used this to test many properties of \MT s. Here it showed that there is a {\sl  nonempty\/} Modular Tower over $\ni(A_n, \bfC_{3^{n-1}})^\inn$ for $\ell=2$ if and only if $n$ is odd.  In particular the characteristic Frattini extensions define the tower levels, but central Frattini extensions control many of their delicate properties. If you change $\bfC_{3^{n-1}}$ to $\bfC_{3^{r}}$, $r\ge n$, there are precisely two components, with one obstructed by the lift invariant, the other not. 

\cite[Thm.~2.8]{Fr02b} gives a procedure to describe the $\ell$-Frattini module for any $\ell$-perfect $G$, and therefore of the sequence $\{\tfG \ell k {}_\ab\}_{k=0}^\infty$. \cite[Thm.~2.8]{Fr02b}  labels Schur multiplier types, especially those called {\sl antecedent}. Example: In \MT s where $G=A_n$, the antecedent to the level 0 spin cover affects \MT\ components and cusps at all levels $\ge  1$ (as in \cite{Fr06}). 

\cite[I.4.5]{Se97a} extends the classical notion of {\sl Poincar\'e duality\/} to any pro-$\ell$ group. It applied to the pro-$\ell$ completion of $\pi_1(X)$ with $X$ any compact Riemann surface. \cite{We05} uses the extended notion, intended for groups that have extensions by pro-$\ell$ groups of any finite group.

Main Result: The universal ${\ell}$-Frattini cover $\fG \ell$ (and $\fG \ell {}_\ab$) is an $\ell$-Poincar\'e duality group of dimension 2.  The result was Thm.~\ref{obstabMT}  \cite[Cor.~4.19]{Fr06}.  \S\ref{Ancomps}  and \S\ref{exA43-2cont} applied this $\Spin_n\to A_n$, for $\ni(A_5,\bfC_{3^4})$ and $\ni(A_4,\bfC_{+3^2-3^2}$ for different $\ell$-Frattini lattice quotients (Def.~\ref{latquotdef}). 

\subsubsection{Progress on the \MT\ conjectures and the \OIT} \label{05-to-now}  
As with modular curves, the actual \MT\ levels come alive by recognizing moduli properties attached to particular (sequences) of cusps. It often happens with \MT s that level 0 of the tower has no resemblance to modular curves, though a modular curve resemblance arises at higher levels.

Level 0 of alternating group towers illustrate: They have little resemblance to modular curves. Yet, often level 1 starts a subtree of cusps that contains the cusptree of modular curves. We can see this from the  classification of cusps discussed in \sdisplay{ \cite[\S3]{Fr06}}. 

We leave much discussion of generalizing Serre's \OIT\ to \cite{Fr20}. Still, we make one point here, based on what is in \sdisplay{\cite{Se68}} and \sdisplay{\cite{Fr78}}.

With $\dagger$ either inner or absolute equivalence,  it is the interplay of two Nielsen classes that gives a clear picture of the bifurcation between the two types of decomposition groups, \CM\ and $\GL_2$. Those Nielsen classes are \begin{equation} \label{twoNCs} \text{$\ni(D_\ell,\bfC_{2^4})^{\dagger,\rd}$ and $\ni((\bZ/\ell)^2\xs \bZ/2,\bfC_{2^4})^{\dagger,\rd}$.}\end{equation}   

In the example(s) of \cite{Fr20}, the same thing happens. Of course, the $j$ values don't have the same interpretation as  for Serre's modular curve case. Further, as in \sdisplay{\cite{FrH20}}, there are nontrivial lift invariants, and more complicated, yet still tractible, braid orbits. 

\Sdisplay{\cite{D06}, \cite{DDes04}  and \cite{DEm05}} RETURNM \cite{D06}  has expositions on \cite{DFr94}, \cite{FrK97}, \cite{DDes04} and \cite{DEm05} in one place.

\cite{DDes04}  assumes the Main \MT\ Conjecture \eqref{MainConj} is wrong:  for some finite group $G$  satisfying the usual conditions for $\ell$ and $\bfC$ and some number field $K$, the corresponding \MT\ has a $K$ point at every level. Using \cite{W98} compactifications of the \MT\ levels, for almost all primes $\bp$ of $K$, this would give a projective system of $\sO_{K,\bp}$ (integers of $K$ completed at $\bp$) points on cusps. The results here considered what \MT s (and some generalizations) would support such points for almost all $\bp$ using Harbater patching (from \cite{Ha84}) around the Harbater-Mumford cusps. 

 We continue the concluding paragraphs of \sdisplay{\cite{Fr95}} on \HM\ components and cusps.  \cite{DEm05}  continues the results of \cite{DDes04}. It ties together notions of \HM\  components and the cusps that correspond to them, connecting several threads in the theory. They construct, for {\sl every\/} projective system $\{G_k\}_{k=0}^\infty$ -- not just those coming from a universal Frattini cover as characteristic quotients --  a tower of corresponding Hurwitz spaces, geometrically irreducible and defined over $\bQ$ (using the \cite[Thm.~3.21]{Fr95} criterion). 
 
 These admit  projective systems of points over the Witt vectors with algebraically closed residue field of $\bZ_p$, avoiding only $p$ dividing some $|G_k|$.  If you compactify the tower levels, you get complete spaces, with cusps lying on their boundary. The \MT\ approach gives precise labels to these cusps using elementary finite group theory \sdisplay{\cite{Fr06} and \cite{CaTa09}}.  

\Sdisplay{\cite{Fr06} and \cite{CaTa09}}  We continue \sdisplay{\cite{Fr02b},  \cite{We05}  and  \cite{DEm05}}.  
\cite[\S 3.2.1]{Fr06} has three generic cusp types that reflect on Hurwitz space components and properties of \MT s containing such components:  $$\text{$\ell$-cusps, g(roup)-$\ell'$ and o(nly)-$\ell'$.}$$ Modular curve towers have only the first two types, with the g-$\ell'$  cusps the special kind called shifts of \HM. 

\cite[\S 3.2.1]{Fr06} develops these cusps when $r=4$ (as alluded to in \eql{cusplift}{cuspliftb}). For $\bg=(g_1,g_2,g_3,g_4)$ in the cusp orbit (\S\ref{compsr=4}):\footnote{The names stand, respectively,  for group-$\ell'$ and only-$\ell$. }

\begin{edesc} \label{cusptype} \item  The cusp is g-$\ell'$ if $$H_{1,4}=\lrang{g_1,g_4}\text{ and } H_{2,4}=\lrang{g_2,g_3}\text{ are $\ell'$ groups.}$$
\item It is o-$\ell'$, if $\ell\not | g_2g_3$ but the cusp is not g-$\ell'$. 
\item It is an  $\ell$ cusp otherwise. \end{edesc} These generalize to all $r$. For example: 
\begin{defn}[g-$\ell'$ type] \label{gl'} For $\bg$ in a braid orbit $O$ on $\ni(G,\bfC)$, we say $O$ is g-$\ell'$ if $\bg=(\row g r)$  has a partition  with elements 
$$\text{$P=[g_{u},g_{u\np 1},\dots,g_{u\np u'}]$ (subscripts taken $\mod r$)}$$ and  $H_P=\lrang{g_{u},g_{u\np 1},\dots, g_{u\np u'}}$ is an $\ell'$ group for each partition element.  \end{defn} 

The following is in \cite[Prin.~3.6, Frattini Princ. 2]{Fr06}. 

\begin{thm}  \label{gl'thm} There is a full \MT\ over the Hurwitz space component corresponding to $O$ if $O$ contains a g-$\ell'$ representative (no need to check central Frattini extensions as in Thm.~\ref{obstabMT}). \end{thm}  The approach to more precise results has been to consider a Harbater patching converse: Identify the type of a g-$\ell'$ cusp that supports a Witt-vector realization of $\fG \ell$. 

Typically we label a braid orbit  $O$ in  $\ni(G,\bfC)$ by the type of cusp it contains. In actual examples, say Thm.~\ref{level0MT}, even these generic names get refinements where we call particular o-$\ell'$ cusps {\sl double identity}. 

Both approaches to the Main \OIT\ Conj.~\ref{mainconj} give its truth for $r=4$, Thm.~\ref{truer=4}. Both use Falting's Theorem. First, \cite{CaTa09}, which is more general and far less explicit. As those authors say, the Main Conjecture \ref{MTconj}, which Tamagawa saw at my lectures  \cite{Fr02b}, motivated this.

Let $\chi: G_K \to  \bZ_p^*$ be a character, and $A[p^\infty](\chi)$ be the $p$-torsion on an abelian variety $A$ on which the action is through $\chi$-multiplication. Assume $\chi$ does not appear as a subrepresentation on any Tate module of any abelian variety (see \cite{Se68}, \cite{DFr94} and \cite{BFr02}). Then, for $A$ varying in a 1-dimensional family over a curve $S$ defined over $K$, there is a uniform bound on $|A_s[p^\infty](\chi)|$ for $s \in S(K)$. In particular, this gives the Main \RIGP\ conj.~\ref{mainconj} when $r=4$. 

It is the growth of the genus of these \MT\ levels, as in \cite[Prop. 5.15]{Fr06}, that assures Main \RIGP\ Conj.~\ref{mainconj} if a level contains at least 2  $\ell$-cusps. This gives the  Main \RIGP\ conj. \ref{mainconj} as equivalent to \eql{MainConj}{MainConja}. Also, we don't know at what level the surmised finite number of rational points will become {\sl no rational points}. If, on a \MT, there are no $\ell$-cusps at level 0, then we need  \cite[Princ.~4.23]{We05} to find a level  with an $\ell$-cusp lying  over an o-$\ell'$  cusp. \cite{Fr20b} does a deeper analysis relating these two distinct proofs -- noting \sh-incidence matrix aspects -- than we have space for here.

\Sdisplay{\cite{CaD08}}   Abelian groups and abelian varieties are obviously related going back to Abel. Here is a positive \RIGP\ result applied to $A\xs H$ with $A$ abelian, and $H$ any finite group acting on $A$. \begin{prop} \label{splitembedding} Regularly realizing $H$, acting on a finite abelian $A$, over an Hilbertian field $K$, extends to regularly  realizing $A\xs H$.  \end{prop}  This comes from the two steps in \eqref{lconjs}.  
\begin{edesc} \label{lconjs} \item \label{lconjsa} There is a regular realization of $A$ over $K$.  
\item \label{lconjsb} Combine regular realizations of $H$ and $A^{|H|}$ to give one of $A\wr H$, and so of $A\xs H$.\end{edesc} 
An explicit (resp.~abstract) proof of these pieces is in  \cite[Prop.~16.3.5]{FrJ86}$_2$ (resp.~\cite[\S4.2]{Se92}). For \eql{lconjs}{lconjsa} it suffices to consider the case $A_u=\bZ/\ell^u$, and from the \BCL\ -- for $\ell$ odd -- the minimal number of possible branch points for such a realizing cover is $\ell^{u\nm1}(\ell-1)$. Conclude Lem.~\ref{bpstoinfty}, using these regular realizations.  

\begin{lem} \label{bpstoinfty} Consider any sequence of regular realizations of $\{A_u\}_{u=1}^\infty$ over $\bQ$, with corresponding classes $\bfC_u$, and branch point numbers $r_u$. Then, $r_u\mapsto \infty$. More generally, suppose $H$ acts on $\tilde A'=\bZ_\ell^m$. 

Then, this gives a projective sequence $\{A_u'\xs H\}_{u=0}^\infty$ with regular realizations and limit $\tilde A'\xs H$. Yet, branch point number  goes to $\infty$ with $u$.  \end{lem}

\cite{CaD08} gives a new constraint for the \RIGP. Suppose a finite group $G$ has a regular realization over $\bQ$, and $P_\ell$ is an $\ell$-Sylow with $m=[G:P_\ell]$. Then the abelianization of $P_\ell$,  has order $\ell^u$ bounded by an expression involving
\begin{equation} \label{rigplsylow} \begin{array}{c} \text{$m$, the branch point number $r$  and the least} \\ \text{good reduction prime $\nu$ of the cover.}\end{array}\end{equation}

To wit: If $\ell^u$ is large compared to $r$ and $m$, branch points of the realizing cover must coalesce modulo some prime $\nu$; a $\nu$-adic measure of proximity to a cusp on the corresponding Hurwitz space.

Conj.~\ref{rmconj} is a stronger version of the Main \RIGP\ conj.~\ref{mainconj}. The parameters $m$ and $r$ are fixed as a function of the level $k$ for defining a \MT. So the Torsion Conjecture implies it, too. 

\begin{guess} \label{rmconj} Some expression in $r$ and $m$, independent of $\nu$, bounds $\ell^u$. This follows from the Torsion Conjecture on abelian varieties. \end{guess} 

Maybe we can turn this around. Instead of waiting for a proof of the Torsion Conjecture, use  $r=4$, as in \sdisplay{\cite{Fr06} and \cite{CaTa09}}  via \S\ref{prodladic}. That is, prove the Torsion Conjecture according to statement \eql{cusplift}{cuspliftc}: using cusps on a Hurwitz space to form a height on abelian varieties labeled by these cusps to test Conj.~\ref{rmconj}.

We remind of cases from \MT s. Choose from \S\ref{lfratlatqt} an $\ell$-Frattini lattice quotient $L^\star \to G^\star \to G=G_0$. For the maximal quotient $\fG \ell {}_\ab\to G$, the rank of $L^\star$  is necessarily $\ge 2$ unless $G_0$ is the $\ell$-supersolvable generalization of the dihedral case \cite{GriessFrat}.   

1-dimensional $\ell$-Frattini lattice quotients do include the superelliptic spaces of \cite{MaSh19}. These, though don't fit into using {\sl all\/} primes $\ell$ as a natural generalization to Serre's \OIT. Instead,  Thm.~\ref{level0MT} considers where the level 0 groups are $(\bZ/\ell)^2\xs \bZ/3={}_\ell G_0$ and the lattice rank is 2. $$\text{That is, take $L^\star \xs G_0$ instead of the $\ell$-Frattini cover.}$$
For $\ell\equiv 1 \mod 3$, superelliptic covers for $\bZ/3$ do appear, but not otherwise.  Lem.~\ref{bpstoinfty} now gives \RIGP\ realizations of all of the $(\bZ/\ell^{k\np1})^2\xs \bZ/3={}_\ell G_k$ over some fixed number field $K$. Yet, only by increasing the branch point number.  Again: The question of finding $\ell'$ \RIGP\ realizations for them and the hyperelliptic jacobian interpretation of this as an analog of Prop.~\ref{dihcase}. 

\Sdisplay{\cite{LO08}  and  \cite{Fr09}}  This gave an especially good place to see the sh-incidence matrix (discussion of \sdisplay{\cite{BFr02}}) in action on a variety of cusps with extra structure inherited from conjugacy classes in $A_n$. 

\cite{LO08} considered $\ni(G,\bfC)^\abs$ with these two conditions: The covers are genus 0; and  $\bfC$ is pure cycle:elements in the conjugacy classes have only one length $\ge$ 2 disjoint cycle.  They showed the Hurwitz space $\sH(G,\bfC)^\abs$ has one connected component. 

This overlaps with the 3-cycle result of \cite[Thm.~1.3]{Fr10}, the case of four 3-cycles in $A_5$ (\S\ref{l=2level0MT}).  \cite[\S5]{LO08} gives the impression that all these Hurwitz spaces are similar, without significant distinguishing properties. \cite[\S9]{Fr09}, however, dispels that. 

The stronger results come by considering the inner (rather than absolute) Hurwitz spaces. \cite[Prop.~5.15]{Fr09} uses the sh-incidence matrix to display cusps, elliptic fixed points, and  genuses of the inner Hurwitz spaces in two infinite lists of \cite{LO08} examples. In one there are two level 0 components (conjugate over a quadratic extension of $\bQ$). For the other just one. 

Further, using \eqref{cusptype}, the nature of the 2-cusps in the \MT s over them differ greatly. None have 2-cusps at level 0. For those with level 0 connected, the tree of cusps, starting at level 1, contains a subtree isomorphic to the cusp tree on a modular curve tower.\footnote{We call this a  {\sl spire}.} 

For the other list, there are 2-cusps, though not like those of modular curves. 
\cite{Fr20b} includes this case as one of its $\sh$-incidence matrix examples; another example where classical group theory -- here from  the Clifford Algebra -- shows in the naming of the cusps in the $\sh$-incidence display. 
 
\Sdisplay{\cite{FrH20}} The culminating topic of \cite{Fr20}  is the system of \MT s based on the Nielsen class $\ni_{\ell,3}\eqdef \ni((\bZ/\ell)^2\xs \bZ/3, \bfC_{+3^2-3^2})$ (as in Ex.~\ref{gen0}).  Notice our choice of conjugacy classes (at first in $\bZ/3$, but extended to $\ni_{\ell,3}$) is a rational union, as given by Thm.~\ref{bclthm}. Using the \BCL\ (Thm.~\ref{bcl}) the moduli definition field of the Hurwitz spaces is $\bQ$. 

This is parallel to Serre's dihedral group example. We have a series of groups $\{G_\ell\}_{\ell \text{ prime}}$ like the series of dihedral groups $\{D_\ell\}_{\ell \text{ prime}}$. 
This is an example of  \S \ref{normlgp} with the $\ell$-Sylow of $G_\ell$  normal. For  completeness: 
$$\begin{array}{c} \text{For a given $\ell$, the group sequences appearing in the canonical } \\  \text{$\ell$-Frattini lattice quotient here are $\{G_k\eqdef (\bZ/\ell^{k\np1})^2\xs \bZ/3\}_{k=0}^\infty$. }\end{array} $$ 

If we follow the general statement for a \MT\ that a prime $\ell$ is considered only if $G$ is $\ell$-perfect, then our condition would be $(\ell,3)=1$, and for each $\ni_{\ell,3}$, only $\ell$ would be involved in the \MT. 
\begin{edesc} \label{3vs2}  \item \label{3vs2a} Serre includes $\ell$;  how will we include $\ell=3$?   
\item \label{3vs2b} Then, for each $\ell$, why  leave out the $\ell= 3$ in applying to the particular case of $\ni_{\ell,3}$? \end{edesc} 

The trick is this: Since in Serre's case (resp.~our case), the $\bZ/2$ (resp.~$\bZ/3$) is a splitting coming from a semi-direct product $\bZ\xs \bZ/2$ (resp.~$(\bZ/2)^2\xs \bZ/3$), we are free to ignore the copy of $\bZ/2$ (resp.~$\bZ/3$) even if $\ell=2$ (resp.~$3$). For the same reason we ignore that prime if $\ell$ is not 2 (resp.~3).\footnote{But we don't include $\ell=3$ here.}  

Below we quote only the level 0 braid orbit description, but for all $\ell$. There are several different \MT s in our case. This doesn't occur in dihedral group cases. For one there is no central $\ell$-Frattini cover of $D_\ell$ (for $\ell$ odd). So, no lift invariant occurs in the Nielsen class that starts Serre's \OIT. 

There is, though as in Thm.~\ref{Anlift} in the alternating group case, though we didn't set that up with a lattice action as in this case. 

Consider the matrix $$ \label{InvHell} 
M(x,y,z)\eqdef \begin{pmatrix} 1 &x &z \\ 0& 1 &y \\ 0& 0& 1\end{pmatrix} , \text{ with inverse } 
 \begin{pmatrix} 1 &-x &xy\nm z \\ 0& 1 &-y \\ 0& 0& 1\end{pmatrix}. $$
With $R$ a commutative, the $3\times 3$  {\sl  Heisenberg group\/} with entries in $R$ is 
$$\bH_{R}=\{M(x,y,z)\}_{x,y,z\in
R}.$$ 
\cite[\S4.2.1]{FrH20}  shows the $\bZ/3$ action extends to the {\sl small Heisenberg group\/}
providing this Nielsen class with a non-trivial lift invariant: comments on \eql{A44}{A44c}. Even at level 0, that separates the braid orbits with theri respective  lift invariants in  0 from versus in $(\bZ/\ell)^*$.  

The lift invariant values grow with the tower level, because the Heisenberg group kernel grows. This is a case of Ex.~\ref{lcentext}. 

\cite[Prop.~4.18]{FrH20} gives a formula for the lift invariant in this case when $r=4$, the first such formula going beyond the Nielsen classes for $A_n$ and $\bfC$ odd order classes (as in the discussion \sdisplay{\cite{Se90a}}). 

Thm.~\ref{level0MT} is part of  \cite[Thm.~5.2]{FrH20}: level 0 components of the Hurwitz space.  For $\ell>3$ prime,  define $$K_\ell= \left\{\begin{array}{ll}{{\ell\np1}\over 6}, \text{ for }\ell \equiv -1\mod 3 \\   {{\ell\nm1}\over 6}, \text{ for }\ell \equiv +1\mod 3. \end{array} \right.$$
Prop.~\ref{A43-2} did $\ell=2$, with subtly different group theory. 

{\sl Double identity Nielsen class representatives\/}  have the form $(g,g,g_3,g_4)$, $$\text{o-$\ell'$ components as in  \eqref{cusptype}; nothing like \HM\ reps (Def.~\ref{HMrep}).}$$ 

\begin{thm}[Level 0 Main Result] \label{level0MT} For $ \ell> 3$ prime and $k = 0$  there are $K_\ell$  braid orbits with trivial (0) lift invariant. All are \HM\ braid orbits. 

The remaining orbits are distinguished by having nontrivial lift invariant (in $(\bZ/\ell)^*$; each value is achieved). Each contains a {\sl double identity\/}  cusp. \end{thm} 

Hurwitz space components for braid orbits with nontrivial lift invariant are conjugate over $\bQ(e^{2\pi i/\ell})$. As with Serre's \OIT, there is another Nielsen class, $\ni((\bZ/\ell)^4\xs\bZ/3,\bfC_{+3^2-3^2})^{\inn,\rd}$, with limit group $G^\lm=(\bZ/\ell)^4\xs\bZ/3$.  A new phenomena occurs in Thm.~\ref{level0MT}:  $$\text{the presence of more than one \HM\ braid orbit.}$$   
As with the case $\ell=2$, we used the $\sh$-incidence matrix to see this result. Putting the higher levels of the \MT\ in a graphical display is still a work in progress, but we expect to have it in \cite{Fr20b}.\footnote{That will also include how displays of the higher tower levels that already appear for $(A_5,\bfC_{3^4},\ell=2$) in \cite{BFr02}.} 

As in Def.~\ref{namedclasses}, use $\CM$ and $\GL_2$ for the respective Nielsen classes $\ni(D_\ell,\bfC_{2^4})$, $\ni((\bZ/\ell)^2,\bfC_{2^4})$ and the decorations that appear with them. 

In both \S\ref{earlyOIT-MT} and in \S\ref{pre95}, under \sdisplay{\cite{Se68}} we call attention to the eventually Frattini Def.~\ref{evenfratt} and its abstraction for the \CM\ and $\GL_2$ cases in \eqref{serreles}. Then, in  \sdisplay{\cite{Fr78}}  we take advantage of the relation between the two distinct Nielsen classes as follows. 
\begin{edesc} \label{CMGL2} \item \label{CMGL2a} The $\GL_2$ covers of the $j$-line are related to the \CM\ covers over the $j$-line, by the former being the Galois closure of the latter. 
\item \label{CMGL2b} For \CM\ or $\GL_2$, the extension of constants indicates that we have the appropriate description of the fiber according to the \OIT.  
\item \label{CMGL2c} In the $\GL_2$ case, in referring to \eql{CMGL2}{CMGL2b}, braids give the geometric elements, $\text{\rm SL}_2(\bZ/\ell^{k\np1})/\lrang{\pm1}$ in this case, of $N_{S_n}(G)/G$. That gives all the geometric monodromy of the $j$-line covers. \end{edesc} 

Comment on \eql{CMGL2}{CMGL2a}: This was what our discussion of \sdisplay{\cite{Fr78}} was about. 
From \eql{CMGL2}{CMGL2c} we \lq\lq see\rq\rq\  the $\GL_2$ geometric monodromy. The most well-known proof of \cite{Se68} is -- our language:  $
\{\text{\rm SL}_2(\bZ/\ell^{k\np1}/\lrang{\pm1}\}_{k=0}^{\infty}$ is  $\ell$-Frattini (resp.~eventually $\ell$-Frattini) for $\ell > 3$ (resp.~for all $\ell$). 

Such moduli spaces, affording refined ability to interpret cusps, enable objects of the style of the Tate curves (as in \sdisplay{\cite{Fr78}}{) around, say, the {\sl Harbater-Mumford\/} type cusps. This is compatible with those cusps in the discussions of \sdisplay{\cite[Thm.~3.21]{Fr95} ,  \cite{W98}  and  \cite{DEm05}}.  

As in the discussion of \sdisplay{\cite{Fr78}} we can expect inexplict versions of Faltings \cite{Fa83} if we can find any version of them at this time. Also, we must ask how far into one of the eventually $\ell$-Frattini strands we must go to assure the fiber over  $j'\in \bar \bQ$ has revealed itself? 

\begin{rem}[Rem.~\eqref{braidorbits1} Cont.] \label{braidorbits2} \cite{FrH20} shows regular behavior on the Thm.~\ref{level0MT} \MT\ sequences above level 0. There are different types of such towers, mostly depending on the powers of $\ell$ dividing lift invariants of the corresponding component sequences.  Even at level 0,  Thm.~\ref{level0MT}  shows an increasing numbers of components, many conjugate over $\bQ$, as $\ell$ changes. 

\begin{edesc} \label{unfinishedC3^4} \item  \cite{FrH20}
 will completely display all \MT\  levels. 
\item  \label{unfinishedC3^4b}  What are the $G_\bQ$ orbits on the \HM\ components with their moduli structures (carrying families of covers) of $\prP^1_z$? 
\end{edesc}
Having several \HM\ orbits  in  \eql{unfinishedC3^4}{unfinishedC3^4b} leaves a lift invariant puzzle. \cite[Ex.~9.3]{BFr02} has carefully diagnosed cases of this where $G=A_4$ and $A_5$ at level $k=1$ of their \MT s for $\ell=2$. Trivial lift invariants are often valuable. They give a distinguished component, as in \cite{FrV92}. Yet, when there is more than one, we don't yet know how  to geometrically distinguish them. 
\end{rem} 

\begin{rem}[Nielsen limit group] \label{nieslimrem} For Thm.~\ref{level0MT},  the analog of \S\ref{limgptie}, especially Ex.~\ref{ncmodcurves-cont} for $G^\lm$ works here, too. You construct the element in the limit group Nielsen class from those in the original Nielsen class. With that comes a natural braid action, which embeds $H_4$ into the symplectic group acting on $H^1(X,\bZ_\ell)$ of some hyperelliptic curve chosen as a base point on the Hurwitz space. 
\end{rem} 

\providecommand{\bysame}{\leavevmode\hbox to3em{\hrulefill}\thinspace}
\providecommand{\MR}{\relax\ifhmode\unskip\space\fi MR}
\providecommand{\MRhref}[2]{%
\href{http://www.ams.org/mathscinet-getitem?mr=#1}{#2}}
\providecommand{\href}[2]{#2}


\end{document}